\documentclass[11pt]{amsart}
\usepackage{amsmath, amsfonts, amsthm, amssymb}
\usepackage[all,cmtip,line]{xy}
\usepackage{array}
\usepackage{enumerate}
\usepackage{srcltx} 

\newtheorem{thm}{Theorem}[section]
\newtheorem{lem}{Lemma}[section]
\newtheorem{prop}[lem]{Proposition}
\newtheorem{cor}[lem]{Corollary}
\newtheorem{defn}[lem]{Definition}
\newtheorem{rem}[lem]{Remark}

\numberwithin{equation}{section}

\newcommand{\supp}{ \mbox{supp}}

\newcommand \eps{\varepsilon}

\newlength{\originalbase}
\newcommand{\spacing}[1]{\setlength{\baselineskip}{#1\originalbase}}

\begin{document}               

\newcommand{\avint}{{- \hspace{-3.5mm} \int}}

\spacing{1}

\title[rigorous treatment of moist convection]{A rigorous treatment of moist convection in a single column}

\author{Bin Cheng}
\address{Department of Mathematics\\
         University of Surrey\\GU2 7XH
        UK}
\email{B.Cheng@surrey.ac.uk}

\author{Jingrui Cheng}
\address{Department of Mathematics\\
         University of Wisconsin\\
         Madison WI 53706\\
USA}
\email{jrcheng@math.wisc.edu}

\author{Michael Cullen}
\address{
Met Office\\ Fitzroy Road\\ Exeter Devon EX1 3PB\\ UK}
\email{mike.cullen@metoffice.gov.uk}

\author{John Norbury}
\address{Mathematical Institute\\
University of Oxford\\
Oxford\\OX 2 6GG\\
UK}
\email{John.Norbury@lincoln.ox.ac.uk}

\author{Matthew Turner}
\address{Department of Mathematics\\
         University of Surrey\\GU2 7XH
        UK}
\email{M.Turner@surrey.ac.uk}

\date{\today}

\begin{abstract}

We study a single column model of moist convection in the atmosphere. We state the conditions for it to represent a stable steady state. We then evolve the column by subjecting it to an upward displacement which can release instability, leading to a time dependent sequence of stable steady states. We propose a definition of measure valued solution to describe the time dependence and prove its existence.
\end{abstract}
\maketitle
$\mathbf{Key\,words}$: Rearrangement, measure-valued solutions, Lagrangian equations.

$\mathbf{AMS\,\,subject\,\,Classifications}$: 35B45, 35D30, 35Q86.
\section{Introduction}

This paper studies a simple mathematical model of moist convection in the atmosphere set out in Bokhove $et\, a. $ \cite{report}.  Moist convection is responsible for much of the severe weather in the extratropics, and is the main driver of the tropical circulation which is a fundamental part of the climate system. 
While convective storms have a very complicated structure, in which the physics of water in various phases is critical, the essential process can be captured by a one-dimensional model which only takes into account the saturation of air parcels with the associated release of latent heat. Such a model is used routinely by practising weather forecasters in interpreting the likely weather that will result from a given vertical profile of temperature and moisture, see \cite{Met}, chapter 4. It also forms a key component of many theoretical studies of moist convection in the atmosphere; for instance Holt \cite{Holt}, Lock and Norbury \cite{Lock} and Shutts \cite{Shutts1}.  

The model expresses conservation of heat and moisture, together with the change of phase of moisture from vapour to and from liquid and the associated release or absorption of latent heat. This takes place at a moisture concentration which depends on temperature and pressure, and introduces a strong nonlinearity into the problem. Moist convection results from an instability of the vertical profile, which can be triggered by the upward bulk motion of the vertical profile generated by extratropical weather systems. In our model we represent the effect of this by making the saturation moisture content a monotonically decreasing function of time. This allows the model to be solved in a fixed vertical domain, which simplifies the presentation.

The conservation properties are expressed in Lagrangian form, so that a discrete version of the problem can be solved by rearranging fluid parcels as in Bokhove $et\,a.$ \cite{report}, Holt \cite{Holt} and Lock and Norbury \cite{Lock}.
 These conservation properties have been shown to be quite accurate even in more complicated models, e.g. by Shutts and Gray \cite{Shutts2}. 
The rearrangement procedure is designed to reflect the underlying physics of the problem. 

The first attempt to rigorously study this model was made by Dorian Goldman in his Master's thesis \cite{Goldman}, where he considered a particular choice of moisture content and initial data and proved the existence of weak solutions in Lagrangian variables. 
However, there seems be certain gaps in the proofs and the solution was not completely characterized. Besides, his proof does not generalize to more general choice of moisture content and initial data, which can be physically interesting.

The aim of this paper is to show that the discrete problem converges to a limit solution as the number of parcels is increased and to interpret the resulting solution as a weak Lagrangian solution of the governing equations. 
We take a probabilistic approach in this paper, which is completely different from \cite{Goldman} and allows us to deal with more general choice of moisture content function and initial data, which is physically meaningful.

The plan of the paper is as follows. In section 2 we present the problem to be solved and write it as a set of Lagrangian evolution equations. We note that we can only expect a probabilistic solution for general choices of initial data. In section 3, we describe the procedure to construct approximate(discrete) solutions given some deterministic discrete initial data, and show they satisfy the physical constraints. In section 4, we establish necessary estimates about these discrete solutions. In section 5, we come up with the notion of measure valued solutions, and show this coincides with a natural definition of the solution when the initial data and evolution is deterministic. In section 6, we take the limit of the discrete solutions as time/space step size tends to zero and obtain the existence of measure valued solutions.

\section{Definition of the problem}

The problem to be studied, Bokhove {\em et al.} (2016), is
\begin{align}\label{1.1}
&D_t(\theta+q)=0\textrm{ in $(z,t)\in[0,1]\times(0,T)$}.\\
\label{1.2}&
D_t\theta = \left\{ \begin{array}{ll}
0 & \textrm{if $q<Q^{sat}(\theta,z,t)$}\\
\,[D_t(Q^{sat}(\theta,z,t))]^- & \textrm{if $q=Q^{sat}(\theta,z,t)$}\\
\end{array} \right.\textrm{ in $(z,t)\in[0,1]\times(0,T)$.}\\
\label{1.3}&
\frac{\partial u}{\partial z}=0\textrm{ in $(z,t)\in[0,1]\times(0,T)$.}\\
\label{1.4}&
q(z,t)\leq Q^{sat}(\theta(z,t),z,t).
\end{align}
As usual, we denote
$D_t=\partial_t+u\cdot\partial_z
$. 
The equation (\ref{1.3}) above should be interpreted as the divergence free condition with respect to the space variable $z$, namely its flow is measure preserving. The unknown functions are the potential temperature $\theta(z,t)$ and the moisture content $q(z,t)$. Equation \ref{1.4} expresses the physical constraint that the moisture content is limited by the known saturation value $Q^{sat}$ which is time dependent. 
The interesting case, which we study, is where $Q^{sat}$ is monotonically decreasing in time. However, this not needed in the subsequent argument. In the above, $Q^{sat}:\mathbb{R}^3\rightarrow\mathbb{R}$ is a smooth function in its variables, and the following strict monotonicity conditions hold:
\begin{equation}\label{mq}
\partial_{\theta}Q^{sat}>0,\,\,\partial_zQ^{sat}<0,\textrm{ for any $(\theta,z,t)\in\mathbb{R}^3$.}
\end{equation}
Physical solutions to (\ref{1.1})-(\ref{1.4}) should also satisfy the following constraint:
\begin{equation}\label{1.6}
z\longmapsto\theta(z,t)\textrm{ is monotone increasing in $z$ for any $t\in(0,T)$.}
\end{equation}
The reason for imposing such a constraint is that physical solutions should minimize the energy functional  
$$
E(\bar{\theta})=-\int_{[0,1]}z\bar{\theta}(z)dz,
$$
where $\bar{\theta}(z)$ is a bounded Borel function on $[0,1]$, 
among all the possible rearrangements of the particles.
 It is easy to see that a function $\theta(z)$ achieves the minimum of $E$ among all the functions $\bar{\theta}$ with the same distribution 
as $\theta$ iff $\theta(z)$ is monotone increasing. 
Rearranging the parcels is a measure preserving map which does not change the distribution.

It is not hard to see that, in general, the solution does not have good regularity. 
Indeed, if everything is smooth, then let $F:[0,T)\times[0,1]\rightarrow[0,1]$ be the flow map. From (\ref{1.3}),
 we get for each fixed $t\in[0,T)$, $F_t(\cdot)$ preserves $\mathcal{L}^1_{[0,1]}$. If $F$ were continuous, then we can obtain that $F(z)=z$ or $F(z)=1-z$. But $F_0=id$, hence by continuity in $t$, we would be getting $F(t,z)=z$ for all $t$. 
This is not compatible with (\ref{1.4}) and (\ref{1.6}) except in trivial cases. Hence $F$ cannot be continuous. Therefore, the velocity $u$ is defined only as a measure.
 It is not clear how to define weak solutions to (\ref{1.1})-(\ref{1.4}) in a standard way since the set $\{q=Q^{sat}(\theta,z,t)\}$
 is only a general Borel set and may not have nice regularity.

We next define the solution in Lagrangian variables. Let $F_t(z)$ be the flow map, we then get a reformulation of (\ref{1.1})-(\ref{1.4}):
\begin{align}\label{1.7}
&\partial_t(\hat{\theta}+\hat{q})=0\textrm{ in $(z,t)\in[0,1]\times(0,T)$}.\\
\label{1.8}&
\partial_t\hat{\theta} = \left\{ \begin{array}{ll}
0 & \textrm{if $\hat{q}<Q^{sat}(\hat{\theta},F_t(z),t)$}\\
\,[\partial_t(Q^{sat}(\hat{\theta},F_t(z),t))]^- & \textrm{if $\hat{q}=Q^{sat}(\hat{\theta},F_t(z),t)$}\\
\end{array} \right.\textrm{ in $(z,t)\in[0,1]\times(0,T)$.}\\
\label{1.9}&
F_t\#\mathcal{L}^1_{[0,1]}
=\mathcal{L}^1_{[0,1]}\textrm{, for any $t\in[0,T)$.}\\
\label{1.10}&
\hat{q}(z,t)\leq Q^{sat}(\hat{\theta}(z,t),F_t(z),t).
\end{align}
In the above, $\hat{q}$ and $\hat{\theta}$ denote the corresponding variables in Lagrangian coordinates, namely $\hat{q}(t,z)=q(t,F_t(z))$, and $\hat{\theta}(t,z)=\theta(t,F_t(z))$. Here we remark that the equations can be interpreted in a natural way. Indeed, (\ref{1.7}) means $\hat{\theta}+\hat{q}$ is conserved along flow lines. As for equation (\ref{1.8}), notice that the right hand side of (\ref{1.8}) is nonnegative, hence $\partial_t\hat{\theta}$ will be a nonnegative measure. If we can show $t\longmapsto F_t(z)$ has bounded variation, then $\partial_t(Q^{sat}(\hat{\theta},F_t(z),t))$ is a well-defined finite signed measure and its negative part can be defined. Therefore, (\ref{1.8}) can be naturally defined as an equality of measures.

It will be convenient to consider the function $\Theta(w,z,t)$ as the solution $\theta$ to the equation:
\begin{equation}\label{dT}
\theta+Q^{sat}(\theta,z,t)=w
\end{equation}
This function is well defined thanks to the assumed strict monotonicity of $Q^{sat}$ about $\theta$. 
Also we know that $\Theta$ is smooth and satisfies the strict monotonicity
\begin{equation}\label{mT}
\partial_w\Theta>0,\,\,\partial_z\Theta>0.
\end{equation}
This is clear from (\ref{mq}). First we make a simple observation whose proof is elementary.
\begin{lem}\label{1.1l}
Define $\theta^M(t, z)=\theta(t,z)+q(t,z)$. Then $q(t,z)\leq Q^{sat}(\theta(t,z),z,t)$ is equivalent to $\theta(t,z)\geq\Theta(\theta^M(t,z), z,t)$. Equivalence holds true also if we replace the above by a strict inequality.
\end{lem}

We assume the inital data satisfies the physical constraint. Namely, we are given $\theta_0(z),q_0(z)\in L^{\infty}([0,1])$, such that $z\longmapsto\theta_0(z)$ is monotone increasing, and $q_0(z)\leq Q^{sat}(\theta_0(z),z,0)$ for $a.e$-$z$. Inspired by the previous discussions, we propose the following definition of weak Lagrangian solutions.

\begin{defn}\label{1.3d}
Let $\theta_t(z),\,q_t(z)\in L^{\infty}([0,T)\times[0,1])\cap C([0,T),L^1([0,1]))$, and $F, F^*:[0,T)\times[0,1]\rightarrow[0,1]$ be \sloppy Borel measure preserving maps such that $F_t(\cdot), F_t^*(\cdot)\in C([0,T);L^1([0,1]))$, and $F_{\cdot}(z)\in L^{\infty}([0,1];BV([0,T))$. Let $\theta_0(z),q_0(z)$ be as in previous paragraph. Denote $\hat{\theta}_t(z)=\theta_t(F_t(z))$ and $\hat{q}_t(z)=q_t(F_t(z))$. Then we say $(q_t,\theta_t,F_t)$ is a weak Lagrangian solution to initial data $\theta_0$,$q_0$ if the following holds:
\begin{equation*}
\begin{split}
&\textrm{(i)$\theta_t\rightarrow\theta_0$,  $q_t\rightarrow q_0$ in $L^1([0,1])$, $F_t\rightarrow id$ in $L^1([0,1])$ as $t\rightarrow0$.}\\
&\textrm{(ii)$z\longmapsto\theta_t(z)$ is monotone increasing for each $t\in[0,T)$.}\\
&\textrm{(iii)For any $t\in[0,T)$, $F_t\circ F_t^*(z)=z$, $F_t^*\circ F_t(z)=z$ for $\mathcal{L}^1$-a.e $z\in[0,1]$.}\\
&\textrm{(iv)For $\mathcal{L}^1-$a.e $z\in[0,1]$,  $\hat{\theta}_t(z)+\hat{q}_t(z)
=\theta_0(z)+q_0(z)$, for $\mathcal{L}^1-a.e$ $t\in(0,T)$}\\
&\textrm{(v)For $\mathcal{L}^1-a.e$ $z\in[0,1]$, $t\longmapsto\hat{\theta}_t(z)\leq
\hat{\theta}_{t^{\prime}}(z)$ for $\mathcal{L}^2-a.e-(t,t^{\prime})$ with $t<t^{\prime}$}\\
&\textrm{(vi)$\partial_t\hat{\theta}(\cdot,z)=[\partial_t(Q^{sat}(\hat{\theta}(\cdot,z),F_{\cdot}(z),\cdot)]^-\lfloor E_z$, where $E_z=\{t\in(0,T):\hat{q}_t^*(z)=Q^{sat}(\hat{\theta}^*_t(z),F_t(z),t)\}$,}\\
&\textrm{and $\hat{q}^*_t(z)$, $\hat{\theta}_t^*(z)$ is the monotone, left continuous version of $\hat{q}_t(z)$, $\hat{\theta}_t(z)$. chosen according to}\\
&\textrm{ Remark \ref{1.4r}.}
\end{split}
\end{equation*}
\end{defn}
\begin{rem}
Let $f:(0,T)\rightarrow\mathbf{R}$ be Borel measurable, such that $f=\tilde{f}$ for $\mathcal{L}^1-a.e$ $t\in(0,T)$, with $\tilde{f}\in BV(0,T)$, then $\partial_tf$ is a finite signed measure, defined by
$$
\int_0^Tf(t)\partial_t\zeta(t)dt=-\int_0^T\zeta(t)(\partial_tf)(dt),\,\forall\zeta\in C_c^1((0,T)).
$$
If we choose $\tilde{f}$ such that it is left continuous, then:
$$
\partial_tf([a,b))=\tilde{f}(b)-\tilde{f}(a),\textrm{ for any $[a,b)\subset(0,T)$.}
$$
\end{rem}
\begin{rem}\label{1.4r}
Let $f:(0,T)\rightarrow\mathbb{R}$ be Borel measurable such that for $\mathcal{L}^2-a.e$ $(t,t^{\prime})\in(0,T)^2$ with $t<t^{\prime}$, 
we have $f(t)\leq f(t^{\prime})$, then there exists a unique function $f^*:(0,T)\rightarrow\mathbb{R}$, monotone increasing, continuous from the left, such that $f=f^*$ for $\mathcal{L}^1-$a.e $t\in(0,T)$.
\end{rem}
It turns out that above definition of weak Lagrangian solutions is still too strong, and one cannot expect the existence of solutions in the sense defined above except for some special choice of the function $Q^{sat}$ and the initial data. 

One difficulty with the system (\ref{1.1})-(\ref{1.4}) is that we do not have much regularity in space. The only regularity in space comes from the monotonicity of $\theta$, and in general, 
no regularity in space for $q$, as well as the flow maps $F_t$, $F_t^*$. This means we lack the necessary compactness to get a function $q_t(z)$, or a measure preserving flow map $F,F^*$ in the limit.

The evolution of $\theta$ and $q$ is highly unstable under small perturbations of the initial data, which can be seen from the construction of the discrete problem. This suggests the use of probabilistic description of the solution. 
Under this description, heuristially for each time $t$, we have a certain probability distribution for $\{\theta,q(z)\}_{z\in[0,1)}$, 
and we make a random choice of $\theta$, which is a monotone increasing function on $[0,1)$, also make a random choice for $q(z)$ for each $z$, 
according to this probability distribution and then evolve.
 This determines the probability distribution for $\{\theta,q(z)\}_{z\in[0,1)}$ at later times. In this spirit, 
we need to prescribe some probability distribution as initial data.

On the other hand we need the correct equation to be satisfied (point (vi) of Definition \ref{1.3d}), this suggests considering some "path-spaces" which describes all the possible paths of some parcel. 
Inspired by the probabilistic approach of transport equation, 
we wish to obtain the solution as a measure in some path space, 
and the correct probability distribution is obtained by projecting to each $t$.

We will make the above heuristic discussions rigorous in section 5.

\section{Solution of the discrete problem}

In this section, we construct discrete solutions following the method of Bokhove $et\,a.$ \cite{report} and do estimates about them. 

The discrete procedure is designed to reflect the underlying physics of the problem, as expressed in Met Office \cite{Met}, chapter 4. It is based on a  representation of the fluid as discrete parcels, so that $\hat\theta$ and $\hat q$ are piecewise constant. The initial values satisfy the physical constraints (\ref{1.10}) and (\ref{1.6}) at $t=0$. We define $Q^{sat}$ to be a monotonically decreasing function of time and discretise the time variation. Thus after some time interval the constraint (\ref{1.10}) will be violated. 

The Lagrangian form of the equations (\ref{1.7})-(\ref{1.10}) is solved by representing the flow map $F_t$ as a rearrangement of the fluid parcels. The evolution of $\hat\theta$ and $\hat q$ on each parcel is computed using (\ref{1.7}) and (\ref{1.8}). If (\ref{1.10}) is violated for any parcel, then (\ref{1.8}) is used to update $\hat\theta$ and set $q$ equal to $Q^{sat}$. The update to $\hat\theta$ may result in the constraint (\ref{1.6}) being violated, in which case the parcels have to be rearranged to restore the constraint. $\hat\theta^M=\hat\theta+\hat q$ is conserved for each parcel under the rearrangement as required by (\ref{1.7}), and $\hat\theta\ge\Theta(\hat\theta^M,z,t)$ at the final positions beacuse of Lemma \ref{1.1l}. 

As found by Bokhove $et\,a.$\cite{report}, finding this rearrangement is non-trivial because of the dependence of $Q^{sat}$ on $\theta$ and $z$, and because (\ref{1.10}) may be violated on several parcels simultaneously. We call these 'wet' parcels. In this case there may be many ways to satisfy the constraints. The physics of the problem requires that the final position of the wet parcel with the largest $\hat\theta^M$ is determined first. This is done by moving it upwards, thus increasing $\hat\theta$, until it encounters a larger value of $\hat\theta$ at some $z=z_t$. We refer to this as the parcel "beating" all other parcels with $z<z_t$. All overtaken parcels have to move down to compensate for the upward displacement. Extreme care is required in showing that this procedure has a well defined limit as the timestep tends to zero.

We now define this procedure precisely. Denote $z_i=\frac{i}{n}$, $J_i=[\frac{i-1}{n},\frac{i}{n})$, for any integer $1\leq i\leq n$. Let $\{\theta_j^n\}_{j=1}^n$, and $\{q_j^n\}_{j=1}^n$ be given, such that $\theta_j^n\leq\theta_{j+1}^n$, and $q_j^n\leq Q^{sat}(\theta_j^n,z_j,0)$,
 for any $1\leq j\leq n$. This means that a discrete version of (\ref{1.6}) and (\ref{1.10}) is satisfied. It follows from Lemma \ref{1.1l} that $\theta_j^n\geq
\Theta(\theta_j^n+q_j^n,z_j,0)$.
Let $\delta t=\frac{1}{Cn}$ for some large constant $C>0$ to be determined later on. This will be chosen so that it depends only on the function $Q^{sat}$, $T$, and the initial data. 
Define the wet set at time step 0 to be $W_{n}=\{1\leq j\leq n:\theta_j^n<\Theta(\theta_j^n+q_j^n,z_j,\delta t)\}$. We will also denote $\theta_j^{M,n}=\theta_j^n+q_j^n$.
 First we decide which parcels move to $z_n$. Define 
\begin{equation*}
\begin{split}
W_{n}^{\prime}=&\{j_0\in W_{n}:\textrm{for any $j>j_0$ with $j\notin W_{n}$,  $\theta_j^n<
\Theta(\theta_{j_0}^{M,n},z_j,\delta t)$, and if $j\in W_{n}$,}\\
&\theta_{j_0}^{M,n}>\theta_j^{M,n}\}.
\end{split}
\end{equation*}
Here we make the convention that $n\in W_n^{\prime}$ if and only if $n\in W_n$.
The set $W_n^{\prime}$ is exactly the set of parcels which are "wet" and can beat all other parcels above up to $n$. We will sometimes call them "eligible".
First assume $W_{n}^{\prime}\neq\emptyset$. Let $j_0\in W_{n}^{\prime}$ be the parcel with the largest $\theta_{j}^{M,n}$ among $W_n^{\prime}$(if there are more than one such parcels, simply choose $j_0$ to be largest possible), define the first rearrangement 
$$
\sigma_n(k) = \left\{ \begin{array}{ll}
k & \textrm{if $1\leq k<j_0$;}\\
\,n & \textrm{if $k=j_0$;}\\
\,k-1& \textrm{if $j_0<k\leq n$.}
\end{array} \right.
 $$
To explain this in English, a parcel can jump to $z_n$ only if it is wet and has the largest $\theta^{M,n}$ among all "eligible" parcels.

We also update $\theta_j^n$ after the first rearrangement in the following way:
$$
\theta_j^{n,n-1} = \left\{ \begin{array}{ll}
\theta_{\sigma_n^{-1}(j)}^n & \textrm{if $j\neq n$;}\\
\,\Theta(\theta_{j_0}^{M,n},z_j,\delta t) & \textrm{if $j=n$.}\\
\end{array} \right.
$$
That is, we update the $\theta$ of parcels which jumped according to its final position, and leave the $\theta$ of other parcels unchanged. In Lagrangian coordinates,
 define $\hat{\theta}_j^{n,n-1}=
\theta_{\sigma_n(j)}^{n,n-1}$, $\hat{q}_j^{n,n-1}=
\theta_j^n+q_j^n-
\hat{\theta}_j^{n,n-1}$.
This is consistent with (\ref{1.7}).
Define the new wet set 
\begin{equation}\label{2.1}
\begin{split}
&W_{n-1}=\{1\leq j\leq n:\theta_j^{n,n-1}<\Theta(\theta_{\sigma_n^{-1}(j)}^{M,n},z_j,\delta t)\},\\
&W_{n-1}^{\prime}=\{j_0\in W_{n-1}:\textrm{for any $j_0<j\leq n-1$ with $j\notin W_{n-1}$, $\theta_j^{n,n-1}<\Theta(\theta_{\sigma_n^{-1}(j_0)}^{M,n},z_j,\delta t)$.}\\
&\textrm{ or if $j\in W_{n-1}$, $\theta_{\sigma_n^{-1}(j_0)}^{M,n}>\theta_{\sigma_n^{-1}(j)}^{M,n}$.}\}.
\end{split}
\end{equation}
Notice that $W_{n-1}\subset\{1,2,\cdots,n-1\}$.

If $W_{n}^{\prime}=\emptyset$, then simply take $\sigma_n=id$, and take $\theta_j^{n,n-1}=\theta_j^n$, $q_j^{n,n-1}=q_j^n$. Then we have $W_{n-1}=W_{n}$.

Next we repeat the above procedure to $\{\theta_j^{n,n-1}\}_{j=1}^{n-1}$, $\{q_j^{n,n-1}\}_{j=1}^{n-1}$, and the new wet set defined by (\ref{2.1}). Let $\sigma_{n-1}$ be the resulting rearrangement of the first $n-1$ parcels. Let $\sigma_{n-1}(n)=n$, so that it becomes the rearrangement for $n$ parcels.

In general, let $\sigma_k$ be the rearrangement when we decide which parcel moves to $z_k$ with $\sigma_k(l)=l$ for $l>k$. Denote $\beta_k=\sigma_{k+1}\circ\cdots\circ\sigma_n$, with $\beta_n=id$. 
Let $\{\theta_j^{n,k}\}_{j=1}^n$, $\{q_j^{n,k}\}_{j=1}^n$ be the updated $\theta$ and $q$ after $\sigma_{k+1}$. We also denote $\theta_j^{n}=\theta_j^{n,n}$, and $q_j^n=q_j^{n,n}$.
The wet set at this stage is given by
\begin{equation}\label{2.3n}
\begin{split}
&W_{k}=\{1\leq j\leq n:\theta_j^{n,k}<\Theta
(\theta^{M,n}_{\beta_k^{-1}(j)},z_j,\delta t)\}\\
&W_{k}^{\prime}=\{j_0\in W_{k}:\textrm{ for any $j_0<j\leq k$ with $j\notin W_{k}$, $\theta_j^{n,k}<\Theta(\theta^{M,n}_{\beta_k^{-1}(j_0)},z_j,\delta t)$,}\\
&\textrm{or if $j\in W_{k}$, $\theta^{M,n}_{\beta_k^{-1}(j_0)}>\theta^{M,n}_{\beta_k^{-1}(j)}$.}\}.
\end{split}
\end{equation}
As before, we make the convention that if $j_0\geq k$, then $j_0\in W_k^{\prime}$ if and only if $j_0\in W_k$. The sets $W_k$, $W_k^{\prime}$ determines the evolution when we decide which parcel moves to $z_{k-1}$.
The following inductive formula holds when $W_{k+1}^{\prime}\neq\emptyset$. 
Let $j_*\in W_{k+1}^{\prime}$ be such that $\theta^{M,n}_{\beta_{k+1}^{-1}(j_*)}\geq\theta^{M,n}_{\beta_{k+1}^{-1}(j)}$ for any $j\in W_{k+1}^{\prime}$, then we move this parcel to $z_{k+1}$, namely,
\begin{equation}\label{2.2nn}
\sigma_{k+1}(j) = \left\{ \begin{array}{ll}
j & \textrm{if $j<j_*$ or $j>k+1$;}\\
k+1 & \textrm{if $j=j_*$;}\\
j-1 &\textrm{if $j_*<j \leq k+1$.}
\end{array} \right.
\end{equation}
We update $\theta_j^n$ accordingly:
\begin{equation}\label{2.2n}
\theta_j^{n,k} = \left\{ \begin{array}{ll}
\theta_{\sigma_{k+1}^{-1}(j)}^{n,k+1} & \textrm{if $j\neq k+1$;}\\
\,\Theta(\theta_{\beta_k^{-1}(k+1)}^{M,n},z_{k+1},\delta t) & \textrm{if $j=k+1$.}\\
\end{array} \right.
\end{equation}
If $W_{k+1}^{\prime}=\emptyset$, then simply put $\sigma_{k+1}=id$ and $\theta_j^{n,k}=\theta_j^{n,k+1}$.

Let $\hat{\theta}_j^{n,k}=
\theta_{\beta_k(j)}^{n,k}$. Define $q_j^{n,k}=q_{\sigma_{k+1}^{-1}(j)}^{n,k+1}+\theta_{\sigma_{k+1}^{-1}(j)}^{n,k+1}-\theta_j^{n,k}$.
 $\hat{q}_j^{n,k}=q_{\beta_k(j)}^{n,k}$. 
We observe some useful properties of above rearrangement algorithm:
\begin{lem}\label{3.1nl}
For each index $j\in\{1,2,\cdots,n\}$, one of the following must hold:\\
(i)There exists a unique $k_1\in\{1,2,\cdots,n\}$, such that $\beta_{k_1-1}(j)=\sigma_{k_1}(\beta_{k_1}(j))=k_1>\beta_{k_1}(j)$. Besides, for any $k\geq k_1$, $\beta_k(j)\leq\beta_{k+1}(j)$, and any $k\leq k_1-1$, $\beta_k(j)=k_1$.\\
(ii)$\beta_k(j)\leq \beta_{k+1}(j)$ for any $0\leq k\leq n-1$.\\
Morover, if for some $k_2$, $\beta_{k_2}(j)\leq k_2$ and $\notin W_{k_2}$, then the second alternative must hold.
On the other hand, if for some $j_1,\,j_2$, and some $k_3$, 
it holds $\beta_{k_3}(j_1)<\beta_{k_3}(j_2)$, but $\beta_{k_3-1}(j_1)>\beta_{k_3-1}(j_2)$, then the first alternative above holds for $j_1$ with $k_1=k_3$.
\end{lem}
This lemma says that for any given parcel, either it experiences no lifts at all among the $\sigma_k$`s, or there is a unique $\sigma_k$ which lifts this parcel and it stays there in the latter rearrangements of the same time step. If a parcel becomes dry in a certain time step, then it will stay dry in the latter arrangement. The only way the order of two parcels can change is that the lower parcel experience a jump.
\begin{proof}
Fix an index $j$. Suppose there exists some $0\leq k_1\leq n-1$, for which $\beta_{k_1-1}(j)>\beta_{k_1}(j)$.
 From the definition of the $\sigma_k$ given in (\ref{2.2nn}), we see that if $\sigma_k(j)>j$ for some $j$, $k$, it must hold that $\sigma_k(j)=k$.
Hence from $\beta_{k_1-1}(j)=\sigma_{k_1}(\beta_{k_1}(j))>\beta_{k_1}(j)$, we see $\sigma_{k_1}(\beta_{k_1}(j))=k_1$. If for some $k\geq k_1$,
 $\beta_k(j)>\beta_{k+1}(j)$, then the same argument shows $\beta_k(j)=\sigma_{k+1}(\beta_{k+1}(j))=k+1$. 
Now for any $k'\leq k$, we have $\sigma_{k'}(\beta_k(j))=\beta_k(j)$, a contradiction. 
This proves for $k\geq k_1$, $\beta_k(j)\leq \beta_{k+1}(j)$. The case for $k\leq k_1-1$ follows directly from definition of $\sigma_k$. 

To see the moreover part,  observe that for any $k\geq k_2$, we must have $\beta_k(j)\leq\beta_{k+1}(j)$.
If not, by the argument given in the first part, we can then conclude $\beta_{k^{\prime}}(j)=k+1$, for any $k^{\prime}\leq k$. In particular, this means $\beta_{k_2}(j)=k+1>k_2$, a contradiction. Since $\beta_{k_2}(j)\notin W_{k_2}$, from (\ref{2.2nn}), we see $\beta_{k_2-1}(j)=\sigma_{k_2}(\beta_{k_2}(j))\leq\beta_{k_2}(j)$.
This means that
$$
\theta^{n,k_2-1}_{\beta_{k_2-1}(j)}=\theta_{\beta_{k_2}(j)}^{n,k_2}\geq\Theta
(\theta_j^{M,n},z_{\beta_{k_2}(j)},\delta t)\geq\Theta(\theta_j^{M,n},z_{\beta_{k_2-1}(j)},\delta t).
$$
It follows that $\beta_{k_2-1}(j)\notin W_{k_2-1}$. Hence one conclude $\beta_{k_2-2}(j)=\sigma_{k_2-1}(\beta_{k_2-1}(j))\leq\beta_{k_2-1}(j)$. Same argument as above applies and shows $\beta_{k_2-2}(j)\notin W_{k_2-2}$. One can apply the same argument and shows $\beta_k(j)$ is monotone decreasing in $k$.
The "on the other hand" part follows directly from the definirtion of $\sigma_k$ given in (\ref{2.2nn}).
\end{proof}
We want to show above define algorithm preserves a discrete version of the physical constraint:
\begin{lem}\label{2.1l}
(i) $\theta_j^{n,k}\leq\theta_{j+1}^{n,k}$ for any $1\leq j\leq n-1$;\\
(ii)$\hat{q}_j^{n,k}+\hat{\theta}_j^{n,k}=q_j^n+\theta_j^n.$\\
(iii)$\hat{\theta}_j^{n,k}\geq\hat{\theta}_j^{n,k+1}$.\\
(iv)$q_j^{n,k}\leq Q^{sat}(\theta_j^{n,k},z_j,0)$, for $1\leq j\leq k$, and $q_j^{n,k}\leq Q^{sat}(\theta_n^{n,k},z_j,\delta t)$, for $k+1\leq j\leq n$.
\end{lem}
\begin{proof}
First we prove the point (ii). From our definition, we know that $q_j^{n,k}+\theta_j^{n,k}=q_{\sigma_{k+1}^{-1}(j)}^{n,k+1}+\theta_{\sigma_{k+1}^{-1}(j)}^{n,k+1}$. From this it immediately follows $\hat{\theta}_j^{n,k}+\hat{q}_j^{n,k}
=\hat{\theta}_j^{n,k+1}
+\hat{q}_j^{n,k+1}$.

We prove the other three statements by induction on $n$. First observe that statements (i)-(iv) are true for $k=n$.(point (iii) is empty when $k=n$.) Now assume these are true for $k+1$ and above with $1\leq k+1\leq n$, wish to prove these for $k$.

Now we prove point (iii) for $k$, assuming $W_{k+1}^{\prime}\neq\emptyset$. One can see point (iii) is equivalent to $\theta_{\sigma_{k+1}(j)}^{n,k}\geq\theta_j^{n,k+1}$, by our definition of $\hat{\theta}_j^{n,k}$. If $\sigma_{k+1}(j)\neq k+1$,
 then one has $\theta_{\sigma_{k+1}(j)}^{n,k}=\theta_j^{n,k+1}$ from (\ref{2.2n}). 
Now if $\sigma_{k+1}(j_*)=k+1$, then
$$
\theta_{k+1}^{n,k}=\Theta
(\theta^{M,n}_{\beta_{k+1}^{-1}(j_*)},z_{k+1},\delta t)\geq\theta_{k+1}^{n,k+1}\geq\theta_{j_*}^{n,k+1}.
$$
The first inequality above used the fact that $j_*\in W_{k+1}^{\prime}$, hence it must "beat" the parcel originally at $z_{k+1}$. The second inequality used the induction hypothesis that point (i) holds with $k+1$. If $W_{k+1}^{\prime}=\emptyset$, then we simply have $\sigma_{k+1}=id$, and $\theta_j^{n,k}=\theta_j^{n,k+1}$, so there is nothing to prove.

Then we prove point (i). We only consider the case when $W_{k+1}^{\prime}\neq\emptyset$, otherwise nothing is changed by $\sigma_{k+1}$ and the proof is trivial. Let $j_*\leq k+1$ be such that $\sigma_{k+1}(j_*)=k+1$.

The only nontrivial cases to check is when $j=k$ and $j=k+1$, the rest of the cases will follow from (\ref{2.2nn}), (\ref{2.2n}) and the induction hypothesis that (i) holds for $k+1$. So it boils down to prove \begin{equation}\label{3.5n}
\theta_{k+1}^{n,k+1}\leq\Theta
(\theta_{\beta_{k+1}^{-1}(j_*)}^{M,n},z_{k+1},\delta t)\leq\theta_{k+2}^{n,k+1}.
\end{equation} 
The first part of the inequality follows from that $j_*\in W_{k+1}^{\prime}$, that is,
 it needs to "beat" the parcel originally at $z_{k+1}$ in order to rise to $z_{k+1}$. To be precise, suppose that $k+1\notin W_{k+1}$, then we know from $j_*\in W_{k+1}^{\prime}$ and the formula for $W_{k}^{\prime}$ in (\ref{2.3n}) that $\theta_{k+1}^{k+1}<
\Theta(\theta^{M,n}_{\beta_{k+1}^{-1}(j_*)},z_{k+1},\delta t)$. 
This is exactly what we want. Now if $k+1\in W_{k+1}$,
 then again from $j_*\in W_{k+1}$ one concludes that $\theta^{M,n}_{\beta_{k+1}^{-1}(j_*)}\geq\theta_{\beta_{k+1}^{-1}(k+1)}^{M,n}$. Hence
$$
\Theta(\theta_{\beta_{k+1}^{-1}(j_*)}^{M,n},z_{k+1},\delta t)\geq\Theta(\theta^{M,n}_{\beta_{k+1}^{-1}(k+1)},z_{k+1},\delta t)\geq\theta_{k+1}^{n,k+1}.
$$
The first inequality used the monotonicity of $\Theta$ with respect to $\theta^M$, the second inequality used that $k+1\in W_{k+1}$.
This proves the first part of (\ref{3.5n}).

The reason why the second part of the inequality holds is that if it were not true, then $j_*$ would have risen to $z_{k+2}$ instead of $z_{k+1}$ in the rearrangement $\sigma_{k+2}$. 
To make this precise, let $\sigma_{k+2}(j_{**})=j_*$, then we must have $j_{**}\geq j_*$. If not, we will have $\sigma_{k+2}(j_{**})=k+2=j_*$, not possible. Since $j_*\in W_{k+1}$, we have $$
\theta_{j_{**}}^{n,k+2}=\theta_{j_*}^{n,k+1}<
\Theta
(\theta^{M,n}_{\beta_{k+1}^{-1}(j_*)},z_{j_*},\delta t)\leq\Theta
(\theta^{M,n}_{\beta_{k+2}^{-1}(j_{**})},z_{j_{**}},\delta t).
$$
This implies $j_{**}\in W_{k+2}$. Now we claim that for any $j$ with $j_{**}<j\leq k+1$ and $j\notin W_{k+2}$, then $\theta_j^{n,k+2}<\Theta(\theta^{M,n}_{\beta_{k+2}^{-1}(j_{**})},z_j,\delta t)$.
Indeed, since $j\notin W_{k+2}$, we have $\sigma_{k+2}(j)\leq j$, $\sigma_{k+2}(j)\notin W_{k+1}$, and $\theta_j^{n,k+2}=\theta^{n,k+1}_{\sigma_{k+2}(j)}$. If the claim is not true, then

$$
\Theta
(\theta^{M,n}_{\beta_{k+1}^{-1}(j_{*})},z_{\sigma_{k+2}(j)},\delta t)\leq\Theta
(\theta^{M,n}_{\beta_{k+1}^{-1}(j_{*})},z_j,\delta t)\leq\theta_{j}^{n,k+2}=\theta_{\sigma_{k+2}(j)}^{n,k+1}.
$$
Notice $\sigma_{k+2}(j)\leq k+1$, also $\sigma_{k+2}(j)>\sigma_{k+2}(j_{**})=j_*$, this contradicts $j_*\in W_{k+1}^{\prime}$.

Let $j_1$ be the maximal $j$ such that $j\geq j_{**}$, $j\in W_{k+2}$, and $\theta^{M,n}_{\beta_{k+2}^{-1}(j_{**})}\leq\theta^{M,n}_{\beta_{k+2}^{-1}(j)}$. From the induction hypothesis with $k+2$ and point (iv), we know $j_1\leq k+2$. Consider 2 cases:

If $j_1\notin W_{k+2}^{\prime}$. First observe for any $j>j_1$ and $j\in W_{k+2}$, we must have $\theta^{M,n}_{\beta_{k+2}^{-1}(j)}<\theta^{M,n}_{\beta^{-1}_{k+2}(j_1)}$.
Otherwise it will contradict the maximality of $j_1$. Also for any $j$ with $j_1<j\leq k+1$, and $j\notin W_{k+2}$, we conclude from the claim $\theta_j^{n,k+2}<\Theta(\theta^{M,n}_{\beta_{k+2}^{-1}(j_{**})},z_j,\delta t)\leq\Theta(\theta^{M,n}_{\beta_{k+2}^{-1}(j_1)},z_j,\delta t)$. 
The only possibility remains is that $k+2\notin W_{k+2}$, and $\theta^{n,k+2}_{k+2}\geq\Theta(\theta^{M,n}_{\beta_{k+2}^{-1}(j_1)},z_{k+2},\delta t)$. That is, $k+2$ is a dry parcel and cannot be beaten by $j_1$. Hence
$$
\theta_{k+2}^{n,k+1}\geq
\theta_{k+2}^{n,k+2}\geq
\Theta(\theta^{M,n}_{\beta^{-1}_{k+2}(j_1)},z_{k+2},\delta t)\geq\Theta(\theta^{M,n}_{\beta_{k+2}^{-1}(j_{**})},z_{k+1},\delta t).
$$
This is what we want. 

If $j_1\in W_{k+2}^{\prime}$, let $j_2\leq k+2$ be such that $\sigma_{k+2}(j_2)=k+2$. From the definition of the procedure, 
we have $\theta^{M,n}_{\beta_{k+2}^{-1}(j_2)}\geq\theta^{M,n}_{\beta_{k+2}^{-1}(j_1)}$, since the parcel that actually jumps up should have the largest $\theta^M$ among all "eligible" parcels.  From the inductive formula (\ref{2.2n}), we see 
$$
\theta_{k+2}^{n,k+1}=\Theta(\theta^{M,n}_{\beta_{k+2}^{-1}(j_2)},z_{k+2},\delta t)\geq\Theta(\theta^{M,n}_{\beta^{-1}_{k+2}(j_1)},z_{k+2},\delta t)\geq\Theta(\theta^{M,n}_{\beta_{k+2}^{-1}(j_{**})},z_{k+1},\delta t).
$$
So far we finished the proof of point (i).

It only remains to show the point (iv). This is equivalent to showing $\theta_j^{n,k}\geq \Theta(\theta_j^{n,k}+q_j^{n,k},z_j,0)$ for $1\leq j\leq k$, and $\theta_j^{n,k}\geq\Theta(\theta_j^{n,k}+q_j^{n,k},z_j,\delta t)$, for $k+1\leq j\leq n$. To see the first part, we know for $1\leq j\leq k$, $\theta_j^{n,k}=\theta_{\sigma_{k+1}^{-1}(j)}^{n,k+1}$, and we also know from point (ii) already proved that $\theta_j^{n,k}+q_j^{n,k}=\theta^{n,k+1}_{\sigma_{k+1}^{-1}(j)}+q_{\sigma_{k+1}^{-1}(j)}^{n,k}$. 
Also $\sigma_{k+1}^{-1}(j)\leq k+1$, since the rearrangement $\sigma_{k+1}$ never moves down a parcel by 2. Apply the induction hypothesis that (iv) holds for $k+1$ we see
$$
\theta_j^{n,k}=\theta_{\sigma_{k+1}^{-1}(j)}^{n,k+1}\geq\Theta(\theta^{n,k+1}_{\sigma_{k+1}^{-1}(j)}+q^{n,k+1}_{\sigma_{k+1}^{-1}(j)},z_{\sigma_{k+1}^{-1}(j)},0)\geq\Theta(\theta_j^{n,k}+q_j^{n,k},z_j,0).
$$
To see the second part, consider first when $W^{\prime}_{k+1}\neq\emptyset$, and if $j=k+1$, then we know from (\ref{2.2n})
$$
\theta^{n,k}_{k+1}=\Theta(\theta^{M,n}_{\beta_k^{-1}(k+1)},z_{k+1},\delta t)=\Theta(q^{n,k}_{k+1}+\theta^{n,k}_{k+1},z_{k+1},\delta t).
$$
In the second equality above, we used the point (ii) already proved, and also the definition of $\theta^{M,n}_j$ given in the beginning of this section. If instead $W_{k+1}^{\prime}=\emptyset$, then we know in particular $k+1\notin W_{k+1}$, hence
$$
\theta_{k+1}^{n,k}=\theta_{k+1}^{n,k+1}\geq
\Theta(\theta^{M,n}_{\beta_k^{-1}(k+1)},z_{k+1},\delta t)\geq\Theta(\theta_{k+1}^{n,k}+q_{k+1}^{n,k},z_{k+1},\delta t).
$$
If $k+2\leq j\leq n$, note $\sigma_{k+1}(j)=j$ we use the induction hypothesis and (\ref{2.2nn}) to conclude 
$$
\theta_j^{n,k}=\theta_j^{n,k+1}\geq
\Theta(\theta_j^{n,k+1}+q_j^{n,k+1},z_j,\delta t)=\Theta(\theta_j^{n,k}+q_j^{n,k},z_j,\delta t).
$$
This finishes the proof.
\end{proof}
Denote $\theta^n_j(\delta t)=\theta_j^{n,0}$, $q_j^n(\delta t)=q_j^{n,0}$. 
Then we have $q_j^n(\delta t)\leq Q^{sat}(\theta_j^n(\delta t),z_j,\delta t)$, 
for any $1\leq j\leq n$, and $j\longmapsto\theta_j^n(\delta t)$ is monotone increasing by Lemma \ref{2.1l}.  
Define the flow map at first time step $\tilde{F}^n_{\delta t}:[0,1)\rightarrow[0,1)$ be such that it shifts $J_i$ to $J_{\beta_0(i)}$ by translation, that is $\tilde{F}^n_{\delta t}(z)=z-z_i+z_{\beta_0(i)}$ for $z\in J_i$. 
Then $\tilde{F}_{\delta t}^n\#\mathcal{L}^1_{[0,1]}=\mathcal{L}^1_{[0,1]}$.
Apply the previous procedure to $\{\theta_j^n(\delta t)\}_{j=1}^n$, $\{q_j^n(\delta t)\}_{j=1}^n$, 
but with $Q^{sat}$ evaluated at $\delta t$ to get $\{\theta_j^n(2\delta t)\}_{j=1}^n$, $\{q_j^n(2\delta t)\}_{j=1}^n$, 
and the corresponding flow map $\tilde{F}_{2\delta t}^n:[0,1)\rightarrow[0,1)$. 
Repeating the procedure,
 we get a sequence of solutions at discrete times $\{\theta_j^n(k\delta t)\}_{j=1}^n$, $q_j^n(k\delta t)\}_{j=1}^n$, and a sequence of flow maps $\tilde{F}^n_{k\delta t}$ connecting $k\delta t$ and $(k+1)\delta t$.
Denote $\theta^{M,n}_j(k\delta t)=\theta_j^n(k\delta t)+q_j^n(k\delta t)$.
 Here $k$ is an integer with $0\leq k\leq\frac{T}{\delta t}+1$.
Define $F_{k\delta t}=\tilde{F}_{k\delta t}\circ\cdots \tilde{F}_{\delta t}$. We will also denote $\alpha_{k\delta t}$, $\tilde{\alpha}_{k\delta t}$ be the corresponding rearrangement map on the discrete indices $\{1,2,\cdots,n\}$.
Denote $\theta^n(t,z)=\theta_j^n(k\delta t)$, $q^n(t,z)=q_j^n(k\delta t)$ if $z\in J_j$ and $k\delta t\leq t<(k+1)\delta t$. 
Also $F^n(t,z)=F_{k\delta t}$ if $k\delta t\leq t<(k+1)\delta t$. 
Define $\theta_0^n(z)=\theta_j^n$, $q^n_0(z)=q_j^n$, and $\theta^{M,n}_0(z)=
\theta^n_0(z)+q^n_0(z)$, for $z\in J_j$.
We deduce a immediate corollary of Lemma \ref{2.1l}.
\begin{cor}\label{2.2c}
(i)$z\longmapsto\theta^n(t,z)$ is monotone increasing for any $t\in[0,T)$.

(ii)Denote $\hat{\theta}^n(t,z)=\theta^n(t,F^n(t,z))$, $\hat{q}^n(t,z)=q^n(t,F^n(t,z))$, then we have $\hat{\theta}^n(t,z)+\hat{q}^n(t,z)=
\theta^n_0(z)+q^n_0(z)$.

(iii)$t\longmapsto\hat{\theta}^n(t,z)$ is monotone increasing for any $z\in[0,1)$.

(iv)$q^n(t,z)\leq Q^{sat}(\theta^n(t,z),z,k\delta t)$, where $k$ is the integer such that $k\delta t\leq t<(k+1)\delta t$.
\end{cor}

Now we can have a more precise description of the motion of a single particle.
\begin{lem}\label{3.4nl}
Suppose that $n\delta t<\frac{\inf\partial_z\Theta}{\sup|\partial_t\Theta|}$, where both sup and inf are taken on the set $\{(w,z,t):|w|\leq\max_j|\theta^{M,n}_j|\,,z\in[0,1],\,t\in[0,T]\}$, then one of the following must hold:\\
(i)There exists a unique $k_1\in\{1,2,\cdots,n\}$, such that $\beta_{k_1-1}(j)=\sigma_{k_1}(\beta_{k_1}(j))=k_1>\beta_{k_1}(j)$. Besides, for any $k\geq k_1$, $\beta_k(j)=\beta_{k+1}(j)$, and any $k\leq k_1-1$, $\beta_k(j)=k_1$.\\
(ii)$\beta_k(j)\leq \beta_{k+1}(j)$ for any $0\leq k\leq n-1$.
\end{lem}
This lemma says that if $n\delta t$ is small enough, then for any parcel experiencing a jump, the rearrangements before and after the jump will fix this parcel. In particular, if a parcel gets pushed down($\beta_{k-1}(j)<\beta_k(j)$ for some $k$), then we must be in the second alternative, and by Lemma \ref{3.1nl}, it cannot overtake any other parcel.
\begin{proof}
The only difference between this lemma and Lemma \ref{3.1nl} is that in the first alternative, we can now conclude $\beta_k(j)=\beta_{k+1}(j)$ for any $k\geq k_1$.
Suppose we are in the first alternative of Lemma \ref{3.1nl}, and $\beta_k(j)<\beta_{k+1}(j)$ for some $k\geq k_1$, we will show that $\beta_k(j)\notin W_k$. 
Clearly we have $\beta_{k+1}(j)\leq k+1$, since $\sigma_{k+1}$ fix all index strictly bigger than $k+1$. From Lemma \ref{2.1l}, we know $\theta_{\beta_{k+1}(j)}^{n,k+1}\geq\Theta(\theta_j^{M,n},z_{\beta_{k+1}(j)},0)$. Since $\sigma_{k+1}$ moves $\beta_{k+1}(j)$ down, it does not change the value of $\theta$, hence $\theta^{n,k+1}_{\beta_{k+1}(j)}=\theta^{n,k}_{\beta_k(j)}$, and $\beta_{k+1}(j)\geq\beta_k(j)+1$. It follows that:
\begin{equation*}
\begin{split}
\theta^{n,k}_{\beta_k(j)}&=\theta_{\beta_{k+1}(j)}^{n,k+1}\geq\Theta(\theta_j^{M,n},z_{\beta_{k+1}(j)},0)\\
&\geq\Theta(\theta_j^{M,n},z_{\beta_k(j)},\delta t)+(\inf\partial_z\Theta) n^{-1}-\sup|\partial_t\Theta|\delta t>\Theta(\theta_j^{M,n},z_{\beta_k(j)},\delta t).
\end{split}
\end{equation*}
The last step used the smallness of $n\delta t$. Hence $\beta_k(j)\notin W_k$. It follows that $\beta_{k-1}(j)=\sigma_k(\beta_k(j))\leq\beta_k(j)$. Repeat the argument shown in the proof of "Moreover" part of Lemma \ref{3.1nl}, we see $\beta_k(j)$ will keep decreasing starting from $k$, and then no jump up is possible. 
\end{proof}
To conclude this section, we make a simple observation which will be useful in the next section.
\begin{lem}\label{2.3l}
Suppose for some pair of index $j_1,j_2\in\{1,2,\cdots,n\}$, and for some $k$, we have $\alpha_{k\delta t}(j_1)<\alpha_{k\delta t}(j_2)$, 
and $\alpha_{(k+1)\delta t}(j_1)>\alpha_{(k+1)\delta t}(j_2)$, then $\theta^{M,n}_{j_1}>\theta^{M,n}_{j_2}$. In particular, $\alpha_{l\delta t}(j_1)>\alpha_{l\delta t}(j_2)$ for all $l>k$.
\end{lem}
In plain English, this lemma says if the index $j_1$ is initially below $j_2$, then a necessary condition for $j_1$ to overtake $j_2$ is to have strictly larger $\theta^M$. 
This means, in particular, that $j_2$ cannot overtake $j_1$ again since $\theta^M$ is invariant along parcels.
\begin{proof}
Here we write $\tilde{\alpha}_{(k+1)\delta t}=\sigma_1\circ\cdots\circ\sigma_n$ and $\beta_k=\sigma_{k+1}\circ\cdots\circ\sigma_n$. Let $m_0$ be the maximal integer $m$ for which $\beta_{m-1}(\alpha_{k\delta t}(j_1))>\beta_{m-1}(\alpha_{k\delta t}(j_2))$. Then $\beta_{m_0}(\alpha_{k\delta t}(j_1))<\beta_{m_0}(\alpha_{k\delta t}(j_2))$, and $\beta_{m_0}(\alpha_{k\delta t}(j_1))\in W_m^{\prime}$.
 If $\beta_{m_0}(\alpha_{k\delta t}(j_2))\in W_m$, then we immediately have $\theta^{M,n}_{j_1}>\theta^{M,n}_{j_2}$. Otherwise,
$$
\Theta(\theta^{M,n}_{j_1},z_{\beta_{m_0}(\alpha_{k\delta t}(j_2))},(k+1)\delta t)>\theta_{\beta_{m_0}(\alpha_{k\delta t}(j_2))}^{n,m_0}(k\delta t)\geq\Theta(\theta^{M,n}_{j_2},z_{\beta_{m_0}(\alpha_{k\delta t}(j_2))},(k+1)\delta t).
$$
The first inequality used the definition of $W_m^{\prime}$, the second used definition of $W_m$.
\end{proof}

\section{Estimates on the discrete solution}

Next we do some estimates on the discrete solutions.
Denote $M'=||\theta_0^n||_{L^{\infty}(0,1)}+||q_0^n||_{L^{\infty}(0,1)}$.   
In the following, we say a constant is universal, if it only depends on $M'$, $T$, and $Q^{sat}$. 
We will derive the following estimates for the discrete solutions in this section. They are collected in the following theorem.
\begin{thm}\label{3.1t}

(i)$||\theta^n||_{L^{\infty}((0,1)\times(0,T))}+||q^n||_{L^{\infty}((0,1)\times(0,T))}\leq C_1$, for some universal constant $C_1$.

(ii)There exists a universal constant $C_2>0$, such that for any $\eps>\frac{C_2}{n}$, and any $t\in[0,T)$, if $\hat{\theta}^n(t,z)>\Theta(\theta^{M,n}(t,z),F^n_t(z),t)+\eps$, 
then we have $\hat{\theta}_{t^{\prime}}(z)=\hat{\theta}_t(z)$, and $F_{t^{\prime}}(z)\leq F_t(z)$ holds for any $t^{\prime}-t<\frac{\eps}{C_2}$.

(iii)For any $z\in[0,1]$, $TV_{t\in[0,T)}(F^n_t(z))\leq C_3$.

If $n\delta t<\frac{1}{C_4'}$ for some universal constant $C_4'>0$, then the following hold:

(iv)For any $\eps>0$, and any $[t-\eps,t+\eps]\subset[0,T)$, we have
$|\hat{\theta}^n_{t+\eps}(z)-\hat{\theta}^n_{t-\eps}(z)-(\Theta(\theta^{M,n}(z),F^n_{t+\eps}(z),t+\eps)-\Theta(\theta^{M,n}(z),F^n_{t-\eps}(z),t-\eps))^+|\leq2 C_4(\eps+\delta t)$.

(v)For any $s<t\in[0,T)$, we have $||\theta^n(t,\cdot)-\theta^n(s,\cdot)||_{L^1([0,1])}=||\hat{\theta}^n(t,\cdot)-\hat{\theta}^n(s,\cdot)||_{L^1([0,1])}\leq C_5\sqrt{t-s+\delta t}$.

(vi)For any $s<t\in[0,T)$, we have $||F^n(t,\cdot)-F^n(s,\cdot)||_{L^1([0,1])}\leq C_6\sqrt{t-s+\delta t}$ for some universal constant $C$.

\end{thm}
Thoughout this section, we make the following conventions: when we write expressions like $\sup|\partial_t\Theta|$ and so on, they are assumed to be taken over the set $\{(w,z,t)\in\mathbb{R}^3:|w|\leq M',\,z\in[0,1],\,t\in[0,T]\}$ unless otherwise stated.

We start with point (i) of above theorem.
\begin{lem}
There exists a universal constant $C_1>0$, such that
$$
||\theta^n||_{L^{\infty}((0,1)\times(0,T))}+||q^n||_{L^{\infty}((0,1)\times(0,T))}\leq C_1.
$$
\end{lem}
\begin{proof}
We need to go back to the construction of the discrete solution. 
First we know from Corollary \ref{2.2c} point (ii) that for any $z\in[0,1]$ $$\theta^n(t,z)+q^n(t,z)=\hat{\theta}^n(t,(F^n_t)^{-1}(z))+\hat{q}^n(t,(F_t^n)^{-1}(z))=\theta^n_0((F_t^n)^{-1}(z))+q^n_0((F_t^n)^{-1}(z)).$$
Therefore $||\theta^n+q^n||_{L^{\infty}}\leq M'$. 

On the other hand, from the construction of $\theta_l^n(k\delta t)$, 
we know either $\theta_l^n(k\delta t)=\theta^n_j((k-1)\delta t)$ for some $1\leq j\leq n$, or $\theta_l^n(k\delta t)=\Theta(\theta^{M,n}_j,z_l,k\delta t)$.
 In the former case, we have $|\theta_l^n(k\delta t)|\leq \max_j|\theta^n_j((k-1)\delta t)$. In the latter case, we have $|\theta_l^n(k\delta t)|\leq\sup|\Theta|$. 
Here $\sup$ is taken over the set $\{(w,z,t):|w|\leq M',\,z\in[0,1],\,t\in[0,T]\}$.
 But $\Theta$ is determined via (\ref{dT}) in terms of $Q^{sat}$, hence $\sup|\Theta|$ satisfies a universal bound. 
In any case, we have $\max_j|\theta^n_j(k\delta t)|\leq \max(\max_j|\theta^n_j((k-1)\delta t)|,\sup|\Theta|)$. 
It then follows easily by induction that $||\theta^n||_{L^{\infty}}\leq\max(M',\sup|\Theta|)$.
The bound for $q$ then follows automatically.
\end{proof}
Next we prove the point (ii). Roughly speaking, point (ii) says if a parcel is "strictly" dry, then it will remain dry and go down for a while, the length of time this state lasts depends in a universal way how dry this parcel is. 
\begin{lem}\label{3.2l}
There is a universal constant $\tilde{C}_2>0$, such that for any $\eps>0$, 
if for some interger $k,j$, it holds $\hat{\theta}^n_j(k\delta t)>\Theta(\theta^{M,n}_j,z_{\alpha_{k\delta t}(j)},k\delta t)+\eps$, we have $\hat{\theta}^n_j(l\delta t)=\hat{\theta}^n_j(k\delta t)$, and $\alpha_{l\delta t}(j)\leq\alpha_{k\delta t}(j)$ for any $l$ with $0\leq (l-k)\delta t\leq\frac{\eps}{\tilde{C}_2}$.
\end{lem}
\begin{proof}
Actually we will see one can take $\tilde{C}_2=2\sup|\partial_t\Theta|$. 
We prove this by induction on $l$. First observe that the statement is trivial if $l=k$. Assume this is true for some $l$ with $(l+1-k)\delta t\leq\frac{\eps}{2\sup|\partial_t\Theta|}$. Need to show this is true also for $l+1$. Using the induction hypothesis, we can calculate
\begin{equation*}
\begin{split}
\hat{\theta}_j^n(l\delta t)=\hat{\theta}_j^n(k\delta t)&>\Theta(\theta_j^{M,n},z_{\alpha_{k\delta t}(j)},k\delta t)+\eps\\
&\geq\Theta(\theta_j^{M,n},z_{\alpha_{l\delta t}(j)},(l+1)\delta t)+\eps-\sup|\partial_t\Theta|(l+1-k)\delta t\\
&\geq\Theta(\theta_j^{M,n},z_{\alpha_{l\delta t}(j)},(l+1)\delta t).
\end{split}
\end{equation*}
The first equality is the induction hypothesis. In the second inequality, we used the induction hypothesis that $\alpha_{l\delta t}(j)\leq\alpha_{k\delta t}(j)$.

Above calculation shows that at time $l\delta t$, the parcel $\alpha_{l\delta t}(j)$ is still "dry" by taking one more time step forward. Hence we know $\alpha_{(l+1)\delta t}=\tilde{\alpha}_{(l+1)\delta t}(\alpha_{l\delta t}(j))\leq\alpha_{l\delta t}(j)$, and $\hat{\theta}^n_j((l+1)\delta t)=\hat{\theta}^n_j(l\delta t)$, from the procedure.
\end{proof}

Now we can deduce point (ii) as a corollary of previous lemma.
\begin{cor}\label{3.3c}
Let $\eps>0$, $n\delta t\leq 1$. Then there exists a universal constant $C_2>0$, such that if $\eps>\frac{C_2}{n}$ and $\hat{\theta}^n(t,z)>\Theta(\theta^{M,n}(z),F^n_t(z),t)+\eps$ for some $t\in[0,T)$, 
we have $\hat{\theta}^n_{t^{\prime}}(z)=\hat{\theta}^n_t(z)$, and $F^n_{t^{\prime}}(z)\leq F^n_t(z)$ holds for any $t^{\prime}-t<\frac{\eps}{C_2}$.
\end{cor}
\begin{proof}
First we can find integers $k,j$, such that $k\delta t\leq t<(k+1)\delta t$, and $z\in J_j$. It then follows from the definition of $\hat{\theta}^n$ that
\begin{equation*}
\begin{split}
\hat{\theta}^n(t,z)&=\hat{\theta}^n_j(k\delta t)>\Theta(\theta^{M,n}(t,z),F_t^n(z),t)+\eps\\
&\geq\Theta(\theta^{M,n}_j,z_{\alpha_{k\delta t}(j)},k\delta t)+\eps-\sup|\partial_z\Theta|n^{-1}-\sup|\partial_t\Theta|\delta t.
\end{split}
\end{equation*}
Since $\delta t\leq n^{-1}$, 
we will have $\eps-\sup|\partial_z\Theta|n^{-1}-\sup|\partial_t\Theta|\delta t>\frac{\eps}{2}$, if $\eps>\frac{C_2^{\prime}}{n}$ for some universal constant $n$. With such a choice of $\eps$, we then have 
$$
\hat{\theta}^n_j(k\delta t)>\Theta(\theta_j^{M,n},z_{\alpha_{k\delta t}(j)},k\delta t)+\frac{\eps}{2}.
$$
Now we can conclude from Lemma \ref{3.2l} that $\hat{\theta}^n_j(l\delta t)=\hat{\theta}^n_j(k\delta t)$, 
and $\alpha_{l\delta t}(j)\leq\alpha_{k\delta t}(j)$ for any integer $l$ with $0\leq(l-k)\delta t\leq\frac{\eps}{C_2^{\prime\prime}}$ for some universal constant $C_2^{\prime\prime}$.
This means precisely that $\hat{\theta}^n(t^{\prime},z)=\hat{\theta}^n(t,z)$  and $F_{t^{\prime}}^n(z)\leq F_t^n(z)$ for any $t^{\prime}-t<\frac{\eps}{C_2^{\prime\prime}}$.
\end{proof}
Next we wish to prove point (iii). For this, we need to establish a lemma which gives control over the total variation of $t\longmapsto F_t^n(z)$ in terms of the absolute bound of $\theta$.
\begin{lem}
There exists a universal constant $C_3^{\prime}>0$, such that for any indices $j\in\{1,2\cdots,n\}$,
\begin{equation*}
\begin{split}
&\sum_{0\leq k\leq\frac{T}{\delta t}}\frac{1}{n}(\alpha_{(k+1)\delta t}(j)-\alpha_{k\delta t}(j))^+\leq C_3^{\prime}||\hat{\theta}^n||_{L^{\infty}((0,T)\times(0,1))},\\
&\sum_{0\leq k\leq\frac{T}{\delta t}}\frac{1}{n}|\alpha_{(k+1)\delta t}(j)-\alpha_{k\delta t}(j)|\leq 2C_3^{\prime}(||\hat{\theta}^n||_{L^{\infty}((0,T)\times(0,1))}+2.
\end{split}
\end{equation*}
\end{lem}
\begin{proof}
First we observe that the second estimate follows from the first. Indeed, we just need to notice
\begin{equation*}
\begin{split}
&\sum_{0\leq k\leq\frac{T}{\delta t}}\frac{1}{n}|\alpha_{(k+1)\delta t}(j)-\alpha_{k\delta t}(j)|\\
&=\sum_{0\leq k\leq\frac{T}{\delta t}}\frac{2}{n}(\alpha_{(k+1)\delta t}(j)-\alpha_{k\delta t}(j))^+-\sum_{0\leq k\leq\frac{T}{\delta t}}\frac{1}{n}(\alpha_{(k+1)\delta t}(j)-\alpha_{k\delta t}(j))\\
&\leq 2C_3^{\prime}||\hat{\theta}^n||_{L^{\infty}((0,T)\times(0,1))}+2.
\end{split}
\end{equation*}
Now we only need to focus on the first estimate. 
Fix some $k$ such that $\alpha_{(k+1)\delta t}(j)>\alpha_{k\delta t}(j)$. Then we know $\alpha_{(k+1)\delta t}(j)=\tilde{\alpha}_{(k+1)\delta t}(\alpha_{k\delta t}(j))>\alpha_{k\delta t}(j)$. This means the parcel $\alpha_{k\delta t}(j)$ is "wet" at $k\delta t$,
 or $\hat{\theta}_j^n(k\delta t)<\Theta(\theta^{M,n}_j,z_{\alpha_{k\delta t}(j)},(k+1)\delta t)$. 
After the time step, we know $\hat{\theta}^n_j((k+1)\delta t)=\Theta(\theta_j^{M,n},z_{\alpha_{(k+1)\delta t}(j)},(k+1)\delta t)$. Therefore
\begin{equation}\label{3.1}
\begin{split}
\hat{\theta}_j^n((k+1)\delta t)-&\hat{\theta}_j^n(k\delta t)\geq\Theta(\theta_j^{M,n},z_{\alpha_{(k+1)\delta t}(j)},(k+1)\delta t)-\Theta(\theta_j^{M,n},z_{\alpha_{k\delta t}(j)},(k+1)\delta t)\\
&\geq(\inf\partial_z\Theta)(z_{\alpha_{(k+1)\delta t}(j)}-z_{\alpha_{k\delta t}(j)})^+=\frac{1}{n}(\inf\partial_z\Theta)(\alpha_{(k+1)\delta t}(j)-\alpha_{k\delta t}(j))^+.
\end{split}
\end{equation}
Now we sum (\ref{3.1}) over $k$, the first estimate follows.
\end{proof}
Then we can deduce point (iii) as a corollary
\begin{lem}
For any $z\in[0,1)$
\begin{equation*}
\begin{split}
&PV_{t\in[0,T)}(F_t^n(z))\leq C_3^{\prime}||\hat{\theta}^n||_{L^{\infty}((0,T)\times(0,1))}.\\
&TV_{t\in[0,T)}(F_t^n(z))\leq 2C_3^{\prime}||\hat{\theta}^n||_{L^{\infty}((0,T)\times(0,1))}+2.
\end{split}
\end{equation*}
Here $PV$ denotes the positve variation, and $TV$ denotes the total variation. $C_3^{\prime}$ is the same constant as in previous lemma.
\end{lem}
\begin{proof}
Find indices $j\in\{1,2\cdots,n\}$ such that $z\in J_j$, then 
\begin{equation*}
\begin{split}
&PV_{t\in[0,T)}(F_t^n(z))=\sum_{0\leq k\leq\frac{T}{\delta t}}(z_{\alpha_{(k+1)\delta t}(j)}-z_{\alpha_{k\delta t}(j)})^+=\sum_{0\leq k\leq\frac{T}{\delta t}}\frac{1}{n}(\alpha_{(k+1)\delta t}(j)-\alpha_{k\delta t}(j))^+\\
&TV_{t\in[0,T)}(F_t^n(z))=\sum_{0\leq k\leq\frac{T}{\delta t}}|z_{\alpha_{(k+1)\delta t}(j)}-z_{\alpha_{k\delta t}(j)}|=\sum_{0\leq k\leq\frac{T}{\delta t}}\frac{1}{n}|\alpha_{(k+1)\delta t}(j)-\alpha_{k\delta t}(j)|.
\end{split}
\end{equation*}
\end{proof}
It only remains to prove the point (iv) and (v). For this we need the following key lemma, which concludes that any given parcel can only be overtaken at a finite rate.
\begin{lem}
Fix $j_0\in\{1,2,\cdots,n\}$. Let $0\leq k<l\leq\frac{T}{\delta t}$. Define the set
$$
J_{k,l}=\{j\in\{1,2,\cdots,n\}:\alpha_{k\delta t}(j)<\alpha_{k\delta t}(j_0),\,\alpha_{l\delta t}(j)>\alpha_{l\delta t}(j_0)\}.
$$
Then there exists a universal constant $C_4^{\prime}>0$, such that if $n\delta t<\frac{1}{C_4^{\prime}}$, we have 
$$
\#J\leq 2(l-k).
$$
\end{lem}
\begin{proof}
We will prove this statement with the choice of constant $C_4^{\prime}=\frac{\sup|\partial_t\Theta|}{\inf\partial_z\Theta}$. Here $\sup$, $\inf$ is taken over the set $\{(w,z,t):|w|\leq M,z\in[0,1],t\in[0,T]\}$.
With this choice, Lemma \ref{3.4nl} applies.
For $k\leq s\leq l-1$, we may define
\begin{equation*}
\begin{split}
A_s=\{j\in\{1,2,\cdots,n\}:\alpha_{s^{\prime}\delta t}(j)&<\alpha_{s^{\prime}\delta t}(j_0),\textrm{ for any $s^{\prime}$ with $k\leq s^{\prime}\leq s$, and}\\
&\alpha_{(s+1)\delta t}(j)>\alpha_{(s+1)\delta t}(j_0)\}.
\end{split}
\end{equation*}
Then we have $J_{k,l}=\cup_{s=k}^{l-1}A_s$. That $J_{k,l}\subset\cup_{s=k}^{l-1}A_s$ is clear and the reverse inclusion follows from Lemma \ref{2.3l}. Therefore it suffices to show $\#A_s\leq 2$ for each $s$, 
when $n\delta t$ is small. 
Here we use the notation of section 2 and write $\tilde{\alpha}_{(s+1)\delta t}=\sigma_1\circ\cdots\circ\sigma_n$.
 Here $\sigma_k$ is the rearrangement map of the indices when we decide which parcel moves to $z_k$. 
Denote $\beta_k=\sigma_{k+1}\circ\cdots\circ\sigma_n$.
Without loss of generality,
 we may assume that between time step $s\delta t$ and $(s+1)\delta t$, the second alternative of Lemma \ref{3.4nl} holds for $\alpha_{s\delta t}(j_0)$. 
Indeed, if the first alternative holds, we will have $A_s=\emptyset$, and there is nothing to prove in this case.

If $\#A_s\leq 2$ fails, 
then let $k_1>k_2>k_3$ be the 3 biggest integers for which there exists some index $j$ such that $\alpha_{s\delta t}(j)<\alpha_{s\delta t}(j_0)$ and $\beta_{k-1}(\alpha_{s\delta t}(j))>\beta_{k-1}(\alpha_{s\delta t}(j_0))$.
Let $j_1,\,j_2,\,j_3$ be the index corresponding to $k_1,\,k_2,\,k_3$. Namely $\alpha_{s\delta t}(j_i)<\alpha_{s\delta t}(j_0)$, but $\beta_{k_i-1}(\alpha_{s\delta t}(j_i))>\beta_{k_i-1}(\alpha_{s\delta t}(j_0))$, $i=1,2,3$.
\sloppy It is clear that $\beta_{k_i}(\alpha_{s\delta t}(j_i))<\beta_{k_i}(\alpha_{s\delta t}(\alpha_{s\delta t}(j_0))$, $1\leq i\leq 3$. It is also clear that such index $j_i$ must be unique since each $\sigma_k$ lifts at most one index. 
From Lemma \ref{3.4nl}, we can see it must hold $\alpha_{s\delta t}(j_3)<\alpha_{s\delta t}(j_2)<\alpha_{s\delta t}(j_1)<\alpha_{s\delta t}(j_0)$.
The last inequality is clear,
 since $\alpha_{s\delta t}(j_0)$ will not jump.
If, say $\alpha_{s\delta t}(j_2)>\alpha_{s\delta t}(j_1)$,
 then $\alpha_{s\delta t}(j_1)$ overtakes $\alpha_{s\delta t}(j_2)$ since $\alpha_{s\delta t}(j_1)$ jumps first(under $\sigma_{k_1}$) and ends up above $\alpha_{s\delta t}(j_0)$. Thus $\alpha_{s\delta t}(j_2)$ will get pushed down by $\sigma_{k_1}$. 
But then according to Lemma \ref{3.4nl}, it cannot jump later on and cannot overtake any other parcel.

The plan is to show $\theta^{M,n}_{j_3}
\leq\theta^{M,n}_{j_2}
\leq\theta^{M,n}_{j_1}$,
 and also show this implies $\beta_{k_3}(\alpha_{s\delta t}(j_3))\notin W_{k_3}^{\prime}$. This will be a contradiction since we will have $\beta_{k_3}(\alpha_{s\delta t}(j_3))$ cannot jump under $\sigma_{k_3}$.

First we observe that for any $k$ with $k_2<k\leq k_1$, 
and any index $j$ with $\beta_{k_1}(\alpha_{s\delta t}(j_1))<\beta_{k_1}(\alpha_{s\delta t}(j))\leq k_1$, it holds $\beta_{k-1}(\alpha_{s\delta t}(j))=\beta_{k_1}(\alpha_{s\delta t}(j))-1$,
while for index $j$ with $\beta_{k_1}(\alpha_{s\delta t}(j))<\beta_{k_1}(\alpha_{s\delta t}(j_1))$, it must hold $\beta_{k-1}(\alpha_{s\delta t}(j))=\beta_{k_1}(\alpha_{s\delta t}(j))$.
Observe it is clear with $k=k_1$. For $k<k_1$, and $j$ satisfying $\beta_{k_1}(\alpha_{s\delta t}(j_1))<\beta_{k_1}(\alpha_{s\delta t}(j))\leq k_1$, they cannot jump under $\sigma_k$ since they are already overtaken by $\alpha_{s\delta t}(j_1)$.  
For $j$ satisfying $\beta_{k_1}(\alpha_{s\delta t}(j))<\beta_{k_1}(\alpha_{s\delta t}(j_1))$,
 they cannot jump up because once they jump up under $\sigma_k$, 
they will jump to $z_k$, hence overtake $\alpha_{s\delta t}(j_0)$, contradicting our choice of $k_2$. They also cannot be pushed down since if this happens, some parcel below needs to jump, again contradicting the choice of $k_2$.

Now we wish to prove $\theta_{j_2}^{M,n}
\leq\theta_{j_1}^{M,n}$.
If not, we will show below that $\beta_{k_1}(\alpha_{s\delta t}(j_2))\in W_{k_1}^{\prime}$.
 This will give us a contradiction since the parcel $\beta_{k_1}(\alpha_{s\delta t}(j_1))$ does not have the largest $\theta^M$ among the parcels in the set $W_{k_1}^{\prime}$,
 hence cannot jump up under $\sigma_{k_1}$.  First it is clear from the "moreover" part of Lemma \ref{3.1nl} that $\beta_{k_1}(\alpha_{s\delta t}(j_2))\in W_{k_1}$.
For any index $j$ with $\beta_{k_1}(\alpha_{s\delta t}(j_2))<\beta_{k_1}(\alpha_{s\delta t}(j))<\beta_{k_1}(\alpha_{s\delta t}(j_1))$,
 and any $k$ with $k_2<k\leq k_1$, we conclude from last paragraph that $\sigma_k=id$, 
and $\theta^{n,k-1}_{\beta_{k-1}(\alpha_{s\delta t}(j))}=\theta^{n,k_1}_{\beta_{k_1}(\alpha_{s\delta t}(j))}$. Since $\beta_{k_2}(\alpha_{s\delta t}(j_2))\in W_{k_2}^{\prime}$, 
and no changes happen for these $j$`s under $\sigma_k$ with $k_2<k\leq k_1$, we see $\beta_{k_1}(\alpha_{s\delta t}(j_2))$ will beat them under rearrangement $\sigma_{k_1}$. Since we assumed $\theta^{M,n}_{j_2}>\theta^{M,n}_{j_1}$, we know $\alpha_{s\delta t}(j_2)$ beats $\alpha_{s\delta t}(j_1)$ as well.
Now consider index $j$ satisfying $\beta_{k_1}(\alpha_{s\delta t}(j_1))<\beta_{k_1}(\alpha_{s\delta t}(j))\leq k_1$, if $\beta_{k_1}(\alpha_{s\delta t}(j))\in W_{k_1}$,
it can be beaten by $\beta_{k_1}(\alpha_{s\delta t}(j_1))$, which means $\theta^{M,n}_j<\theta^{M,n}_{j_1}$. 
 Since $\theta^{M,n}_{j_2}>\theta^{M,n}_{j_1}$, 
it can also be beaten by $\alpha_{s\delta t}(j_2)$.
If $\beta_{k_1}(\alpha_{s\delta t}(j))\notin W_{k_1}$, we know from Lemma \ref{3.1nl} that $\beta_{k_2}(\alpha_{s\delta t}(j))\notin W_{k_2}$. Hence from $\beta_{k_2}(\alpha_{s\delta t}(j_2))\in W_{k_2}^{\prime}$, we see $\theta^{n,k_2}_{\beta_{k_2}(\alpha_{s\delta t}(j))}<\Theta
(\theta_{j_2}^{M,n},z_{\beta_{k_2}(\alpha_{s\delta t}(j))},(s+1)\delta t).$
Since they are "dry" parcels, we know their $\theta$ does not change, namely $\theta^{n,k_2}_{\beta_{k_2}(\alpha_{s\delta t}(j))}=\theta^{n,k_1}_{\beta_{k_1}(\alpha_{s\delta t}(j))}$, 
also we know from the observation made in previous paragraph with $k=k_2+1$ that $\beta_{k_2}(\alpha_{s\delta t}(j))=\beta_{k_1}(\alpha_{s\delta t}(j)-1$. 
Hence they will be beaten by $\alpha_{s\delta t}(j_2)$ in the rearrangement $\sigma_{k_1}$.
 This shows $\beta_{k_2}(\alpha_{s\delta t}(j_2))\in W_{k_1}^{\prime}$.

By the same argument as above, one can also conclude $\theta^{M,n}_{j_3}
\leq\theta^{M,n}_{j_2}$, following the same logic that if $\theta_{j_3}^{M,n}>\theta^{M,n}_{j_2}$, we will then conclude $\beta_{k_2}(\alpha_{s\delta t}(j_3))\in W_{k_2}^{\prime}$, thus $\alpha_{s\delta t}(j_2)$ will not jump under $\sigma_{k_2}$. So we have shown $\theta^{M,n}_{j_3}
\leq\theta^{M,n}_{j_2}
\leq\theta^{M,n}_{j_1}$.

Next we show $\beta_{k_3}(\alpha_{s\delta t}(j_3))\notin W_{k_3}^{\prime}$. Let $j_4$ be the index such that $\beta_{k_1}(\alpha_{s\delta t}(j_4))=\beta_{k_1}(\alpha_{s\delta t}(j_1))+1$. Since it is overtaken by $j_1$ under $\sigma_{k_1}$ 
it will remain "dry" for all later rearrangements, that is, $\beta_m(\alpha_{s\delta t}(j_4))\notin W_m$, 
for any $1\leq m< k_1$. If $\beta_{k_3}(\alpha_{s\delta t}(j_3))\in W_{k_3}^{\prime}$, 
in particular, one should have $\Theta(\theta_{j_3}^{M,n},z_{\beta_{k_3}(\alpha_{s\delta t}(j_4))},(s+1)\delta t)>\theta_{\beta_{k_3}(\alpha_{s\delta t}(j_4))}^{n,k_3}=\theta_{\beta_{k_1}(\alpha_{s\delta t}(j_4))}^{n,k_1}$. On the other hand, using Corollary \ref{2.2c}, point (i), (iv), we see
\begin{equation}\label{3-2nn}
\theta_{\beta_{k_1}(\alpha_{s\delta t}(j_4))}^{n,k_1}\geq\theta^{n,k_1}_{\beta_{k_1}(\alpha_{s\delta t}(j_1))}\geq\Theta(\theta_{j_1}^{M,n},
z_{\beta_{k_1}(\alpha_{s\delta t}(j_1))},s\delta t).
\end{equation}
Observe that $\beta_{k_3}(\alpha_{s\delta t}(j_4))\leq\beta_{k_1}(\alpha_{s\delta t}(j_4))-2=\beta_{k_1}(\alpha_{s\delta t}(j_1))-1$, since there are at least 2 parcels($j_1$ and $j_2$) overtaking $j_4$. Therefore
\begin{equation}\label{3-3}
\begin{split}
\theta_{\beta_{k_1}(\alpha_{s\delta t}(j_4))}^{n,k_1}<\Theta(\theta_{j_3}^{M,n},&z_{\beta_{k_3}(\alpha_{s\delta t}(j_4))},(s+1)\delta t)\leq\Theta(\theta^{M,n}_{j_3},z_{\beta_{k_1}(\alpha_{s\delta t}(j_1))-1},(s+1)\delta t)\\
&\leq\Theta(\theta_{j_1}^{M,n},z_{\beta_{k_1}({\alpha_{s\delta t}(j_1)})},s\delta t)-(\inf\partial_z\Theta)n^{-1}+(\sup|\partial_t\Theta|)\delta t.
\end{split}
\end{equation}
In the second inequality above, we used above observation, and in the third inequality we used that $\theta^{M,n}_{j_3}\leq
\theta^{M,n}_{j_2}\leq\theta^{M,n}_{j_1}$. Now combining (\ref{3-2nn}) and (\ref{3-3}) gives a contradiction.
\end{proof}

From above discrete estimate, we can get the lemma of "finite speed of penetration".
\begin{lem}\label{3.7l}
Let $z_{i_0}\in[0,1)$. Choose $n,\,\delta t$ such that $n\delta t=\frac{1}{2C_4^{\prime}}$, where $C_4^{\prime}$ is the constant given by previous lemma. Let $t_0\in[0,T)$, $\eps>0$, define the set
$$
J=\{z\in[0,1):F_{t_0}(z)<F_{t_0}(z_{i_0}), F_{t_0+\eps}(z)>F_{t_0+\eps}(z_{i_0})\}.
$$
Then $\mathcal{L}^1(J)\leq C(\eps+\delta t)$, for some universal constant $C$.
\end{lem}
\begin{proof}
Choose $j_0\in\{1,2,\cdots,n\}$, such that $z_{i_0}\in J_{j_0}$. 
Choose integers $k,l$ such that $k\delta t\leq t_0<(k+1)\delta t$, $l\delta t\leq t_0+\eps<(l+1)\delta t$. Then we know that $J=\cup_{j\in J_{k,l}}J_j$. Here $J_{k,l}$ is defined as in previous lemma. Hence 
$$
\mathcal{L}^1(J)=n^{-1}\#J_{k,l}\leq\frac{2(l-k)}{n}=\frac{2(l-k)\delta t}{n\delta t}\leq4C_4^{\prime}(\eps+\delta t).
$$
\end{proof}
As an application of this lemma, we can prove point (iv) of Theorem \ref{3.1t}.

\begin{prop}
Let $n,\,\delta t$ be chosen as in Lemma \ref{3.7l}. There exists a universal constant $C>0$, such that for any $\eps>0$, and $t\in[0,T)$, with $[t-\eps,t+\eps]\subset[0,T)$, we have
\begin{equation}
|\hat{\theta}^n_{t+\eps}(z)-\hat{\theta}^n_{t-\eps}(z)-\big(\Theta(\theta^{M,n}(z),F^n_{t+\eps}(z),t+\eps)-\Theta(\theta^{M,n}(z),F^n_{t-\eps}(z),t-\eps)\big)^+|\leq C(\eps+\delta t).
\end{equation}
\end{prop}
\begin{proof}
Let $\kappa$ be the quantity in the absolute value above. First we show $\kappa\leq C(\eps+\delta t)$. Without loss of generality, we can then assume $\kappa>0$. Then we know that $\hat{\theta}^n_{t+\eps}(z)-\hat{\theta}^n_{t-\eps}(z)\geq\kappa$.
It follows from Corollary \ref{3.3c} that 
\begin{equation}\label{2.7}
\hat{\theta}^n_{t-\eps}(z)\leq\Theta(\theta^{M,n}(z),F^n_{t-\eps}(z),t-\eps)+C\eps.
\end{equation}
On the other hand
\begin{equation}\label{2.8}
\begin{split}
\hat{\theta}^n_{t+\eps}(z)-\hat{\theta}^n_{t-\eps}(z)&\geq\Theta(\theta^{M,n}(z),F^n_{t+\eps}(z),t+\eps)-\Theta(\theta^{M,n}(z),F^n_{t-\eps}(z),t-\eps)+\kappa\\
&\geq\Theta(\theta^{M,n}(z),F^n_{t+\eps}(z),t+\eps)-\hat{\theta}^n_{t-\eps}(z)-\sup|\partial_t\Theta|\delta t+\kappa.
\end{split}
\end{equation}
In the first inequality, we only used the definition of $\kappa$, and in the second inequality, we used Corollary \ref{2.2c}, point (iv).
If we let $t^{\prime}\in[t-\eps,t+\eps]$ be such that $F^n_{t^{\prime}}(z)=\max_{\tau\in[t-\eps,t+\eps]}F^n_{\tau}(z)$, 
we have $\hat{\theta}^n_{t+\eps}(z)\leq\Theta(\theta^{M,n}(z),F^n_{t^{\prime}}(z),t+\eps)+\sup|\partial_t\Theta|\eps.$
Hence it follows from (\ref{2.8}) that
\begin{equation}
\Theta(\theta^{M,n}(z),F_{t^{\prime}}^n(z),t+\eps)+\sup|\partial_t\Theta|\eps\geq\Theta(\theta^{M,n}(z),F_{t+\eps}^n(z),t+\eps)-\sup|\partial_t\Theta|\delta t+\kappa.
\end{equation}
Noticing $\delta t\leq\eps$, we obtain
$$
F_{t^{\prime}}^n(z)-F^n_{t+\eps}(z)\geq\frac{1}{\inf\partial_z\Theta}(\kappa-2\sup|\partial_t\Theta|\eps).
$$
Now consider the set $E^{\prime}=\{z^{\prime}\in[0,1]:F^n_{t^{\prime}}(z^{\prime})<F^n_{t^{\prime}}(z), F^n_{t+\eps}(z^{\prime})>F^n_{t+\eps}(z)\}$. Then we know $\mathcal{L}^1(E^{\prime})\geq\frac{1}{\inf\partial_z\Theta}(\kappa-2\sup|\partial_t\Theta|\eps)$. But it follows from Corollary \ref{3.7l} that $\mathcal{L}^1(E^{\prime})\leq C(\eps+\delta t)$. Hence $\kappa\leq C^{\prime}(\eps+\delta t)$, for some universal constant $C^{\prime}$.

Next we derive a lower bound. We consider two cases. First if $\hat{\theta}^n_{t+\eps}(z)>\hat{\theta}^n_{t-\eps}(z)$, then as has been observed in (\ref{2.7}), we know that 
\begin{equation}\label{3-8}
\hat{\theta}^n_{t+\eps}(z)-\hat{\theta}^n_{t-\eps}(z)\geq\Theta(\theta^{M,n}(z),F^n_{t+\eps}(z),t+\eps)-\Theta(\theta^{M,n}(z),F^n_{t-\eps}(z),t-\eps)-C\eps.
\end{equation}
If $\Theta(\theta^{M,n}(z),F^n_{t+\eps}(z),t+\eps)-\Theta(\theta^{M,n}(z),F^n_{t-\eps}(z),t-\eps)\geq0$, then we can immediately conclude $\kappa\geq-C\eps$. If it is negative, then we can calculate:
\begin{equation}\label{3-9}
\begin{split}
&\Theta(\theta^{M,n}(z),F^n_{t+\eps}(z),t+\eps)-\Theta(\theta^{M,n}(z),F^n_{t-\eps}(z),t-\eps)\\
&\geq-\sup|\partial_z\Theta|(F_{t+\eps}^n(z)-F_{t-\eps}^n(z))^--\sup|\partial_t\Theta|2\eps.
\end{split}
\end{equation}
Define $\tilde{E}=\{z^{\prime}\in[0,1):F^n_{t-\eps}(z^{\prime})<F^n_{t-\eps}(z),\,F^n_{t+\eps}(z^{\prime})>F^n_{t+\eps}(z)\}$. 
Then Lemma \ref{3.7l} shows $\mathcal{L}^1(\tilde{E})\leq C(\eps+\delta t)$. 
But this means $F_{t+\eps}^n(z)-F_{t-\eps}^n(z)\geq-C(\eps+\delta t)$. Thus from (\ref{3-8}), (\ref{3-9}), 
we know $\hat{\theta}^n_{t+\eps}(z)-\hat{\theta}^n_{t-\eps}(z)\geq-C(\eps+\delta t)$. The conclusion follows as well.

If $\hat{\theta}^n_{t+\eps}(z)=\hat{\theta}^n_{t-\eps}(z)$, this means no jumps happen during $[t-\eps,t+\eps]$. Therefore $F^n_{t+\eps}(z)\leq F^n_{t-\eps}(z)$. In this case
$$
\Theta(\theta^{M,n}(z),F^n_{t+\eps}(z),t+\eps)\leq\Theta(\theta^{M,n}(z),F^n_{t-\eps}(z),t-\eps)+C\eps.
$$
Therefore, we also have the quantity $\geq-C\eps$ as well.
\end{proof}

As a second application of Lemma \ref{3.7l}, we finally prove point (v) of the Theorem.
First we derive an obvious corollary of above lemma.

Fix some $t_0\in[0,T)$, take $\eps>\delta t$, $\kappa>0$, and define 
\begin{equation*}
\begin{split}
&J_1=\{z\in[0,1):\sup_{t\in[t_0,t_0+\eps]}F^n_t(z)-F^n_{t_0}(z)\geq\kappa\},\\
&J_2=\{z\in[0,1):F^n_{t_0+\eps}(z)-F^n_{t_0}(z)\geq\frac{\kappa}{2}\}.
\end{split}
\end{equation*}
First we observe the following lemma:
\begin{lem}\label{3.9l}
Let $n,\,\delta t$ be chosen as in Lemma \ref{3.7l}. Then there exists a universal constant $C>0$, such that if $\eps+\delta t\leq\frac{\kappa}{C}$, then $J_1\subset J_2$.
\end{lem}
\begin{proof}
Suppose there exists $z_0\in J_1-J_2$. Let $t^{\prime}\in[t_0,t_0+\eps]$ be such that $F^n_{t^{\prime}}(z_0)=\max_{t\in[t_0,t_0+\eps]}F^n_t(z_0)$. Then we know $F^n_{t^{\prime}}(z_0)-F^n_{t_0}(z)\geq\kappa$.
 It follows that $F^n_{t^{\prime}}(z_0)-F^n_{t_0+\eps}(z_0)\geq\frac{\kappa}{2}$. Consider the set
$$
\hat{J}_{z_0,\eps,\kappa}=\{z\in[0,1]: F^n_{t^{\prime}}(z)<F^n_{t^{\prime}}(z_0); F^n_{t_0+\eps}(z)>F^n_{t_0+\eps}(z_0)\}.
$$
Then we know that $\mathcal{L}^1(\hat{J}_{z_0,\eps,\kappa})\geq\frac{\kappa}{2}$. On the other hand, 
it follows from Lemma \ref{3.7l} that $\mathcal{L}^1(\hat{J}_{z_0,\eps,\kappa})\leq C(\eps+\delta t)$ with a universal constant $C$. Hence we have $\frac{\kappa}{2}\leq C(\eps+\delta t)$.
\end{proof}
The next lemma estimates $\mathcal{L}^1(J_2)$.
\begin{lem}\label{3.10l}
Let $\eps,\,n,\,\delta t$ be chosen as in previous lemma.
Let $\kappa\geq C\eps$, where $C$ is the constant given by previous lemma,
 then for some universal constant $C^{\prime}$, we have $\mathcal{L}^1(J_2)\leq\frac{C^{\prime}\eps}{\kappa}$.
\end{lem}
\begin{proof}
Let $\delta=\mathcal{L}^1(J_2)$. Since $F^n$ measure preserving, we know
$$
\int_{[0,1)}F^n_{t_0+\eps}(z)-F^n_{t_0}(z)dz=0.
$$
On the other hand from the definition of $J_2$,
$$
\int_{J_2}F^n_{t_0+\eps}(z)-F^n_{t_0}(z)dz\geq\frac{\delta\kappa}{2}.
$$
Therefore, there exists $z_1\in[0,1]$, such that $F^n_{t_0+\eps}(z_1)-F^n_{t_0}(z_1)\leq-\frac{\delta\kappa}{2}$. Now consider
$$
\hat{J}_{z_1,\eps}=\{z\in[0,1]:F^n_{t_0}(z)<F^n_{t_0}(z_1), F^n_{t_0+\eps}(z)>F^n_{t_0+\eps}(z_1)\}.
$$
Then we have $\mathcal{L}^1(\hat{J}_{z_1,\eps})\geq\frac{\delta\kappa}{2}$. 
But by Lemma \ref{3.7l}, we know $\mathcal{L}^1(\hat{J}_{z_1,\eps})\leq C(\eps+\delta t)\leq 3C\eps$.
Hence $\frac{\delta\kappa}{2}\leq 3C\eps$. This completes the proof.

\end{proof}
With above preparation, we can obtain the following continuity estimate
\begin{lem}
Let $n,\,\delta t$ be chosen as in Lemma \ref{3.7l}. Let $\eps>\frac{\delta t}{2}$. Then there exists a universal constant $C$ such that if $\eps<\frac{1}{C}$, we have
$$
||\hat{\theta}^n_{t_0+\eps}-\hat{\theta}^n_{t_0}||_{L^1(0,1)}=||\theta^n_{t_0+\eps}-\theta^n_{t_0}||_{L^1(0,1)}\leq C^{\prime}\sqrt{\eps}.
$$
for some universal contant $C^{\prime}$.
\end{lem}
\begin{proof}
From the discrete procedure, we know if $\hat{\theta}^n_{t_0+\eps}(z)>\hat{\theta}^n_{t_0}(z)$,
$$
\hat{\theta}^n_{t_0+\eps}(z)\leq\max_{t^{\prime}\in[t_0,t_0+\eps]}
\Theta(\theta_0^{M,n}(z),F^n_{t^{\prime}}(z),t_0+\eps)+\sup|\partial_t\Theta|\eps.
$$
Since $\hat{\theta}^n_{t_0}(z)\geq\Theta(\theta^{M,n}(z),F^n_{t_0}(z),t_0)-\sup|\partial_t\Theta|\delta t$, by Corollary \ref{2.2c} point (iv), we know
$$
\hat{\theta}^n_{t_0+\eps}(z)-\hat{\theta}^n_{t_0}(z)\leq\sup\partial_z\Theta\cdot\max_{t\in[t_0,t_0+\eps]}(F^n_t(z)-F^n_{t_0}(z))+\sup|\partial_t\Theta|\cdot(\eps+\delta t).
$$
Let $\kappa\geq C\eps$, where $C$ is the universal constant given by Lemma \ref{3.9l}. Combining Lemma \ref{3.9l},\ref{3.10l}, we conclude $\mathcal{L}^1(J_1)
\leq\mathcal{L}^1(J_2)
\leq\frac{C^{\prime}\eps}{\kappa}$.
Hence
\begin{equation*}
\begin{split}
\int_{[0,1]}\hat{\theta}_{t_0+\eps}(z)-\hat{\theta}_{t_0}(z)dz&=\int_{J_1}\hat{\theta}_{t_0+\eps}(z)-\hat{\theta}_{t_0}(z)dz+\int_{J_1^c}\hat{\theta}_{t_0+\eps}(z)-\hat{\theta}_{t_0}(z)dz\\
&\leq ||\hat{\theta}||_{L^{\infty}}\cdot\frac{C^{\prime}\eps}{\kappa}+\sup|\partial_z\Theta|\kappa+3\sup|\partial_t\Theta|\cdot\eps.
\end{split}
\end{equation*}
Now we take $\eps$ small enough such that $\sqrt{\eps}\leq\frac{1}{C}$, where $C$ is given by Lemma \ref{3.9l} and let $\kappa=\sqrt{\eps}$, we obtain the continuity estimates:
$$
\int_{[0,1]}|\theta_{t_0+\eps}(z)-\theta_{t_0}(z)|dz=\int_{[0,1]}\hat{\theta}_{t_0+\eps}(z)-\hat{\theta}_{t_0}(z)dz\leq C^{\prime}\sqrt{\eps}.
$$
\end{proof}
Next we derive the continuity estimate for the flow maps, which follows from the continuity estimate for $\theta$
\begin{lem}
Let $s,\,t\in[0,T)$, then for some universal constant $C$
\begin{equation}\label{3.10}
(F^n(t,z)-F^n(s,z))^+\leq C(\hat{\theta}^n(t,z)-\hat{\theta}^n(s,z)).
\end{equation}
\end{lem}
\begin{proof}
As before, this follows from the discrete estimate. For any $k,l$ integers with $0\leq k<l\leq\frac{T}{\delta t}$, $(F^n_{l\delta t}(z)-F^n_{k\delta t}(z))^+\leq C(\hat{\theta}^n_{l\delta t}(z)-\hat{\theta}^n_{k\delta t}(z))$. This follows from sum (\ref{3.1}) from $k$ to $l-1$.
\end{proof}
\begin{cor}
Let $s,\,t\in[0,T)$, then for the same constant as in previous lemma, we have
$$
\int_{[0,1]}|F^n(t,z)-F^n(s,z)|dz\leq 2C\int_{[0,1]}(\hat{\theta}^n(t,z)
-\hat{\theta}^n(s,z))dz.
$$
\end{cor}
\begin{proof}
SInce $F^n$ measure preserving, we know $\int_{[0,1]}(F^n(t,z)-F^n(s,z))^+dz=\int_{[0,1]}(F^n(t,z)-F^n(s,z))^-dz$. Then the result follows by integrating (\ref{3.10}) in $z$.
\end{proof}
The point (vi) follows from previous corollary and the continuity estimate for $\theta$.
\section{Definition of measure valued solution}
In this section, we wish to define the measure valued solutions. As suggested in the discussion in the first section, we need to consider "path-spaces" which represents all the possible trajectories of an arbitrary parcel. Thanks to the point (iii) of \ref{3.1t}, such paths take value in $[0,1]$, and should have bounded variation, with uniform bound on $BV$.

Let $B_1>0$ and $X_{B_1}$ be the set of functions $f:(0,T)\rightarrow[0,1]$ which is left continuous and has total variation no bigger than $B_1$, that is to say, for any partition of the interval $[0,T]$, denoted as $0<t_0<t_1<t_2<\cdots<t_m<T$, we have 
$$
\sum_{i=1}^m|f(t_i)-f(t_{i-1})|\leq B_1.
$$

Let $d$ be the $L^2$ distance for functions in $X_{B_1}$, that is
$$
d(f,g)=\big(\int_0^T|f(t)-g(t)|^2dt\big)^{\frac{1}{2}}.
$$
It is not hard to see that $d$ is indeed a distance, since we required continuity from left. 
Also one can check $(X_{B_1},d)$ is a complete separable  metric space, by Helly`s selection principle. 
The physical meaning of such a space $X_{B_1}$ is the space of all possible paths of the parcels.
The reason why such paths have bounded variation is due to point (iii) of Theorem 4.1.

Let $B_2>0$ and $Y_{B_2}$ be the space of monotone increasing function on $[0,1]$, right continuous on $[0,1)$, and with absolute bound $\leq {B_2}$, 
equipped with $L^2-$ distance. That is, given $h,k\in Y_{B_2}$, define their distance to be $d^{\prime}(h,k)=||h-k||_{L^2(0,1)}$.
The physical meaning of this space is all the possible profiles of potential temperature $\theta$. We have incorporated the physical constraint that they must be monotone increasing.

Let $B_3>0$, put $Y=C([0,T);Y_{B_3})$, that is, $Y$ is the space of continuous maps from $[0,T)$ to $Y_{B_3}$. Here $B_3$ will be determined later on. 
We can define a metric on the space $Y$: given $h,k:[0,T)\rightarrow Y_{B_3}$, define $d(h,k)=\max_{t\in[0,T)}d^{\prime}(h(t),g(t))$.
The physical meaning of the space is the possible evolutions of the potential temperature profile.

To avoid confusion, we will denote a generic  element from the space $Y_{B}$ to be $\theta$, while a generic element from the space $Y$ will be denoted by $\tilde{\theta}$.
It is easy to check both $Y_{B}$ and $Y$ are complete separable metric spaces, for any fixed $B>0$.

First we specify the class of initial data we will be considering. Since we will be considering solutions in some "probabilistic" sense, our initial data will also be "probabilistic", namely some probability distributions. Let $B_2>0$ be an arbitrary positive constant.
\begin{defn}
Let $\zeta_0\in\mathcal{P}( Y_{B_2}\times\mathbb{R}\times[0,1])$. We say $\zeta_0$ is an admissible data if the following holds:
\begin{equation*}
\begin{split}
&\textrm{(i)$\pi_{13}\#\zeta_0
=\mu_0\times\mathcal{L}^1_{[0,1]}$ for some $\mu_0\in\mathcal{P}(Y_{B_2})$, and $\pi_2\#\zeta_0$ has compact support;}\\
&\textrm{(ii)$Q^{sat}(\theta(z),z,0)\geq s-\theta(z)$ for $\zeta_0-$ a.e $(\theta,s,z)$.}
\end{split}
\end{equation*}
\end{defn}
\sloppy
\begin{rem}
Heuristically, $\zeta_0$ can be thought of prescribing the probability distribution of $\{\theta,\theta^M(z)\}_{z\in[0,1]}$, where $\theta^M=\theta+q$. 
Indeed, using that $\pi_{13}\#\zeta_0=\mu_0\times\mathcal{L}^1_{[0,1]}$ and also the disintegration theorem, we can write $\zeta_0=\int_{Y_{B_2}\times[0,1]}
d(\mu_0\times\mathcal{L}^1_{[0,1]})(\theta,z)
\int_{\mathbb{R}}d\zeta_{\theta,z}(s)$. 
Here $(\theta,z)\longmapsto d\zeta_{\theta,z}(s)$ is a Borel family of probability measures on $\mathbb{R}$..
This describes the probability distribution of $\theta^M$, hence $q$, given $\theta$ and $z$.
The second point simply says the physical constraint is satisfied with probability 1.
\end{rem}
We can define the following evaluation maps for the space $X_{B_1}$ and $Y$. Fix any $t\in(0,T)$, define $e_t:X\rightarrow [0,1]$ by
$\gamma\longmapsto\gamma(t).$ Similarly define $e'_t:Y\rightarrow Y_{B_3}$ by $\tilde{\theta}\longmapsto\tilde{\theta}(t)$.
We will frequently write $e_t(\gamma)=\gamma_t$ and $e'_t(\tilde{\theta})=\tilde{\theta}_t$ to simplify the notation. We see from the definition of the space $Y$ that $e_t^{\prime}$ is a continuous map. Also we can observe $e_t$ is Borel, even if it is not continuous in general.
Here we observe:
\begin{lem}\label{4.3a}
The set $\{(\theta,s,z)\in Y_{B_2}\times\mathbb{R}\times[0,1]:Q^{sat}(\theta(z),z,0)\geq s-\theta(z)\}$ is a Borel subset of $Y_{B_2}\times\mathbb{R}\times[0,1]$. Also the map $e_t$ defined in previous paragraph is Borel map.
\end{lem}
\begin{proof}
We just need to show the evaluation map $A:(\theta,z)\longmapsto\theta(z)$ is Borel. 
Fix $\eps>0$, let $A_{\eps}(\theta,z)=\eps^{-1}\int_z^{z+\eps}\theta(w)dw$.  Here we extended $\theta(w)\equiv C$ for $w>1$.
Then for each fixed $\eps>0$, $A_{\eps}(\theta,z)$ is continuous, and for any fixed $(\theta,z)$, $A_{\eps}(\theta,z)\rightarrow A(\theta,z)$, since $\theta$ is right continuous after extension. This proves $A$ is Borel measurable.

That $e_t$ is Borel is proved in a similar way. First we can define $f(t)=f(0)$ for $t\leq0$. With this extension, $f(t)$ is define on $(-\infty,T)$ and is left continuous. Hence the map $f(t)\longmapsto\eps^{-1}\int_{t-\eps}^tf(s)ds$ will converge to $f(t)$ as $\eps$ tends to 0.
\end{proof}
A "deterministic" initial data takes the form $\zeta_0=\delta_{\theta^0}\times
((\theta^0+q^0)\times id)\#\mathcal{L}^1_{[0,1]}$. In this case, there is only one possible choice $\theta^0$, and for each fixed $z$, $\theta^M$ takes a deterministic value $\theta^0(z)+q^0(z)$.

Here we can make a definite choice of the constant $B_1$ and $B_3$ which are involved in defining the spaces $X_{B_1}$ and $Y$. 
By point (i) in the Definition 5.1, we can assume $\pi_2\#\zeta_0\subset[-M+1,M-1]$ for some $M>0$. We will determine $B_1$ and $B_3$ such that they depend only on $B_2$, $M$, $T$ and also the function $Q^{sat}$.
Now take $B_1$ to be the constant $C_3$  given by point (iii) of Theorem \ref{3.1t} if we have the bound $||\theta_0^n||_{L^{\infty}}+||q_0^n||_{L^{\infty}}\leq M+2B_2$. Let $B_3$ be the constant $C_1$ given by point (i) of Theorem \ref{3.1}. if we have the bound $||\theta_0^n||_{L^{\infty}}+||q_0^n||_{L^{\infty}}\leq M+2B_2$.  Without loss of generality we can assume $B_3>B_2$ so that $Y_{B_2}\subset Y_{B_3}$.

With such a choice of constants in place, we propose the following defintion of measure-valued solutions.
\begin{defn}\label{1.6d}
Let $\lambda\in\mathcal{P}(Y\times\mathbb{R}\times[0,1]\times X_{B_1})$, and denote $\eta_t=(e_t^{\prime}\times id\times id\times e_t)\#\lambda$, $\zeta_t=\pi_{124}\#\eta_t\in\mathcal{P}(Y_{B_3}\times\mathbb{R}\times[0,1])$. Then we say $\lambda$ is a measure valued solution to admissible initial data $\zeta_0$ if the following are satisfied:\\
(i)$\zeta_t\rightarrow\zeta_0$, $\pi_{34}\#\eta_t\in\Gamma(\mathcal{L}^1_{[0,1]},\mathcal{L}^1_{[0,1]})\rightarrow(id\times id)\#\mathcal{L}^1_{[0,1]}$ as $t\rightarrow0$, and $t\longmapsto\eta_t$ narrowly continuous.\\
(ii)For any $t\in(0,T)$, $\pi_{13}\#\zeta_t=\mu_t\times
\mathcal{L}^1_{[0,1]}$, $\pi_2\#\zeta_t$ has compact support. Besides $Q^{sat}(\theta(z),z,t)\geq s-\theta(z)$ for $\zeta_t-$a.e $(\theta,s,z)$.\\
(iii)For $\lambda-$a.e $(\tilde{\theta},s,z,\gamma)$, we have $\tilde{\theta}_t(\gamma_t)\leq
\tilde{\theta}_{t^{\prime}}
(\gamma_{t^{\prime}})$ for $\mathcal{L}^2-a.e$ $(t,t^{\prime})\in(0,T)^2$ and $t<t^{\prime}$.\\
(iv)For $\lambda-$a.e $(\tilde{\theta},s,z,\gamma)$, we have the equality of measures:
$$
\partial_t(\tilde{\theta}_t(\gamma_t))=\big[\partial_t
(Q^{sat}(\tilde{\theta}_t(\gamma_t),\gamma_t,t))\big]^-\lfloor E_{\tilde{\theta},s,\gamma},$$
where $E_{\tilde{\theta},s,\gamma}$ is the wet set given by
$E_{\tilde{\theta},s,\gamma}=\{t\in(0,T):(\tilde{\theta}(\gamma))^*(t)=s-Q^{sat}((\tilde{\theta}(\gamma))^*(t),\gamma_t,t)\}.$
\end{defn}
In the above, $(\tilde{\theta}(\gamma))^*$ is the monotone increasing, 
left continuous version of $\tilde{\theta}_t(\gamma_t)$ chosen according to Remark \ref{1.4r}. 
This is possible due to point (iii) in the above definition.
The notation $\Gamma(\mathcal{L}^1_{[0,1]},\mathcal{L}^1_{[0,1]})$ in point (i) denotes the set of Radon probability measures on $[0,1]^2$ whose projections on both components are equal to $\mathcal{L}^1_{[0,1]}$.

Point (i) simply specifies in what sense the initial data is satisfied. $\zeta_t$ gives the probability distribution of $\{\theta,\theta^M(z)\}_{z\in[0,1]}$ at time $t$, and narrowly converge to $\zeta_0$ as $t\rightarrow0$. 
The second convergence simply means the "flow map" converges to identity as $t\rightarrow0$. That $\pi_{34}\zeta_t\in\mathcal{L}^1_{[0,1]}$ is a reformulation of the measure preserving property of the flow map, namely the incompressibility.

Point (ii) shows that $\zeta_t$ obtained satisfies the same conditions as required by the "admissibility" of the data. Therefore, one can take any $\zeta_t$ as initial data and evolves the solution forward.

Point (iii) shows that for all the possible choice of evolution of $\tilde{\theta}$ and the fluid path $\gamma$, $t\longmapsto\tilde{\theta}_t(\gamma_t)$ is always monotone increasing in $t$(up to some set of Lebesgue measure 0). 

Point (iv) shows that for possible choice of evolution of $\tilde{\theta}$ and the fluid path $\gamma$, the correct equation is satisfied.

Next we show if the random evolution of the solution happens to be deterministic, then Definition \ref{1.3d} and Definition \ref{1.6d} are consistent.
\begin{lem}
Let $\lambda$ be a measure valued solution to admissble initial data $\zeta_0$. Assume $\zeta_0=\delta_{\theta_0}
\times\big((\theta_0+q_0)\times id\big)\#\mathcal{L}^1_{[0,1]}$. Assume also that $\lambda=\delta_{\theta(t)}
\times\big((\theta_0+q_0)\times id\times \Phi_F)\#\mathcal{L}^1_{[0,1]}$, for some Borel map $F:[0,T)\times[0,1]\rightarrow[0,1]$. 
Let $\Phi_F(z)$ associates each $z$ to its path $t\longmapsto F_t(z)$. We also assume that there exists inverse map $F^*:[0,T)\times[0,1]\rightarrow[0,1]$ with $F_t^*\circ F_t=id$ and $F_t\circ F_t^*=id$ for $\mathcal{L}^1_{[0,1]}-a.e\,z$ and any $t\in (0,T)$.
Then $(\theta(t,z),\theta_0(F_t^*(z)+q_0(F_t^*(z))-\theta(t,z),F)$ is a weak Lagrangian solution in the sense of Definition \ref{1.3d}.
\end{lem}
\begin{proof}
From the definition of admissible data, we know that for $\mathcal{L}^1-$ a.e $z\in[0,1]$, it holds $Q^{sat}(\theta_0(z),z,0)\geq q_0(z)$ for $\mathcal{L}^1-a.e$ $z\in[0,1]$. Since $supp\,\pi_2\#\zeta_0=(\theta_0+q_0)\#\mathcal{L}^1_{[0,1]}
\subset[-M,M]$, this means $||\theta_0+q_0||_{L^{\infty}}\leq M$.

That $\theta_t\in L^{\infty}([0,T)\times[0,1])\cap C([0,T),L^1([0,1]))$ follows from that $\theta_t\in C([0,T];Y_{B_2})$. Let $q_t(z)=\theta_0(F_t^*(z))+
q_0(F_t^*(z))-\theta(t,z)$. 
From the boundedness of $\theta_0+q_0$, we immediately get $q_t\in L^{\infty}([0,T)\times[0,1])$.

That $F_t(\cdot)\in C([0,T);L^1([0,1]))$ follows from that $\pi_{34}\#\eta_t=(id\times F_{t})\#\mathcal{L}^1_{[0,1]}$ is narrowly continuous.
From $\pi_{34}\#\eta_t\in\Gamma(\mathcal{L}^1_{[0,1]},\mathcal{L}^1_{[0,1]})$, 
we see $F_t\#\mathcal{L}^1_{[0,1]}=\mathcal{L}^1_{[0,1]}$, namely $F_t$ is measure preserving. 
Combined with the assumption $F_t^*\circ F_t=id$ shows that $(id\times F_t)\#\mathcal{L}^1_{[0,1]}=(F_t^*\times id)\#\mathcal{L}^1_{[0,1]}$, hence $F_t^*$ is also measure preserving.
As before, the narrow continuity of $t\longmapsto \pi_{34}\#\eta_t$ implies $F_t^*(\cdot)\in C([0,T);L^1([0,1]))$. 
This in turn implies $q_t\in C([0,T);L^1([0,1]))$.

From the definition of the space $X_{B_1}$, 
we see $F_{\cdot}(z)\in L^{\infty}([0,1];BV(0,T))$, 
with total variation $\leq C$. 
Next we will check through the points (i)-(vi) in the Definition \ref{1.3d}.

To see point (i), we observe that the measure $\pi_1\#\zeta_t\rightarrow
\pi_1\#\zeta_0$ narrowly. In other words, we have $\delta_{\theta_t}\rightarrow\delta_{\theta_0}$
narrowly in $\mathcal{P}(X_0)$. This implies $\theta_t\rightarrow\theta_0$ in $L^2([0,1])$. That $F_t\rightarrow id$ follows from that $\pi_{34}\#\eta_t=(id\times F_t)\#\mathcal{L}^1_{[0,1]}
\rightarrow(id\times id)\#\mathcal{L}^1_{[0,1]}$.

Point (ii) immediately follows from the definition of the space $Y$ and $Y_{B_2}$.

Point (iii) follows from the assumption of $F$ made in this lemma.

To see point (iv), Recall $q_t(z)=\theta_0(F_t^*(z))+
q_0(F_t^*(z))-\theta(t,z)$. Hence for any $t\in (0,T)$ $\hat{q}_t(z)=q_t(F_t(z))=\theta_0(F_t^*\circ F_t(z))+q_0(F_t^*\circ F_t(z))-\theta(t,F_t(z))=\theta_0(z)+
q_0(z)-\hat{\theta}_t(z)$, for $\mathcal{L}^1-$ a.e $z$.

Point (v) of Definition \ref{1.3d} follows from point (iii) of Definition \ref{1.6d}. Indeed, from point (iii), we know that for $\pi_{14}\#\lambda-a.e$ $(\tilde{\theta},\gamma)$, it holds $\tilde{\theta}_t(\gamma_t)\leq\tilde{\theta}_{t^{\prime}}(\gamma_{t^{\prime}})$ for $\mathcal{L}^2-$a.e $(t,t^{\prime})\in(0,T)^2$ with $t<t^{\prime}$. But $\pi_{14}\#\lambda=\delta_{\theta(t)}\times \Phi_F\#\mathcal{L}^1_{[0,1]}$. Hence for $\mathcal{L}^1-a.e\,z$, it holds $\theta_t(F_t(z))\leq\theta_{t^{\prime}}(F_{t^{\prime}}(z))$ for $\mathcal{L}^2$-a.e $(t,t')\in(0,T)^2$ with $t<t'$.

To see the last point, we know from our assumption on $\lambda$ that for $\lambda-a.e\,(\tilde{\theta},s,z,\gamma)$, it holds $s=\theta_0(z)+q_0(z)$, and $\gamma_t=F_t(z)$. For $(\tilde{\theta},s,z,\gamma)$ such that this holds, we know
\begin{equation*}
\begin{split}
E_{\theta,s,\gamma}=E_{\theta,\theta_0(z)+q_0(z),F_{\cdot}(z)}&=\{t\in(0,T):(\theta(F))^*_t(z)\\
&=\theta_0(z)+q_0(z)-Q^{sat}((\theta(F))^*_t(z),F_t(z),t)\}.
\end{split}
\end{equation*}
We have seen in the above proof that for $\mathcal{L}^1-a.e\,z$, we have $\hat{q}_t(z)=\theta_0(z)+q_0(z)-
\theta_t(F_t(z))$. Hence if we choose the monotone and left continuous representative, we have $\hat{q}_t^*(z)=\theta_0(z)
+q_0(z)-(\theta(F))_t^*(z)$.  Hence for $\lambda-a.e\,(\tilde{\theta},s,z,\gamma)$, $E_{\theta,s,\gamma}=E_z$, where $E_z$ is given in point (vi) of Definition \ref{1.3d}. Finally the measure theoretic equation holds $\lambda-a.e$. Since $\pi_2\#\lambda=\mathcal{L}^1_{[0,1]}$, we see the point (vi) of Definition \ref{1.3d} holds.
\end{proof}
The main existence theorem we will prove will be the following:
\begin{thm}
Let $\zeta_0$ be an admissible initial data, then there exists a measure valued solution to (\ref{1.1})-(\ref{1.4}) with initial data $\zeta_0$.
\end{thm}

\section{Existence of measure valued solutions}

In this section, we will show the existence of measure valued solutions to any admissible data defined in previous section.

The plan is the following: Let $\zeta_0$ be an admissible initial data, 
then we approximate $\zeta_0$ by convex combinations of discrete and "deterministic" initial data. 
For each measure appearing in the convex combination, 
we can run the discrete procedure described in section 3.
 The question is to show one can take the limit of this approximation and the properties given in Definition \ref{1.6d} hold in the limit.

\subsection{Discretizing the initial data}
We assume $\pi_2\#\zeta_0\subset[-M+1,M-1]$ for some $M>0$. Write $K=M-1$. As a preliminary step, we note the following:
\begin{lem}\label{5.1l}
Let $\zeta_0$ be an admissible data. 
Then there exists a Borel family of probability measures $\{\alpha_{\theta}\}_{\theta\in Y_{B_2}}\subset\mathcal{P}(\mathbb{R}\times[0,1])$, such that for $\mu_0-a.e$ $\theta$,
 supp $\pi_1\#\alpha_{\theta}\subset[-K,K]$,  $\pi_2\#\alpha_{\theta}=\mathcal{L}^1_{[0,1]}$, $\theta(z)\geq\Theta(s,z,0)$ for $\alpha_{\theta}-a.e$ $(s,z)$ and for any bounded Borel function $f(\theta,s,z):
Y_{B_2}\times\mathbb{R}
\times[0,1]\rightarrow\mathbb{R}$, it holds:
\begin{equation}\label{5.1}
\int_{Y_{B_2}\times\mathbb{R}\times[0,1]}
f(\theta,s,z)d\zeta_0(\theta,s,z)=\int_{Y_{B_2}}
d\mu_0(\theta)\bigg[\int_{\mathbb{R}\times[0,1]}f(\theta,s,z)d\alpha_{\theta}(s,z)\bigg].
\end{equation}

\end{lem}
\begin{proof}
The existence of this family of probability measures satisfying the integral identity (\ref{5.1}) is the standard disintegration theorem, see Theorem 5.3.1 of \cite{Gradient flows}. 
We just need to check that supp $\pi_1\#\alpha_{\theta}\subset[-K,K]$, $\pi_2\#\alpha_{\theta}=\mathcal{L}^1_{[0,1]}$, $\theta(z)\geq\Theta(s,z,0)$ for $\alpha_{\theta}-a.e$ $(s,z)$.

To see supp $\pi_1\#\alpha_{\theta}\subset[-K,K]$, take $f(\theta,s,z)=f_1(\theta)\chi_{[-K,K]^c}(s)$, and $f_1:Y_{B_2}\rightarrow\mathbb{R}$ bounded and Borel measurable. Then we see 
$$|\int f_1(\theta)\chi_{[-K,K]^c}(s)d\zeta_0(\theta,s,z)|\leq \sup|f|\int\chi_{[-K,K]^c}(s) d\zeta_0(\theta,s,z)=0.$$
From (\ref{5.1}), we know for any choice of bounded Borel function $f_1$, one has 
$$\int f_1(\theta)d\mu_0(\theta)\big[\int \chi_{[-K,K]^c}d\alpha_{\theta}(s,z)\big]=0.$$
This implies $\int \chi_{[-K,K]^c}(s)d\alpha_{\theta}(s,z)=0$ for $\mu_0-a.e$ $\theta$. For such $\theta$, we have supp $\pi_1\#\alpha_{\theta}\subset[-K,K]$.

To see $\pi_2\#\alpha_{\theta}=\mathcal{L}^1_{[0,1]}$,  one can similarly take $f(\theta,s,z)=f_1(\theta)f_2(z)$, using that $\pi_{13}\#\zeta_0=
\mu_0\times\mathcal{L}^1_{[0,1]}$, 
one concludes for any choice of $f_2(z)$ bounded and Borel on $[0,1]$ it holds $\int_{\mathbb{R}\times[0,1]} f_2(z)d\alpha_{\theta}(s,z)=\int f_2(z)dz$ for $\mu_0-a.e$ $\theta$. 
One just needs to choose a countable dense subset $\{f_2^n\}_{n\geq1}$ of $C([0,1])$, 
apply this argument with $f_2=f_2^n$ and concludes for $\mu_0-a.e$ $\theta$, $\pi_2\#\alpha_{\theta}=\mathcal{L}^1_{[0,1]}$.

To see that $\theta(z)\geq\Theta(s,z,0)$ for $\alpha_{\theta}-a.e$ $(s,z)$, we integrate $\chi_{\{(\theta,s,z):\theta(z)\geq\Theta(s,z,0)\}}$, and use (\ref{5.1}). By Lemma \ref{4.3a}, such a function is bounded and Borel, hence its integral is well defined.
\end{proof}
Due to Lemma \ref{5.1l}, we can write $\zeta_0=\int_{Y_{B_2}}\delta_{\theta}\times\alpha_{\theta}(s,z)d\mu_0(\theta)$, and for $\mu_0-a.e$ $\theta$, $\alpha_{\theta}(s,z)$ satisfy the "correct" condition mentioned in Lemma \ref{5.1l} Next we will construct discretization of $\alpha_{\theta}$ for each fixed such $\theta$. 

Recall we have shown supp $\alpha_{\theta}\subset[-K,K]\times[0,1]$, 
Define $K_j=[-K+\frac{(j-1)2K}{n},-K+\frac{j\cdot2K}{n}]$ for $1\leq j\leq n$, call $w_j=-K+\frac{j\cdot2K}{n}$. Suppose $\alpha_{\theta}(K_j\times J_i)>0$, 
since we assumed $\theta(z)\geq\Theta(s,z,0)$ for $\alpha_{\theta}-$a.e $(s,z)$, we have, for some $(\tilde{\alpha}_{ij},z_{ij})\in K_j\times J_i$, 
we have $\theta(z_{ij})\geq\Theta(\tilde{\alpha}_{ij},z_{ij},0)$.
Therefore, for some universal constant $C>2K+\frac{\sup|\partial_z\Theta|}{\inf\partial_w\Theta}$, if we define $\alpha_{ij}=w_j-\frac{C}{n}$, we have
$$
\theta(z_i)\geq\theta(z_{ij})\geq\Theta(\tilde{\alpha}
_{ij},z_{ij},0)\geq \Theta(\tilde{\alpha}_{ij},z_i,0)-\frac{\sup|\partial_z\Theta|}{n}\geq\Theta(w_j-\frac{C}{n},z_i,0).
$$
There is no loss of generality to assume $n$ is chosen sufficiently large so that $\frac{C}{n}<1$. Now we can define a measure $\alpha_{\theta}^n$ which is an approximation to $\alpha_{\theta}$, by putting:
$$
\alpha_{\theta}^n=\sum_{i=1}^n\sum_{j\in H_{\theta,i}}\chi_{J_i}(z)dz\cdot n\alpha_{\theta}(K_j\times J_i)\delta_{\alpha_{ij}}(s).
$$
Here we denote 
\begin{equation}
H_{\theta,i}=\{j:1\leq j\leq n,\,\alpha_{\theta}(K_j\times J_i)>0\}.
\end{equation}
\begin{lem} $\alpha_{\theta}^n\rightarrow\alpha_{\theta}$ narrowly in $\mathcal{P}(\mathbb{R}\times[0,1])$.
\end{lem}
\begin{proof}
Let $f\in C_b(\mathbb{R}\times[0,1])$ with Lipschitz constant 1, and we denote $H_i=H_{\theta,i}$ for simplicity.
\begin{equation*}
\begin{split}
&\int f(s,z)d\alpha_{\theta}(s,z)=\sum_{i=1}^n\sum_{J\in H_i}\int_{K_j\times J_i}f(z,s)d\alpha_{\theta}(s,z)\\
&=\sum_{i=1}^n\sum_{j\in H_i}\int_{K_j\times J_i}\big(f(z,s)-f(z_i,w_j-\frac{C}{n})\big)d\alpha_{\theta}(s,z)+\sum_{i=1}^n\sum_{j\in H_i}f(z_i,w_j-\frac{C}{n})\alpha_{\theta}(K_j\times J_i)\\
=\sum_{i=1}^n&\sum_{j\in H_i}\int_{J_i}f(z,w_j-\frac{C}{n})dz\cdot n\alpha_{\theta}(J_i\times K_j)+\sum_{i=1}^n\sum_{j\in H_i}\int_{J_i\times K_j}\big(f(z,s)-f(z_i,w_j-\frac{C}{n})\big)d\alpha_{\theta}(z,s)\\
&+\sum_{i=1}^n\sum_{j\in H_i}\int_{J_i}f(z_i,w_j-\frac{C}{n})-f(z,w_j-\frac{C}{n})dz\cdot n\alpha_{\theta}(J_i\times K_j).
\end{split}
\end{equation*}
The first term above is exactly the integral of $f$ with respect to $\alpha_0^n$. The last two terms will go to zero, because in each term of the sum, the integrand is  controlled by $\frac{C}{n}$.
\end{proof}
For any $B>0$ and each $\theta\in Y_{B}$, we can define $D^n:Y_{B}\rightarrow Y_{B}$ by putting $D^n(\theta)=\sum_{i=1}^n\theta(z_i)\chi_{J_i}$. Then for each fixed $\theta$, it holds $D^n(\theta)\rightarrow \theta$ in $Y_{B}$. Now with choice $B=B_2$.
By putting $\zeta_0^n=\int_{Y_{B_2}}\delta_{D^n(\theta)}\times\alpha^n_{\theta}(s,z)d\mu_0(\theta)$, 
we then have $\zeta_0^n\rightarrow\zeta_0$ narrowly in $\mathcal{P}(Y_{B_2}\times\mathbb{R}\times[0,1])$ as $n\rightarrow\infty$. Besides, $\zeta_0^n$ satisfies the following properties:

(i)$\pi_{13}\#\zeta_0^n=(D^n
\#\mu_0)\times\mathcal{L}^1_{[0,1]}$, supp $\pi_2\#\zeta_0^n\subset[-K-1,K+1]=[-M,M]$;

(ii)$\theta(z_i)\geq\Theta(w_j-\frac{C}{n},z_i,0)$, whenever $j\in H_{\theta,i}$, or $\theta(z)\geq\Theta(s,z,0)$ for $\zeta_0^n-a.e$ $(\theta,s,z)$.

For simplicity of notation, we will write
\begin{equation}\label{5.3}
\alpha_{\theta}(s,z)=\sum_{i=1}^n\sum_{j\in H_{\theta,i}}n\chi_{J_i}(z)\cdot\mu_{ij}^{\theta}\delta_{\alpha_{ij}}(s),\textrm{ with $\mu_{ij}^{\theta}=\alpha_{\theta}(K_j\times J_i)$, and $\alpha_{ij}=w_j-\frac{C}{n}$.}
\end{equation}
\subsection{Construction of approximate solutions and passage to limit}
The measures $\zeta_0^n$ determines a sequence of discrete probability distributions. 
For each choice of $\theta$, $\alpha_{\theta}$ prescribes the probability distribution of $\theta^M(z)$ for each fixed $z$. 
More precisely, when $z\in J_i$, the possible values of $\theta^M$ are given by $\alpha_{ij}$, with probability $\mu_{ij}$.

In order to apply the discrete procedure, we will make a random choice of $\alpha_{ij}$ on each $J_i$, and this gives us a "deterministic" and discrete initial $\theta^M$. 
Then we run the discrete procedure, with this $\theta^M$ as initial data and it gives us a evolution, with probability determined by the choice of $\alpha_{ij}$.

First we make a random choice of the $\alpha_{ij}$, allowed by the physical constraint. 
Denote $S_{\theta}$ to be the set of functions $\sigma:\{1,2,\cdots,n\}\rightarrow\{1,2,\cdots,n\}$, such that for each $i$, $\sigma(i)\in H_{\theta,i}$. 
Determine a discrete initial $\theta_0^{M,n}$ from $\sigma$ by prescribing:
$$
\theta^{M,n}_{0,\sigma}(z)=\sum_{i=1}^n
\alpha_{i\sigma(i)}\chi_{J_i}(z),\textrm{ or $\theta^{M,n}_{0,\sigma,i}=\alpha_{i\sigma(i)}$.}
$$
The probability of such a choice is given by $n^n\mu_{1\sigma(1)}\mu_{2\sigma(2)}\cdots
\mu_{n\sigma(n)}$.  Notice here that $\sum_j\mu_{ij}=\frac{1}{n}$, hence $\sum_{\sigma}\mu^{\theta}_{1\sigma(1)}\cdots\mu^{\theta}_{n\sigma(n)}=n^{-n}$.
It is straightforward to check that $\alpha_{\theta}$ given by (\ref{5.3}) is equal to
\begin{equation}\label{5.4}
\alpha_{\theta}=\sum_{\sigma\in S_{\theta}}n^n\mu_{1\sigma(1)}^{\theta}\cdots\mu_{n\sigma(n)}^{\theta}(\theta_{0,\sigma}^{M,n}\times id)\#\mathcal{L}^1_{[0,1]}.
\end{equation}
After we make such a choice, we apply the discrete procedure described in section 3 to $\{\theta(z_i)\}_{i=1}^n$,
 and $\{\theta^{M,n}_{\sigma,i}-\theta(z_i)\}_{i=1}^n$, with time step size $\delta t=\frac{1}{nC_4^{\prime}}$, 
here $C_4'$ is given by Theorem \ref{3.1t} in order for all the estimates in that theorem to hold. Notice that $\theta(z_i)\leq\theta(z_{i+1})$, 
and $\theta(z_i)\geq\Theta(\theta^{M,n}_{\sigma,i},z_i,0)$, hence it satisfies the assumptions made in the beginning of section 2. 
Denote $\{\theta_{\sigma,j}^n(k\delta t)\}_{1\leq j\leq n,0\leq k\leq\frac{T}{\delta t}}$, $\{\theta^{M,n}_{\sigma,j}(k\delta t)\}_{1\leq j\leq n,0\leq k\leq\frac{T}{\delta t}}$ to be the discrete solutions contructed according to section 2 and we adopt similar notations as there, 
but here with dependence on $\sigma$.
Define
\begin{align}\label{0}
&\theta_0^n(z)=\sum_i\theta(z_i)\chi_{J_i}.\\
\label{1}
&\theta_{\sigma}^n(t,z)=\theta^n_{\sigma,j}(k\delta t),\\
\label{2}
&\theta^{M,n}_{\sigma}(t,z)=\theta^{M,n}_{\sigma,j}(k\delta t),\\
\label{3}
&\bar{\theta}^n_{\sigma}(t,z)=\frac{(k+1)\delta t-t}{\delta t}\theta^n_{\sigma,j}(k\delta t)+\frac{t-k\delta t}{\delta t}\theta^n_{\sigma,j}((k+1)\delta t),
\end{align}
for $k\delta t\leq t<(k+1)\delta t$ and $z\in J_j$.

Denote $F^n_{\sigma}(t,z)$ be the discrete flow map constructed in section 3. Since we know $\theta\in Y_{B_2}$, we see $||\theta_0^n||_{L^{\infty}}\leq B_2$, and $||\theta_{0,\sigma}^{M,n}||_{L^{\infty}}\leq M$. Hence $||\theta_0^n||_{L^{\infty}}+||q_0^n||_{L^{\infty}}\leq M+2B_2$.
We know from point (iii) of Theorem \ref{3.1t}, as well as the choice of the constant $B_1$ made in the paragraph before Definition \ref{1.6d}, 
that $F_{\sigma}(\cdot,z)\in X_{B_1}$, namely $TV_{t\in(0,T)}(F_{\sigma}^n(t,z))\leq B_1$.  
Similarly we have $||\theta^n_{\sigma}||_{L^{\infty}}\leq B_3$, according to the point (i) of Theorem \ref{3.1t}, as well as the choice of $B_3$ made before Definition \ref{1.6d}.
In particular, for each fixed $t\in[0,T)$, it holds $\bar{\theta}^n_{\sigma}(t,\cdot)\in Y_{B_3}$. Hence $\bar{\theta}^n_{\sigma}\in C([0,T);Y_{B_3})=Y$.

Let $\Phi_{\sigma}^n:[0,1]\rightarrow X_{B_1}$ be defined by $\Phi_{\sigma}^n(z)=F_{\sigma}^n(\cdot,z)$.
From the discussion in the previous paragraph, we know $\Phi_{\sigma}^n$ indeed maps $[0,1]$ into the space $X_{B_1}$ and is easily seen to be a Borel map.

Now we form the probability measure:
\begin{equation}\label{4.3}
\lambda^n=\int_{Y_{B_2}}d\mu_0(\theta)\sum_{\sigma\in S_{\theta}}n^n\mu^{\theta}_{1\sigma(1)}\cdots
\mu^{\theta}_{n\sigma(n)}
\big(\delta_{\bar{\theta}^n_{\sigma}}\times(
\theta^{M,n}_{0,\sigma}\times id\times \Phi_{\sigma}^n)\#\mathcal{L}^1_{[0,1]}\big).
\end{equation}
Then we see $\lambda^n\in\mathcal{P}(Y\times\mathbb{R}\times[0,1]\times X_{B_1})$, by our previous condtruction, and this will be our approximate solutions.
Following the notations in Definition \ref{1.6d}, define $\eta_t^n=(e_t\times id\times id\times e'_t)\#\lambda^n$,  we have
\begin{equation}\label{6.10new}
\eta_t^n=\int_{Y_{B_2}}d\mu_0(\theta)\sum_{\sigma\in S_{\theta}}n^n\mu^{\theta}_{1\sigma(1)}\cdots
\mu^{\theta}_{n\sigma(n)}
\big(\delta_{\bar{\theta}^n_{\sigma}(t)}\times(
\theta^{M,n}_{0,\sigma}\times id\times F^n_{\sigma,t})\#\mathcal{L}^1_{[0,1]}\big).
\end{equation}

In order to take limit, we need to show that $\lambda^n$ is tight. For this, we just need to show $\pi_k\#\lambda^n$ is tight, for $k=1,2,3,4$, where $\pi_k$ is the projection map onto the $k$-th component. Indeed, we have 
$$
\pi_1\#\lambda^n=\int_{Y_{C}}
d\mu_0(\theta)\sum_{\sigma\in S_{\theta}}n^n\mu^{\theta}_{1\sigma(1)}
\cdots\mu^{\theta}_{n\sigma(n)}
\delta_{\theta^n_{\sigma}}.
$$
Because of point (v) in Theorem \ref{3.1t}, we see that for any $\sigma\in S_{\theta}$, and $s,t\in[0,T)$, with $t-s>\frac{\delta t}{2}$ we have
$$
||\theta^n_{\sigma}(s,\cdot)-\theta^n_{\sigma}(t,\cdot)||_{L^1}\leq C_5\sqrt{3(t-s)}.
$$
The constant $C_5$ is given by point (v) of Theorem \ref{3.1t} when we have the bound $||\theta_0^n||_{L^{\infty}}+||q_0^n||_{L^{\infty}}\leq M+2B_2$.
Recalling (\ref{3}), it follows that $\bar{\theta}^n_{\sigma}$ satisfies the same estimate, but without the restriction $t-s>\frac{\delta t}{2}$. Therefore, $\pi_1\#\lambda^n$ is concentrated on the compact set:
$$
F=\{f\in Y:||f(t,\cdot)-f(s,\cdot)||_{L^2}\leq C_5\sqrt{3(t-s)}\textrm{ for any $s<t$}\}.
$$
That this is a compact set follows from the Arzela-Ascoli theorem and the compactness of the space $Y_{B_3}$.
Note that the functions in $Y$ has bound in $L^{\infty}$, so convergence in $L^p$ are equivalent for any $p<\infty$.

As for $\pi_2\#\lambda^n$, since for all $\sigma\in S_{\theta}$, we know from the construction of discrete solutions that $||\theta^{M,n}_{0,\sigma}||_{L^{\infty}}\leq M$. Hence $\supp\,\pi_2\#\lambda^n\subset[-M,M]$, hence is tight. 

There is nothing to prove for $\pi_3\#\lambda^n$ and $\pi_4\#\lambda^n$, as $[0,1]$ and $X_{B_1}$ are compact spaces. 

So up to extracting a subsequence, we can pass to limit for the measure defined in (\ref{4.3}), 
and we denote a limit measure to be $\lambda\in\mathcal{P}(Y\times\mathbb{R}\times[0,1]\times X_{B_1})$.
It only remains to show $\lambda$ is a solution.
Finally we record a lemma which will be useful in the next subsection.
\begin{lem}\label{5.3l}
Define $\eta_t^n=(e_t^{\prime}\times id\times id\times e_t)\#\lambda^n$, then $\eta_t^n\rightarrow \eta_t$ for any $t\in (0,T)$ as $n\rightarrow\infty$.
\end{lem}
\begin{proof}
First we remind the reader that this is not completely obvious as the evaluation map $e_t$ in the last component is not continuous.
Instead, we will use the continuity estimates in point (vi) of Theorem \ref{3.1t}.

For $\gamma\in X_{B_1}$, extend $\gamma(t)=\gamma(0+)$ for $t\leq0$, and define the operator: $A^{\prime}_{\eps}:X_{B_1}\rightarrow X_{B_1}$, given by $\gamma\longmapsto\eps^{-1}\int_{t-\eps}^t\gamma(s)ds$. Then it is straightforward to check $A'_{\eps}$ is a  continuous map for each fixed $\eps$,
 and $A_{\eps}(\gamma)\rightarrow\gamma$ for each fixed $\gamma$ as $\eps\rightarrow0$.

Let $f\in C_b(Y_{B_3}\times\mathbb{R}\times[0,1]\times[0,1])$ and is 1-Lipschitz, then we can compute
\begin{equation}\label{5.7}
\begin{split}
&\int f(\theta,s,z,z^{\prime})
d\eta_t^n(\theta,s,z,z^{\prime})=\int f(\tilde{\theta}_t,s,z,\gamma_t)d\lambda^n(\tilde{\theta},s,z,\gamma)\\
&=\int(f(\tilde{\theta}_t,s,z,\gamma_t)-f(\tilde{\theta}_t,s,z,(A'_{\eps}(\gamma))_t)d\lambda^n(\tilde{\theta},s,z,\gamma)+\int f(\tilde{\theta}_t,s,z,(A_{\eps}(\gamma))_t)d\lambda^n(\tilde{\theta},s,z,\gamma).
\end{split}
\end{equation}
Now observe that for each fixed $\eps>0$, the map $(\tilde{\theta},s,z,\gamma)\longmapsto(\tilde{\theta}_t,s,z,(A'_{\eps}(\gamma))_t)$ is continuous. Hence for each fixed $\eps>0$, the following convergence holds as $n\rightarrow\infty$.
\begin{equation}\label{5.8}
\int f(\tilde{\theta}_t,s,z,(A_{\eps}(\gamma))_t)d\lambda^n(\tilde{\theta},s,z,\gamma)\rightarrow\int f(\tilde{\theta}_t,s,z,(A_{\eps}(\gamma))_t)d\lambda(\tilde{\theta},s,z,\gamma).
\end{equation}
To estimate the first term of (\ref{5.7}), observe
\begin{equation}\label{5.9}
\begin{split}
&|\int(f(\tilde{\theta}_t,s,z,\gamma_t)-f(\tilde{\theta}_t,s,z,(A_{\eps}(\gamma))_t)d\lambda^n(\tilde{\theta},s,z,\gamma)|\leq\int
|\gamma_t-(A_{\eps}(\gamma))_t|d\lambda^n(\tilde{\theta},s,z,\gamma)\\
&\leq\int d\mu_0(\theta)\sum_{\sigma\in S_{\theta}}n^n\mu_{1\sigma(1)}^{\theta}
\cdots\mu_{n\sigma(n)}^{\theta}\int_{[0,1]}|F^n_{\sigma,t}(z)-(A_{\eps}F^n_{\sigma,\cdot}(z))(t)|dz\\
&\leq\int d\mu_0(\theta)\sum_{\sigma\in S_{\theta}}n^n\mu_{1\sigma(1)}^{\theta}
\cdots\mu_{n\sigma(n)}^{\theta}\cdot\eps^{-1}
\int_{t-\eps}^t\int_{[0,1]}|F^n_{\sigma,t}(z)-F_{\sigma,s}^n(z)|dzds\leq C_6\sqrt{\eps+\delta t}.
\end{split}
\end{equation}
In the last inequality, we used point (vi) of Theorem \ref{3.1t}.

Combining (\ref{5.7})-(\ref{5.9}), it follows that 
$$
\lim\sup_{n\rightarrow\infty}|\int f(\theta,s,z,z^{\prime})d(\eta_t^n-\eta_t)|\leq C_6\sqrt{\eps}+\int|\gamma_t-(A_{\eps}(\gamma))_t|d\lambda(\tilde{\theta},s,z,\gamma).
$$
Now using the bounded convergence theorem, we can conclude the integral on the right hand side tends to zero as $\eps\rightarrow0$ since $(A_{\eps}(\gamma))_t\rightarrow\gamma_t$ as $\eps\rightarrow0$ and $\gamma$ is left continuous.
\end{proof}
\subsection{The limit is a solution}
In this section, we will show the limit $\lambda$ obtained in previous subsection is a measure valued solution. Some preparations are needed before we proceed. 

For any $B>0$, we may define the "averaging" operator: $A_{\eps}:Y_B\rightarrow Y_B$, given by $A_{\eps}(\theta)(z)=\eps^{-1}\int_z^{z+\eps}\theta(w)dw$.
Here we extended the definition of $\theta$ so that $\theta(z)=B$ for $z\geq1$.
It is clear that for any $\theta_1,\,\theta_2\in Y_B$, one has $||A_{\eps}(\theta_1)-A_{\eps}(\theta_2)||_{L^2}\leq||\theta_1-\theta_2||_{L^2}$. 
Hence $A_{\eps}$ is a continuous map for each fixed $\eps>0$. Also it is clear that for any $\theta\in Y_B$, we have $A_{\eps}\theta\rightarrow \theta$ in $Y_B$ as $\eps\rightarrow 0$.
We can define a map $A_{\eps}:Y\rightarrow Y$ by the same formula, namely $A_{\eps}(\tilde{\theta})(t,z)=\eps^{-1}\int_z^{z+\eps}\tilde{\theta}(w,t)dw$. Also one can check $A_{\eps}:Y\rightarrow Y$ is a continuous map for each $\eps\rightarrow0$, and $A_{\eps}(\tilde{\theta})\rightarrow\tilde{\theta}$ in $Y$ as $\eps\rightarrow0$.

 Now choose $B=B_3$, we prove the following estimate about $A_{\eps}$.
\begin{lem}\label{6.4ln}
Let $\eps>0$. For any $n\geq1$, and any $t\in(0,T)$, it holds
$$
\int_{Y\times X_{B_1}}|\tilde{\theta}_t(\gamma_t)-(A_{\eps}\tilde{\theta})_t(\gamma_t)|d\pi_{14}\#\lambda^n(\theta,\gamma)\leq B_3\eps.
$$
\end{lem}
\begin{proof}
From the definition of the measure $\lambda^n$, we can calculate:
\begin{equation*}
\begin{split}
&\int_{Y\times X_{B_1}}|\tilde{\theta}_t(\gamma_t)
-(A_{\eps}\tilde{\theta})_t(\gamma_t)|d\pi_{14}\#\lambda^n(\tilde{\theta},\gamma)dt\\
&=\int_{Y_{B_2}}d\mu_0(\theta)\sum_{\sigma\in S_{\theta}}n^n\mu_{1\sigma(1)}
\cdots\mu_{n\sigma(n)}^{\theta}\int_{[0,1]}|\bar{\theta}^n_{\sigma,t}(F^n_{\sigma,t}(z))-(A_{\eps}
\bar{\theta}^n_{\sigma})_t(F^n_{\sigma,t}(z))|dz\\
&=\int_{Y_{B_2}}d\mu_0(\theta)\sum_{\sigma\in S_{\theta}}n^n\mu_{1\sigma(1)}
\cdots\mu_{n\sigma(n)}^{\theta}\int_{[0,1]}|\bar{\theta}^n_{\sigma,t}(z)-(A_{\eps}
\bar{\theta}^n_{\sigma})_t(z)|dz.
\end{split}
\end{equation*}
In the last equality, we used the measure preserving property of $F_{\sigma,t}^n$.
For each fixed $\theta$, $\sigma$, we can estimate
\begin{equation*}
\begin{split}
&\int_{[0,1]}|\bar{\theta}^n_{\sigma,t}(z)-(A_{\eps}
\bar{\theta}^n_{\sigma})_t(z)|dz=\eps^{-1}\int_0^{\eps}ds\int_{[0,1]}\bar{\theta}^n_{\sigma,t}(z+s)-\bar{\theta}^n_{\sigma,t}(z)dz\\
&=\eps^{-1}\int_0^{\eps}ds
\big[\int_1^{1+s}
\bar{\theta}^n_{\sigma,t}(z)dz-\int_0^s\bar{\theta}^n_{\sigma,t}(z)dz\big]ds\leq
\eps^{-1}\int_0^{\eps}2B_3sds\leq B_3\eps.
\end{split}
\end{equation*}
In the first equality above, we used the monotonicity of $\bar{\theta}^n_{\sigma}$ in $z$. 
\end{proof}
Next we will check the properties listed in Definition \ref{1.6d} one by one.

\begin{lem}
Following the notation of Definition \ref{1.6d}, we have $t\longmapsto\eta_t\in\mathcal{P}(Y_{B_3}\times\mathbb{R}\times[0,1]\times[
0,1])$ is narrowly contiuous, and $\zeta_t\rightarrow\zeta_0$, $\pi_{34}\#\eta_t\in\Gamma(\mathcal{L}^1_{[0,1]},\mathcal{L}^1_{[0,1]})\rightarrow(id\times id)\#\mathcal{L}^1_{[0,1]}$ as $t\rightarrow0$.
\end{lem}
\begin{proof}
First we check the continuity. Let $f\in C_b(Y_{B_3}\times\mathbb{R}\times[0,1]\times
[0,1])$ be 1-Lipschitz. We compute
\begin{equation*}
\begin{split}
\int f(\theta,s,z,z^{\prime})d\eta_t^n(\theta,
z,s,z^{\prime})=\int_{Y_{B_2}}d\mu_0(\theta)
\bigg[\sum_{\sigma\in S_{\theta}}n^n\mu^{\theta}_{1\sigma(1)}\cdots
\mu^{\theta}_{n\sigma(n)}
\int f(\bar{\theta}_{\sigma}^n(t),z,\theta_{0,\sigma}^{M,n}(z),F_{\sigma,t}^n(z))dz\bigg].
\end{split}
\end{equation*}
Now choose $t,t^{\prime}$, with $t^{\prime}>t$, one has for each fixed $\theta\in Y_{B_2}$, and each $\sigma\in S_{\theta}$,
\begin{equation*}
\begin{split}
&\int |f(\bar{\theta}_{\sigma}^n(t),\theta_{0,\sigma}^{M,n}(z),z,F_{\sigma,t}^n(z))-f(\bar{\theta}_{\sigma}^n(t^{\prime}),\theta_{0,\sigma}^{M,n}(z),z,F^n_{\sigma,t^{\prime}}(z)|dz\\
&\leq ||\bar{\theta}_{\sigma}^n(t)-\bar{\theta}_{\sigma}^n
(t^{\prime})||_{L^2([0,1])}+\int|F_{\sigma,t}^n(z)-F_{\sigma,t^{\prime}}^n(z)|dz\leq (C_5B_3+C_6)\sqrt{t^{\prime}-t+\delta t}.
\end{split}
\end{equation*}
Here the constants $C_5$ and $C_6$ are given by the points (v) and (vi) of Theorem \ref{3.1t}. Now since $\sum_{\sigma\in S}n^n\mu_{1\sigma(1)}\cdots\mu_{n\sigma(n)}
=1$, we obtain for any $t<t^{\prime}$:
\begin{equation}\label{4.4}
|\int f(\theta,s,z,z^{\prime})d(\eta_t^n-
\eta_{t^{\prime}}^n)(\theta,
s,z,z^{\prime})|\leq (C_5B_3+C_6)\sqrt{t^{\prime}-t+\delta t}.
\end{equation}
The continuity now follows from sending $n\rightarrow\infty$ in (\ref{4.4}) and use Lemma \ref{5.3l}.
To show that $\zeta_t\rightarrow\zeta_0$, we show that for each $f\in C_b(Y_{B_3}\times\mathbb{R}\times[0,1])$, and 1-Lipschitz, one has $\sup_n|\int fd\zeta_t^n-\int f\zeta_0^n|\leq C\sqrt{t}$, with $C$ universal. Indeed we calculate 
\begin{equation*}
\begin{split}
\int f(\theta,s,z^{\prime}
)d\zeta_t^n(\theta,s,z^{\prime})=\int d\mu_0
(\theta)\sum_{\sigma\in S_{\theta}}n^n\mu_{1\sigma(1)}^{\theta}\cdots\mu_{n\sigma(n)}^{\theta}\int f(\bar{\theta}^n_{\sigma}(t),\theta^{M,n}_{0,\sigma}(z),F_{\sigma,t}^n(z))dz.
\end{split}
\end{equation*}
On the other hand, from (\ref{5.4}), we know that
$$
\int f(\theta,s,z^{\prime})
d\zeta_0^n(\theta,s,z^{\prime})=\int d\mu_0(\theta)\sum_{\sigma\in S_{\theta}}n^n\mu_{1\sigma(1)}^{\theta}
\cdots\mu_{n\sigma(n)}^{\theta}\int f(\theta^n_0,\theta^{M,n}_{0,\sigma}(z),z)dz.
$$
Using point (v) and (vi) in Theorem \ref{3.1t} once more, we see that for any choice of $\theta\in Y_{B_2}$ , $s\in\mathbb{R}$, and $\sigma\in S_{\theta}$, we have
\begin{equation*}
\begin{split}
\int_{[0,1]}|f(\theta^n_0,s,z)-f(\bar{\theta}^n_{\sigma}(t),s,F^n_{\sigma,t}(z)|dz&\leq||\theta_0^n-\bar{\theta}^n_{\sigma,t}||_{L^2}+\int_{[0,1]}|z-F^n_{\sigma,t}(z)dz|\\
&\leq(C_5B_3+C_6)\sqrt{t+\delta t}.
\end{split}
\end{equation*}
Then we can proceed in a similar way as before. That $\pi_{34}\#\eta_t^n\in\Gamma(\mathcal{L}^1_{[0,1]},\mathcal{L}^1_{[0,1]})$ follows readily from (\ref{6.10new}). The convergence for $\pi_{34}\#\eta^n_t$ is similar.
\end{proof}

Next we check the point (ii) of Definition \ref{1.6d}.
\begin{lem}\label{5.5l}
For any $t\in (0,T)$, $\pi_{13}\#\zeta_t=\mu_t\times\mathcal{L}^1_{[0,1]}$, $\pi_2\#\zeta_t$ has compact support. Besides, $\theta(z)\geq\Theta(s,z,t)$ for $\zeta_t-a.e$ $(\theta,s,z)$.
\end{lem}
\begin{proof}
From (\ref{6.10new}) we conclude:
\begin{equation}\label{5.14}
\zeta_t^n=\int_{Y_{B_2}}d\mu_0(\theta)
\sum_{\sigma\in S_{\theta}}n^n\mu_{1\sigma(1)}^{\theta}
\cdots\mu_{n\sigma(n)}^{\theta}
\delta_{\bar{\theta}^n_{\sigma}(t)}\times(\theta^{M,n}_{0,\sigma}\times F^n_{\sigma,t})\#\mathcal{L}^1_{[0,1]}.
\end{equation}
Hence we conclude
\begin{equation}\label{6.16new}
\pi_{13}\#\zeta_t^n=\int_{Y_{B_2}}d\mu_0(\theta)
\sum_{\sigma\in S_{\theta}}n^n\mu_{1\sigma(1)}^{\theta}
\cdots\mu_{n\sigma(n)}^{\theta}
\delta_{\bar{\theta}^n_{\sigma}(t)}\times \big(F_{\sigma,t}^n\#\mathcal{L}^1_{[0,1]}\big)
=\mu_t^n\times\mathcal{L}^1_{[0,1]}.
\end{equation}
In the above, we used the measure preserving property of the map $F_{\sigma,t}^n$ and here $\mu^n_t=\int_{Y_{B_2}}d\mu_0(\theta)
\sum_{\sigma\in S_{\theta}}\mu_{1\sigma(1)}^{\theta}
\cdots\mu_{n\sigma(n)}^{\theta}\delta_{\bar{\theta}^n_{\sigma}(t)}\in\mathcal{P}(Y_{C_1})$. 
Lemma \ref{5.3l} implies $\zeta_t^n\rightarrow\zeta_t$. On the other hand, we may assume $\mu_t^n\rightarrow\mu_t$. Passing to limit in (\ref{6.16new}), we see $\pi_{13}\#\zeta_t=\mu_t
\times\mathcal{L}^1_{[0,1]}$.

Now from (\ref{5.14}), one also calculate
$$
\pi_2\#\zeta_t^n=\int_{Y_{B_2}}d\mu_0(\theta)
\sum_{\sigma\in S_{\theta}}n^n\mu_{1\sigma(1)}^{\theta}
\cdots\mu_{n\sigma(n)}^{\theta}
\theta_{0,\sigma}^{M,n}
\#\mathcal{L}^1_{[0,1]}.
$$
From the construction given in the last subsection, we have $||\theta^{M,n}_{0,\sigma}||_{L^{\infty}}\leq M$, 
hence supp $\pi_2\#\zeta_t^n\subset[-M,M]$. 
Passing to the limit, the same will hold for $\pi_2\#\zeta_t$.

It remains to check that $\theta(z)\geq\Theta(s,z,t)$ for $\zeta_t-a.e$ $(\theta,s,z)$.
It suffices to show $\int(\theta(z)-\Theta(s,z,t))^-d\zeta_t(\theta,s,z)=0$. Now we calculate
\begin{equation*}
\begin{split}
&\int_{Y_{B_3}\times\mathbb{R}\times[0,1]}(\theta(z)-\Theta(s,z,t))d\zeta_t^n(\theta,s,z)\\
&=\int_{Y_{B_2}}d\mu_0(\theta)\sum_{\sigma\in S_{\theta}}n^n\mu_{1\sigma(1)}^{\theta}
\cdots\mu_{n\sigma(n)}^{\theta}\int_{[0,1]}(\bar{\theta}^n_{\sigma}(t,F^n_{\sigma,t}(z))-\Theta(\theta^{M,n}_{\sigma}(z),F_{\sigma,t}^n(z),t)^-dz\\
&\leq\int_{Y_{B_2}}d\mu_0(\theta)\sum_{\sigma\in S_{\theta}}n^n\mu_{1\sigma(1)}^{\theta}
\cdots\mu_{n\sigma(n)}^{\theta}\int_{[0,1]}|\bar{\theta}^n_{\sigma}(t,F_{\sigma,t}^n(z))-\theta^n_{\sigma}(t,F_{\sigma,t}^n(z)|dz+\sup|\partial_t\Theta|\delta t\\
&\leq (C_5+\sup|\partial_t\Theta|)\sqrt{\delta t}
\end{split}
\end{equation*}
In the first inequality above, we used that $\theta_{\sigma}^n(t,F_{\sigma,t}^n(z))
\geq\Theta(\theta_{\sigma}^{M,n}(z),F_{\sigma,t}^n(z),k\delta t)$, where $t\leq k\delta t<(k+1)\delta t$. This follows from point (iv) of Corollary \ref{2.2c}.
 In the last inequality, we used the point (v) of Theorem \ref{3.1t} and the measure preserving property of the map $F_{\sigma,t}^n$. 
Passing to the limit as $n\rightarrow\infty$, the conclusion follows.
\end{proof}
Now we check the point (iii) of Definition \ref{1.6d}. 
\begin{lem}\label{5.6l}
Fix $\eps>0$, define the function $I_{\eps}:Y\times X_{B_1}\rightarrow\mathbb{R}$, given by 
$$
(\tilde{\theta},\gamma)\longmapsto\int_{(0,T)^2}((A_{\eps}\tilde{\theta})_{t_1}(\gamma_{t_1})-(A_{\eps}\tilde{\theta})_{t_2}(\gamma_{t_2}))^+\chi_{t_1<t_2}dt_1dt_2
$$
is continuous.
\end{lem}
\begin{proof}
Let $(\tilde{\theta}^k,\gamma^k)\rightarrow(\tilde{\theta},\gamma)$ in $Y\times X_{B_1}$, we need to show $I(\tilde{\theta}^k,\gamma^k)\rightarrow I(\tilde{\theta},\gamma)$.
We prove this by showing for $\mathcal{L}^1-$a.e $t\in(0,T)$, we have pointwise convergence:$(A_{\eps}\tilde{\theta}^k)_t(\gamma^k)
\rightarrow(A_{\eps}\tilde{\theta})_t(\gamma_t)$, and then the desired convergence follows from dominated convergence theorem.

By Helly`s selection principle, we know that $\gamma^k_t\rightarrow\gamma_t$ except for a countable set of $t$. Since $|A_{\eps}(\tilde{\theta}^k)_t(z)-A_{\eps}(\tilde{\theta})_t(z)|\leq\eps^{-1/2}||\tilde{\theta}^k_t-\tilde{\theta}_t||_{L^2}$, 
we know that $A_{\eps}(\tilde{\theta}^k)_t\rightarrow A_{\eps}(\tilde{\theta})_t$ uniformly for each $t\in(0,T)$. Hence for any $t$ with $\gamma^k(t)\rightarrow\gamma(t)$, it holds $(A_{\eps}\tilde{\theta}^k)_t(\gamma^k_t)
\rightarrow(A_{\eps}\tilde{\theta})_t(\gamma_t)$.
It follows that $((A_{\eps}(\tilde{\theta}^k))_{t_1}(\gamma^k_{t_1})-(A_{\eps}\tilde{\theta}^k)_{t_2}(\gamma^k_{t_2}))^+\rightarrow((A_{\eps}(\tilde{\theta}))_{t_1}(\gamma_{t_1})-(A_{\eps}\tilde{\theta})_{t_2}(\gamma_{t_2}))^+$ 
for $\mathcal{L}^2-a.e$ $(t_1,t_2)$. Then the result follows from bounded convergence theorem.

\end{proof}

\begin{lem}\label{5.7l}
For any $n$, the following estimate holds:
$$
\int_{(0,T)^2}\int_{Y\times X_{B_1}}(\tilde{\theta}_{t_1}(\gamma_{t_1})-\tilde{\theta}_{t_2}(\gamma_{t_2}))^+\chi_{t_1<t_2}
d\mathcal{L}^2(t_1,t_2)
d\pi_{14}\#\lambda^n(\theta,\gamma)\leq 2C_5T^2\sqrt{\delta t}.
$$
Here $C_5$ is the universal constant given in point (v) of Theorem \ref{3.1t}.
\end{lem}
\begin{proof}
According to the definition of $\lambda^n$, we find
\begin{equation*}
\begin{split}
&\quad\int_{(0,T)^2}\int_{Y\times X_{B_1}}(\tilde{\theta}_{t_1}(\gamma_{t_1})-\tilde{\theta}_{t_2}(\gamma_{t_2}))^+\chi_{t_1<t_2}
d\mathcal{L}^2(t_1,t_2)
d\pi_{14}\#\lambda^n(\theta,\gamma)\\
&=\int_{Y_{B_2}}\mu_0(\theta)\sum_{\sigma\in S_{\theta}}n^n\mu_{1\sigma(1)}^{\theta}
\cdots
\mu_{n\sigma(n)}^{\theta}
\int_{(0,T)^2}\int_{[0,1]}(\bar{\theta}^n_{\sigma,t_1}(F^n_{\sigma,t_1}(z))-\bar{\theta}^n_{\sigma,t_2}(F^n_{\sigma,t_2}(z)))^+\chi_{t_1<t_2}dt_1dt_2dz.
\end{split}
\end{equation*}
For each fixed $\sigma$ and $\theta$, it holds:
\begin{equation*}
\begin{split}
&\int_{(0,T)^2}\int_{[0,1]}(\bar{\theta}^n_{\sigma,t_1}(F^n_{\sigma,t_1}(z))-\bar{\theta}^n_{\sigma,t_2}(F^n_{\sigma,t_2}(z)))^+\chi_{t_1<t_2}dt_1dt_2dz\\
&\leq2T\int_0^T\int_{[0,1]}|\bar{\theta}^n_{\sigma,t}(F^n_{\sigma,t}(z))-\theta^n_{\sigma,t}(F^n_{\sigma,t}(z))|dtdz\\
&+\int_{(0,T)^2}\int_{[0,1]}(\theta^n_{\sigma,t_1}(F^n_{\sigma,t_1}(z))-\theta^n_{\sigma,t_2}(F^n_{\sigma,t_2}(z)))^+\chi_{t_1<t_2}dt_1dt_2dz\\
&=2T\int_0^T\int_{[0,1]}|\bar{\theta}^n_{\sigma,t}(z)-\theta^n_{\sigma,t}(z)|dtdz\leq 2T^2C_5\sqrt{\delta t}.
\end{split}
\end{equation*}
In the above calculation, we used point (iii) of Corollary \ref{2.2c}, hence for any $\sigma\in S_{\theta}$, $\theta^n_{\sigma,t_1}(F^n_{\sigma,t_1})\leq\theta^n_{\sigma,t_2}(F^n_{\sigma,t_2})$, for any $t_1<t_2$.
When estimating $\bar{\theta}^n_{\sigma}-\theta^n_{\sigma}$, we used (\ref{3}) and point (v) of Theorem \ref{3.1t}.
\end{proof}
\begin{lem}\label{5.8l}
For $\lambda-$a.e $(\tilde{\theta},s,z,\gamma)\in Y\times\mathbb{R}\times[0,1]\times X_{B_1}$, we have 
\begin{equation*}
\tilde{\theta}_{t_1}(\gamma_{t_1})\leq\tilde{\theta}_{t_2}(\gamma_{t_2}), \textrm{ for $\mathcal{L}^2-a.e$ $(t_1,t_2)\in(0,T)^2$ with $t_1<t_2$}.
\end{equation*}
\end{lem}
\begin{proof}
It suffices to show that
\begin{equation}\label{4.8}
\int_{(0,T)^2}\int_{Y\times X_{B_1}}(\tilde{\theta}_{t_1}(\gamma_{t_1})-\tilde{\theta}_{t_2}(\gamma_{t_2}))^+\chi_{t_1<t_2}
d\mathcal{L}^2(t_1,t_2)
d\pi_{1,4}\#\lambda(\tilde{\theta},\gamma)=0.
\end{equation}
For any $\eps>0$, we can write
\begin{equation*}
\begin{split}
&\int_{(0,T)^2}\int_{Y\times X_{B_1}}(\tilde{\theta}_{t_1}(\gamma_{t_1})-\tilde{\theta}_{t_2}(\gamma_{t_2}))^+\chi_{t_1<t_2}
d\mathcal{L}^2(t_1,t_2)
d\pi_{1,4}\#\lambda(\tilde{\theta},\gamma)\\
&\leq 2T\int_0^T\int_{Y\times X_{B_1}}(\tilde{\theta}_t(\gamma_t)
-(A_{\eps}\tilde{\theta})_t
(\gamma_t))^+dtd\pi_{14}
\#\lambda(\tilde{\theta},\gamma)\\
&+\int_{(0,T)^2}\int_{Y\times X_{B_1}}((A_{\eps}\tilde{\theta})_{t_1}(\gamma_{t_1})-(A_{\eps}\tilde{\theta})_{t_2}(\gamma_{t_2}))^+\chi_{t_1<t_2}
d\mathcal{L}^2(t_1,t_2)
d\pi_{14}\#\lambda(\tilde{\theta},\gamma).
\end{split}
\end{equation*}
For the first term, it goes to zero as $\eps\rightarrow0$, since the integrant tends to 0 for each fixed $(t,\tilde{\theta},\gamma)$ and is clearly bounded. For the second term, we estimate:
\begin{equation*}
\begin{split}
&\int_{(0,T)^2}\int_{Y\times X_{B_1}}((A_{\eps}\tilde{\theta})_{t_1}(\gamma_{t_1})-(A_{\eps}\tilde{\theta})_{t_2}(\gamma_{t_2}))^+\chi_{t_1<t_2}
d\mathcal{L}^2(t_1,t_2)
d\pi_{14}\#\lambda^n(\tilde{\theta},\gamma)\\
&\leq 2T\int_0^T\int_{Y\times X_{B_1}}|\tilde{\theta}_{t}(\gamma_{t})-(A_{\eps}\tilde{\theta})_{t}(\gamma_{t})|dtd\pi_{14}\#\lambda^n(\tilde{\theta},\gamma)+2C_5T^2\sqrt{\delta t}\leq 2B_3T^2\eps+2C_5T^2\sqrt{\delta t}.
\end{split}
\end{equation*}
The last inequality follows from the Lemma \ref{6.4ln}, 
while the first inequality used Lemma \ref{5.7l}.
Now send $n\rightarrow\infty$, Lemma \ref{5.6l} allows us to conclude:
$$
\int_{(0,T)^2}\int_{Y\times X_{B_1}}((A_{\eps}\tilde{\theta})_{t_1}(\gamma_{t_1})-(A_{\eps}\tilde{\theta})_{t_2}(\gamma_{t_2}))^+\chi_{t_1<t_2}
d\mathcal{L}^2(t_1,t_2)
d\pi_{14}\#\lambda(\tilde{\theta},\gamma)\leq 2B_3T^2\eps.
$$
The proof is completed by sending $\eps\rightarrow0$.
\end{proof}

Up to now, we have checked points (i)-(iii) in Definition \ref{1.6d}. It only remains to check point (iv). 
Our first goal will be to show the measure $\partial_t(\tilde{\theta}_t(\gamma_t))$ is concentrated on the "wet" set.
As preparation, we prove the following:
\begin{lem}
For any $\eps,\,\eps_1,\,\eps_2>0$, define the function $K_i:Y\times\mathbb{R}\times X_{B_1}\rightarrow\mathbb{R}$, $i=1,2$, given by 
\begin{equation*}
\begin{split}
&K_1(\tilde{\theta},s,\gamma)=\int_{(0,T)^2}\chi_{\{0<t_2-t_1<\eps_1\}}|(A_{\eps}\tilde{\theta})_{t_1}(\gamma_{t_1})-(A_{\eps}\tilde{\theta})_{t_2}(\gamma_{t_2})|\chi_{\{(A_{\eps}\tilde{\theta})_{t_1}(\gamma_{t_1})>\Theta(s,\gamma_{t_1},t_1)+\eps_2\}}
dt_1dt_2,\\
&K_2(\tilde{\theta},s,\gamma)=\int_{(0,T)^2}\chi_{\{0<t_2-t_1<\eps_1\}}(\gamma_{t_2}-\gamma_{t_1})^+
\chi_{\{(A_{\eps}\tilde{\theta})_{t_1}(\gamma_{t_1})>\Theta(s,\gamma_{t_1},t_1)+\eps_2\}}.
\end{split}
\end{equation*}
Then $K_i$ is lower semi-continuous, $i=1,2$.
\begin{proof}
Let $(\tilde{\theta}^k,s^k,\gamma^k)\rightarrow(\tilde{\theta},s,\gamma)$, need to show $K_i(\tilde{\theta},s,\gamma)\leq
\lim\inf_{k\rightarrow\infty}K_i(\tilde{\theta}^k,s^k,\gamma^k)$.

As explained in the proof of Lemma \ref{5.6l}, for any $t$ such that $\gamma^k(t)\rightarrow\gamma(t)$, we have $(A_{\eps}\tilde{\theta}^k)_t
(\gamma^k_{t})\rightarrow (A_{\eps}\tilde{\theta})_t(\gamma_t)$. Hence for such $t$
$$
\chi_{\{(A_{\eps}\tilde{\theta})_{t}(\gamma_t)>\Theta(s,\gamma_t,t)+\eps_2\}}\leq
\lim\inf_{k\rightarrow\infty}
\chi_{\{(A_{\eps}\tilde{\theta}^k)_{t}(\gamma^k_t)>\Theta(s,\gamma^k_t,t)+\eps_2\}}.$$
The integrand is lower semi-continuous with respect to $(\tilde{\theta},s,\gamma)$ if we fix $(t_1,t_2)$ such that $\gamma^k_{t_i}
\rightarrow\gamma_{t_i}$, $i=1,2$. 
Then the lower semi-continuity of $J_i$ follows from applying Fatou`s lemma.
\end{proof}
\end{lem}
\sloppy
\begin{lem}\label{5.10l}
Let $C_2$ be the universal constant given by Theorem \ref{3.1t}, point (ii). Then we have for any $\eps>0$
\begin{equation}\label{5.16}
\int_{[0,T]^2}\chi_{\{0<t_2-t_1<\frac{\eps}{C_2}\}}dt_1dt_2\int|\tilde{\theta}_{t_2}(\gamma_{t_2})-\tilde{\theta}_{t_1}(\gamma_{t_1})|\chi_{\{\tilde{\theta}_{t_1}(\gamma_{t_1})>\Theta(s,\gamma_{t_1},t_1)+2\eps\}}d\lambda(\tilde{\theta}
,s,z,\gamma)=0.
\end{equation}
and 
\begin{equation}\label{5.17}
\int_{[0,T]^2}\chi_{\{0<t_2-t_1<\frac{\eps}{C_2}\}}dt_1dt_2\int(
\gamma_{t_2}-\gamma_{t_1})^+
\chi_{\{\tilde{\theta}_{t_1}(\gamma_{t_1})>
\Theta(s,\gamma_{t_1},t_1)+2\eps\}}
d\lambda(\tilde{\theta},s,z,\gamma)=0.
\end{equation}
\end{lem}
\begin{proof}
We only prove (\ref{5.16}). The proof of (\ref{5.17}) follows similar lines and is simpler. Fix $0<\delta<\eps$. Denote $\chi_{\delta,\eps}=\chi_{\{(A_{\delta}\tilde{\theta})_{t_1}(\gamma_{t_1})>\Theta(s,\gamma_{t_1},t_1)+\eps\}}$. From the lower semi-continuity proved in previous lemma, we conclude:
\begin{equation}\label{5.18}
\begin{split}
&\int_{[0,T]^2}\chi_{\{0<t_2-t_1<\frac{\eps}{C_2}\}}dt_1dt_2\int|(A_{\delta}\tilde{\theta})_{t_2}(\gamma_{t_2})-(A_{\delta}\tilde{\theta})_{t_1}(\gamma_{t_1})|\chi_{\delta,2\eps}d\lambda(\theta
,s,z,\gamma)\\
&\leq\lim\inf_{n\rightarrow\infty}\int_{[0,T]^2}\chi_{\{0<t_2-t_1<\frac{\eps}{C_2}\}}dt_1dt_2\int|(A_{\delta}\tilde{\theta})_{t_2}(\gamma_{t_2})-(A_{\delta}\tilde{\theta})_{t_1}(\gamma_{t_1})|\chi_{\delta,2\eps}d\lambda^n(\tilde{\theta}
,s,z,\gamma).
\end{split}
\end{equation}
We know the left hand side of (\ref{5.18}) will tend to the left hand side of (\ref{5.16}) as $\delta\rightarrow0$. Next we estimate the right hand side of (\ref{5.18}).
\begin{equation}\label{6.21newnew}
\begin{split}
&\int_{[0,T]^2}\chi_{\{0<t_2-t_1<\frac{\eps}{C_2}\}}dt_1dt_2\int|(A_{\delta}\tilde{\theta})_{t_2}(\gamma_{t_2})-(A_{\delta}\tilde{\theta})_{t_1}(\gamma_{t_1})|\chi_{\delta,2\eps}d\lambda^n(\tilde{\theta}
,s,z,\gamma)\\
&\leq2T\int_0^T\int_{Y\times X_{B_1}}|(A_{\delta}\tilde{\theta})_t(\gamma_t)-
\tilde{\theta}_t(\gamma_t)|d\pi_{14}\#\lambda^n(\tilde{\theta},\gamma)\\
&+\int_{[0,T]^2}\chi_{\{0<t_2-t_1<\frac{\eps}{C_2}\}}dt_1dt_2\int|\tilde{\theta}_{t_2}(\gamma_{t_2})-\tilde{\theta}_{t_1}(\gamma_{t_1}))|\chi_{\delta,2\eps}d\lambda^n(\tilde{\theta}
,s,z,\gamma)\\
&\leq2B_3T^2\delta+\int_{[0,T]^2}\chi_{\{0<t_2-t_1<\frac{\eps}{C_2}\}}dt_1dt_2\int|\tilde{\theta}_{t_2}(\gamma_{t_2})-\tilde{\theta}_{t_1}(\gamma_{t_1}))|\chi_{\{\tilde{\theta}_{t_1}(\gamma_{t_1}))>
\Theta(s,\gamma_{t_1},t_1)+1.5\eps\}}d\lambda^n\\
&+B_3T\int_0^Tdt_1\int(\chi_{\delta,2\eps}
-\chi_{\{\tilde{\theta}_{t_1}(\gamma_{t_1}))>
\Theta(s,\gamma_{t_1},t_1)+1.5\eps\}})^+d\lambda^n.
\end{split}
\end{equation}
In the second inequality, we used Lemma \ref{6.4ln}. For the second term of the right hand side above, we calculate:
\begin{equation}\label{6.21new}
\begin{split}
&\int_{[0,T]^2}\chi_{\{0<t_2-t_1<\frac{\eps}{C_2}\}}dt_1dt_2\int|\tilde{\theta}_{t_2}(\gamma_{t_2})-\tilde{\theta}_{t_1}(\gamma_{t_1}))|\chi_{\{\tilde{\theta}_{t_1}(\gamma_{t_1}))>
\Theta(s,\gamma_{t_1},t_1)+1.5\eps\}}
d\lambda^n\\
&=\int_{Y_{B_2}}d\mu_0(\theta)\sum_{\sigma\in S_{\theta}}n^n\mu_{1\sigma(1)}^{\theta}
\cdots
\mu_{n\sigma(n)}^{\theta}
\int_{(0,T)^2}\int_0^1\chi_{\{0<t_2-t_1<\frac{\eps}{C_2}\}}|\bar{\theta}^n_{\sigma,t_2}(F^n_{\sigma,t_2}(z))-\bar{\theta}^n_{\sigma,t_1}(F^n_{\sigma,t_1}(z))|\\
&\chi_{\{\bar{\theta}^n_{\sigma,t_1}(F^n_{\sigma,t_1}(z))>
\Theta(\theta^{M,n}_{0,\sigma}(z),F^n_{\sigma,t_1}(z),t_1)+1.5\eps\}}dt_1dt_2dz.
\end{split}
\end{equation}
Now define: 
\begin{equation*}
\begin{split}
&\chi_1=\chi_{\{\bar{\theta}^n_{\sigma,t_1}(F^n_{\sigma,t_1})>\Theta(\theta^{M,n}_{0,\sigma}(z),F_{\sigma,t_1}^n(z),t_1)+1.5\eps\}},\\
&\chi_2=\chi_{\{\theta^n_{\sigma,t_1}(F^n_{\sigma,t_1})>\Theta(\theta^{M,n}_{0,\sigma}(z),F_{\sigma,t_1}^n(z),t_1)+\eps\}}.
\end{split}
\end{equation*}
We now estimate the right hand side of (\ref{6.21new}). For any $\theta\in Y_{B_2}$ and $\sigma\in S_{\theta}$, 
\begin{equation*}
\begin{split}
&\int_{(0,T)^2}\int_0^1\chi_{\{0<t_2-t_1<\frac{\eps}{C_2}\}}|\bar{\theta}^n_{\sigma,t_2}(F^n_{\sigma,t_2}(z))-\bar{\theta}^n_{\sigma,t_1}(F^n_{\sigma,t_1}(z))|\chi_1dt_1dt_2dz\\
&\leq2T\int_0^T\int_0^1|\bar{\theta}^n_{\sigma,t}(F^n_{\sigma,t}(z))-\theta^n_{\sigma,t}(F^n_{\sigma,t}(z))|dtdz\\
&+\int_{(0,T)^2}\int_0^1\chi_{\{0<t_2-t_1<\frac{\eps}{C_2}\}}|\theta^n_{\sigma,t_2}(F^n_{\sigma,t_2}(z))-\theta^n_{\sigma,t_1}(F^n_{\sigma,t_1}(z))|\chi_1dt_1dt_2dz\\
&\leq 2T^2C_5\sqrt{\delta t}+\int_{(0,T)^2}\int_0^1\chi_{\{0<t_2-t_1<\frac{\eps}{C_2}\}}|\theta^n_{\sigma,t_2}(F^n_{\sigma,t_2}(z))-\theta^n_{\sigma,t_1}(F^n_{\sigma,t_1}(z))|\chi_2dt_1dt_2dz\\
&+2B_1\int_{(0,T)^2}
\int_0^1(\chi_1-\chi_2)^+dt_1dt_2dz.
\end{split}
\end{equation*}
In the last inequality, we used the point (v) of Theorem \ref{3.1t}.
The second term of right hand side above is 0, due to point (iv) of Theorem \ref{3.1t}.
To estimate the last term, we notice
\begin{equation*}
\begin{split}
&\int_{(0,T)^2}\int_0^1(\chi_1-\chi_2)^+
dt_1dt_2dz\leq\int_{(0,T)^2}\int_0^1\chi_{\{\bar{\theta}^n_{\sigma,t_1}(F^n_{\sigma,t_1}(z))-\theta^n_{\sigma,t_1}(F^n_{\sigma,t_1}(z))>0.5\eps\}}\\
&\leq\frac{2T}{\eps}\int_0^T\int_0^1|\bar{\theta}^n_{\sigma,t}(z)-\theta^n_{\sigma,t}(z)|dzdt\leq\frac{2TC_5\sqrt{\delta t}}{\eps}.
\end{split}
\end{equation*}

For the last term of (\ref{6.21newnew}), we have
\begin{equation*}
\begin{split}
&\int_0^Tdt_1
\int(\chi_{\delta,2\eps}-
\chi_{\{\tilde{\theta}_{t_1}(\gamma_{t_1}))>
\Theta(s,\gamma_{t_1},t_1)+\eps\}})^+
d\lambda^n\\
&\leq\int_0^Tdt_1\int\chi_{\{|(A_{\delta}\tilde{\theta})_{t_1}(\gamma_{t_1})-\tilde{\theta}_{t_1}(\gamma_{t_1}))|\geq\eps\}}d\lambda^n\leq\eps^{-1}
\int_0^Tdt_1\int|(A_{\delta}\tilde{\theta})_{t_1}(\gamma_{t_1})-\tilde{\theta}_{t_1}(\gamma_{t_1}))|d\lambda^n\\
&\leq\eps^{-1}B_3T\delta.
\end{split}
\end{equation*}
In the last inequality, we use Lemma \ref{6.4ln} again. 
Combining the calculations above, we obtain the left hand side of (\ref{5.18}) $\leq\eps^{-1}B_3T\delta+2B_3T^2\delta$.
The proof follows from sending $\delta\rightarrow0$.
\end{proof}

By Remark \ref{1.4r}, and Lemma \ref{5.8l}, we know for $\pi_{14}\#\lambda-a.e$ $(\theta,\gamma)$, one can determine a unique monotone increasing, and left continuous function $(\theta(\gamma))^*$, which equals $\theta_t(\gamma_t)$ for $\mathcal{L}^1-a.e$ $t\in(0,T)$. We will simply denote this function by $\alpha(t)$ in the following lemma.

\begin{lem}\label{5.11l}
Let $\eps>0$, then for $\lambda-$a.e $(\tilde{\theta},s,z,\gamma)\in Y\times X_{B_1}$, the following property holds:

For any $t\in(0,T)$ such that $\alpha(t)>\Theta(s,\gamma_t,t)+\eps$, it holds $\alpha(t^{\prime})=\alpha(t)$, and $\gamma_{t^{\prime}}\leq\gamma_t$, for any $0<t^{\prime}-t<\frac{\eps}{2C_2}$. Here $C_2$ is the constant given by point (ii), Theorem \ref{3.1t}.
\end{lem}
\begin{proof}
Let $(\theta,\gamma)$ be chosen so that the statement of Lemma \ref{5.8l} holds true, so that we can define $\alpha(t)$. Let $(\tilde{\theta},s,z,\gamma)$ also satisfy that 
\begin{equation}\label{6.23new}
\int_{[0,T]^2}\chi_{\{0<t_2-t_1<\frac{\eps}{2C}\}}|\tilde{\theta}_{t_2}(\gamma_{t_2})-\tilde{\theta}_{t_1}(\gamma_{t_1}))|\chi_{\{\tilde{\theta}_{t_1}(\gamma_{t_1}))>\Theta(s,\gamma_{t_1},t_1)+\eps\}}dt_1dt_2=0,
\end{equation}
and
\begin{equation}\label{6.24new}
\int_{[0,T]^2}\chi_{\{0<t_2-t_1<\frac{\eps}{2C}\}}dt_1dt_2\int(
\gamma_{t_2}-\gamma_{t_1})^+
\chi_{\{\tilde{\theta}_{t_1}(\gamma_{t_1}))>
\Theta(s,\gamma_{t_1},t_1)+\eps\}}=0.
\end{equation}
We know from Lemma \ref{5.10l} that (\ref{6.23new}) and (\ref{6.24new}) holds for $\lambda-$a.e $(\tilde{\theta},s,z,\gamma)$, 
for $\lambda-a.e$ $(\tilde{\theta},s,z,\gamma)$. Choose some $(\tilde{\theta},s,z,\gamma)$ so that (\ref{6.23new}) and (\ref{6.24new}) hold, then we have $\tilde{\theta}_{t_2}(\gamma_{t_2})=\tilde{\theta}_{t_1}(\gamma_{t_1}))$ and $\gamma_{t_2}\leq\gamma_{t_1}$, for $\mathcal{L}^2-a.e$ $(t_1,t_2)$ with $\tilde{\theta}_{t_1}(\gamma_{t_1}))>\Theta(s,\gamma_{t_1},t_1)+\eps$ and $0<t_2-t_1<\frac{\eps}{2C}$.

Fix any $(t_1,t_2)$ with $0<t_2-t_1<\frac{\eps}{2C_2}$ and $\alpha(t_1)>\Theta(s,\gamma_{t_1},t_1)+\eps$. By left continuity, there exists $\delta >0$, such that $\alpha(t_1')>\Theta(s,\gamma_{t_1'},t_1')+\eps$ for any $t_1'\in(t_1-\delta,t_1)$. 
By Fubini`s theorem, we know for $\mathcal{L}^1-$a.e $t_1'\in(t_1-\delta ,t_1)$, 
we have $\alpha(t'_2)=\tilde{\theta}_{t'_2}(\gamma_{t'_2})=\tilde{\theta}_{t'_1}(\gamma_{t'_1})$ and $\gamma_{t'_2}\leq\gamma_{t'_1}$ for $\mathcal{L}^1-a.e$ $t'_2\in(t_1',t_1'+\frac{\eps}{2C_2})$.
By left contiuity of $\alpha(t)$ and $\gamma(t)$,
 we conclude $\alpha(t'_2)=\tilde{\theta}_{t'_1}(\gamma_{t'_1})$ and $\gamma_{t'_2}\leq\gamma_{t'_1}$ for $any$ $t'_2\in(t_1',t_1'+\frac{\eps}{2C_2})$. This is true for $\mathcal{L}^1-a.e$ $t_1'\in(t_1-\delta,t_1)$.
Hence we can find a sequence $\{t_1^n\}_{n=1}^{\infty}\subset(t_1-\delta,t_1)$ such that this is true for $t_1^n$ and $t_1^n\rightarrow t_1$ as $n\rightarrow\infty$.
We can assume $\tilde{\theta}_{t_1^n}(\gamma_{t_1^n})=\alpha(t_1^n)$ holds.
Define $t_2^n=t_1^n+t_2-t_1$, then $t_2^n\rightarrow t_2$, and we have $\alpha(t_2^n)=\alpha(t_1^n)$ and $\gamma_{t_2^n}\leq\gamma_{t_1^n}$.
Let $n\rightarrow\infty$, and use left continuity of $\alpha$ and $\gamma$ one more time, we can conclude $\alpha(t_1)=\alpha(t_1)$ and $\gamma_{t_2}\leq\gamma_{t_1}$.
\end{proof}
\begin{lem}
For any $C>0$, $\eps>0$, define the function $H_{\eps}:Y\times\mathbb{R}\times X_{B_1}\rightarrow\mathbb{R}$, given by
\begin{equation*}
\begin{split}
H_{\eps}(\tilde{\theta},s,\gamma)&=\int_{(0,T)^2}\bigg(|(A_{\eps}\tilde{\theta})_{t_2}
(\gamma_{t_2})
-(A_{\eps}\tilde{\theta})_{t_1}(\gamma_{t_1})-(\Theta(s,\gamma_{t_2},t_2)\\
&-\Theta(s,\gamma_{t_1},t_1))^+|-C(t_2-t_1)\bigg)^+\chi_{t_1<t_2}dt_1dt_2.
\end{split}
\end{equation*}
Then $H_{\eps}$ is continuous.
\begin{proof}
The proof of this lemma is quite similar to Lemma \ref{5.8l}. We already noted that $(\tilde{\theta}^k,\gamma^k)\rightarrow(\tilde{\theta},\gamma)$ implies $(A_{\eps}\tilde{\theta}^k)_t
(\gamma^k_t)\rightarrow
(A_{\eps}\tilde{\theta})_t(\gamma_t)$ for any $t$ such that pointwise convergence of $\gamma^k_t$ happens. 
The proof follows then from dominated convergence since everything is bounded.
\end{proof}
\end{lem}
\begin{lem}
Let $C_4$ be the constant given in point (iv), Theorem \ref{3.1t}, then we have:
\begin{equation}\label{5.19}
\int_{(0,T)^2}\int\bigg(|\tilde{\theta}_{t_2}(\gamma_{t_2})-\tilde{\theta}_{t_1}(\gamma_{t_1}))-(\Theta(s,\gamma_{t_2},t_2)-\Theta(s,\gamma_{t_1},t_1))^+|-C_4(t_2-t_1)\bigg)^+\chi_{t_1<t_2}dt_1dt_2d\lambda=0.
\end{equation}
\end{lem}
\begin{proof}
Write the left hand side of (\ref{5.19}) to be $\int H(\tilde{\theta},s,\gamma)d\lambda$, with the definition of $H$ similar to $H_{\eps}$ in the last lemma(without $A_{\eps}$). Then we can estimate
\begin{equation*}
\begin{split}
\int_{Y\times\mathbb{R}\times[0,1]\times X_{B_1}}H(\tilde{\theta},s,\gamma)d\lambda(\tilde{\theta},s,z,\gamma)=\int (H-H_{\eps})(\tilde{\theta},s,\gamma)d\lambda+\lim_{n\rightarrow\infty}\int H_{\eps}(\tilde{\theta},s,\gamma)d\lambda^n.
\end{split}
\end{equation*}
The first term will tend to zero as $\eps\rightarrow0$. For the second term, we have
\begin{equation*}
\begin{split}
\int H_{\eps}(\tilde{\theta},s,\gamma)d\lambda^n&\leq\int|H_{\eps}-H|d\lambda^n+\int H(\tilde{\theta},s,\gamma)d\lambda^n\\
&\leq2T\int_0^T\int|\tilde{\theta}_t(\gamma_t)-
(A_{\eps}\tilde{\theta})_t(\gamma_t)|d\lambda^n+\int H(\tilde{\theta},s,\gamma)d\lambda^n\\
&\leq 2T^2B_3\eps+\int H(\tilde{\theta},s,\gamma)d\lambda^n.
\end{split}
\end{equation*}
In the last inequality above, we used Lemma \ref{6.4ln}. To deal with the remaining term, first we can write: 
$$
\int H(\tilde{\theta},s,\gamma)d\lambda^n=\int_{Y_{B_2}}
d\mu_0(\theta)\sum_{\sigma\in S_{\theta}}n^n\mu_{1\sigma(1)}^{\theta}
\cdots\mu_{n\sigma(n)}^{\theta}\int_{[0,1]}H(\bar{\theta}^n_{\sigma},\theta^{M,n}_{0,\sigma}(z),F^n_{\sigma}(z))dz.
$$
Fix some $\theta\in Y_{B_2}$, and $\sigma\in S_{\theta}$, we can calculate:
\begin{equation*}
\begin{split}
\int H(\bar{\theta}^n_{\sigma}&,\theta^{M,n}_{\sigma}(z),F^n_{\sigma}(z))dz\leq 2T\int_0^T\int_{[0,1]}|\bar{\theta}^n_{\sigma,t}(F^n_{\sigma,t}(z))-\theta^n_{\sigma,t}(F^n_{\sigma,t}(z))|dtdz\\
&+\int H(\theta^n_{\sigma},\theta^{M,n}_{\sigma}(z),F_{\sigma}^n(z))dz\leq 2T^2C_5\sqrt{\delta t}+T^2C_4\delta t.
\end{split}
\end{equation*}
In the second inequality above, we used the point (iv) and (v) of Theorem \ref{3.1t}, and that $F^n_{\sigma,t}$ is measure preserving. 

Now the proof is finished by first letting $n\rightarrow\infty$ and then let $\eps\rightarrow0$.
\end{proof}

As a corollary, we deduce that 
\begin{cor}\label{5.14c}
For $\lambda-a.e$ $(\tilde{\theta},s,z,\gamma)$, it holds:
$$
|\alpha(t_2)-\alpha(t_1)-(\Theta(s,\gamma_{t_2},t_2)-\Theta(s,\gamma_{t_1},t_1)^+|\leq C_4(t_2-t_1),
$$
for any $0<t_1<t_2<T$. Here $C_4$ is the constant given in point (iv) of Theorem \ref{3.1t}.
\end{cor}

With above preparation, we can check the point (iv) of Theorem \ref{1.6d}.
\begin{prop}
For $\lambda-$a.e $(\tilde{\theta},s,z,\gamma)$, we have the equality of measures:
\begin{equation}\label{5.20}
\partial_t(\theta_t(\gamma_t))=\big[\partial_t
(Q^{sat}(\theta_t(\gamma_t),\gamma_t,t))\big]^-\lfloor E_{\theta,s,\gamma}.
\end{equation}
where $E_{\theta,s,\gamma}$ is the wet set given by
$$
E_{\theta,s,\gamma}=\{t\in(0,T):(\theta(\gamma))^*(t)=s-Q^{sat}((\theta(\gamma))^*(t),\gamma_t,t)\}.
$$
\end{prop}
\begin{proof}
We choose $(\tilde{\theta},s,z,\gamma)$ such that the statement of Lemma 
\ref{5.8l}, \ref{5.11l}, and Corollary \ref{5.14c} hold. The plan is to apply Lemma \ref{6.1} to the functions: $f(t)=\alpha(t)$, and $g(t)=s-Q^{sat}(\alpha(t),\gamma(t),t)$. Here $\alpha(t)$ is the monotone increasing and left continuous version of $\theta_t(\gamma_t)$ chosen according to Remark \ref{1.4r}. This is possible since Lemma \ref{5.8l} holds.

First we verify that for $\lambda-a.e$ $(\tilde{\theta},s,z,\gamma)$,
 it holds $\alpha(t)\geq s-Q^{sat}(\alpha(t),\gamma(t),t)$, for any $t\in(0,T)$.
 This is the same as $\alpha(t)\geq\Theta(s,\gamma(t),t)$. This follows from Lemma \ref{5.5l}. Indeed, we have shown there that for any $t\in (0,T)$, $\theta(z^{\prime})\geq\Theta(s,z^{\prime},t)$ for $\zeta_t-a.e$ 
$(\theta,s,z^{\prime})$.
Recall the definition of $\zeta_t$, this is the same as saying for any fixed $t\in(0,T)$, $\theta_t(\gamma_t)\geq\Theta(s,\gamma_t,t)$ for $\lambda-a.e$ $(\tilde{\theta},s,z,\gamma)$. 
Now choose a countable dense subset $D$ of the set $\{t\in(0,T):\alpha(t)=\theta_t(\gamma_t)\}$.
Then $D\subset(0,T)$ is also dense.
Then for $\lambda-a.e$ $(\tilde{\theta},s,z,\gamma)$, it holds $\alpha(t_i)\geq\Theta(s,\gamma_{t_i},t_i)$, for any $t_i\in D$. Since $D$ is dense and both $\alpha(t)$, $\gamma(t)$ continuous from left, we see it is true for all $t\in(0,T)$.

Next we verify points (i)-(iii) in Lemma \ref{6.1}.

The point (i) follows from Lemma \ref{5.11l}. 
Indeed, if  $\alpha(t)>s-Q^{sat}(\alpha(t),\gamma(t),t)+\eps$, 
then we know $\alpha(t)>\Theta(s,\gamma(t),t)+\frac{\eps}{C_7}$ for some universal constant $C_7$.
 Now Lemma \ref{5.11l} implies that $\alpha(t^{\prime})=\alpha(t)$, $\gamma(t^{\prime})\leq\gamma(t)$, for any $0<t^{\prime}-t<\frac{\eps}{C_8}$.
Hence $\alpha(t^{\prime})=\alpha(t)>s-Q^{sat}(\alpha(t^{\prime}),\gamma(t^{\prime}),t^{\prime})-\sup|\partial_tQ^{sat}|(t^{\prime}-t)+\eps$.
If we still have $\eps>2\sup|\partial_tQ^{sat}|(t^{\prime}-t)$, then we have $f(t^{\prime})>g(t^{\prime})$. 
Hence $f(t')=f(t)$ and $f(t')>g(t')$ as long as $0<t'-t<\frac{\eps}{C_8+2\sup|\partial_tQ^{sat}|}$.
This verifies point (i).

Next we verify point (ii). This follows from Corollary \ref{5.14c}. 
Indeed, from that corollary,
 we can deduce $\alpha(t^+)-\alpha(t)=(\Theta(s,\gamma(t^+),t)-\Theta(s,\gamma(t),t))^+$, by fixing any $t_1=t$, and $t_2\searrow t$.
 If $\alpha(t^+)=\alpha(t)$, we conclude $\Theta(s,\gamma(t^+),t)\leq\Theta(s,\gamma(t),t)$. 
Therefore $\gamma(t^+)\leq\gamma(t)$ by strict monotonicity of $\Theta$ in $z$ variable. Then $Q^{sat}(\alpha(t^+),\gamma(t^+),t)\geq Q^{sat}(\alpha(t),\gamma(t),t)$.
If $\alpha(t^+)>\alpha(t)$, we must have $\alpha(t)=\Theta(s,\gamma(t),t)$, or $\alpha(t)=s-Q^{sat}(\alpha(t),\gamma(t),t)$, otherwise it contradicts Lemma \ref{5.11l}. 
But then $\alpha(t^+)=\Theta(s,\gamma(t^+),t)$.
Hence $\alpha(t^+)=s-Q^{sat}(\alpha(t^+),\gamma(t^+),t)$. In any case, we have $f(t^+)-f(t)=(g(t^+)-g(t))^+$.
This verifies point (ii).

Point (iii) again follows from Lemma \ref{5.11l}.
Indeed, from the monotonicity of $Q^{sat}$, we just need to show for any $[a,b)\subset\{f>g\}$,
 it holds $\alpha(t_2)=\alpha(t_1)$,
 and $\gamma(t_2)\leq\gamma(t_1)$, for any $t_1<t_2$, $t_1,\,t_2\in[a,b)$.
 Let $t^*=\sup\{t\in[t_1,t_2]:\alpha(t^{\prime})=\alpha(t_1),\,\gamma(t^{\prime})\leq\gamma(t_1),\textrm{ for any $t^{\prime}\leq t$}\}$. Then we must have $t^*=t_2$. Otherwise, since $f(t^*)>g(t^*)$, Lemma \ref{5.11l} allows us to push beyond $t_*$, giving a contradiction.

\end{proof}

\section{appendix}
\begin{lem}\label{6.1}
Let $f:(0,T)\rightarrow\mathbb{R}$ be monotone increasing, and $g:(0,T)\rightarrow\mathbb{R}\in BV((0,T))$, both continuous from left and bounded, with $f(t)\geq g(t)$, $\forall t\in(0,T)$. Suppose that for some constant $C>0$
\begin{equation*}
\begin{split}
&(i)\textrm{$f(t^{\prime})\equiv f(t)$, $f(t^{\prime})>g(t^{\prime})$ for any $t\in(0,T)$ with $f(t)>g(t)+\eps$ and any $t^{\prime}$ with $t^{\prime}-t<\frac{\eps}{C}$.}\\
&(ii)\textrm{$f(t^+)-f(t)=[g(t^+)-g(t)]^+$, for any $t\in(0,T)$.}\\
&(iii)\textrm{$g(t_2)-g(t_1)\leq C(t_2-t_1)$ for any $t_1<t_2\in[a,b)$ with $[a,b)\subset\{f>g\}$.}
\end{split}
\end{equation*} 
Then $\partial_t f$ is concentrated on the set $\{f=g\}$ and for any Borel set $E\subset\{f=g\}$, one has $\partial_tf(E)=(\partial_tg)^+(E)$. 
\end{lem}
Before we prove this result, we prove the following lemma as a preparation.
\begin{lem}\label{6.2}
Under the assumptions of previous lemma, we have for any $[a,b)\subset(0,T)$
\begin{equation}
(\partial_tg)^+([a,b))\leq f(b)-f(a)+C\mathcal{L}^1([a,b)
\cap\{f>g\}).
\end{equation}
\end{lem}
\begin{proof}
First recall
$$
(\partial_tg)^+([a,b))=\sup\big\{\sum_{i=1}^n(g(t_i)-g(t_{i-1}))^+:a=t_0<t_1<\cdots<
t_n=b\big\}.
$$
We fix a partition appearing in the right hand side above. For each $i$, define
\begin{equation*}
\begin{split}
&t_{i-1}^{\prime}=\sup\{t\in[t_{i-1},t_i):[t_{i-1},t)\cap\{f=g\}=\emptyset\},\\
&t_i^{\prime}=\inf\{t\in[t_{i-1},t_i):[t,t_i)\cap\{f=g\}=\emptyset\}.
\end{split}
\end{equation*}
Let $D$ be the set of $i$ for which $[t_{i-1},t_i)\cap\{f=g\}\neq\emptyset$. Then for $i\in D$ one can calculate:

If $t_i>t_i^{\prime}$, choose $\eps<t_i-t_i^{\prime}$, since $[t_i^{\prime}+\eps,t_i)\subset\{f>g\}$, we see from point (iii):
\begin{equation*}
\begin{split}
g(t_i)-g(t_{i-1}^{\prime})&=g(t_i)-g(t_i^{\prime}+\eps)+g(t_i^{\prime}+\eps)-g(t_{i-1}^{\prime})\\
&\leq C(t_i-t_i^{\prime}-\eps)+g(t_i^{\prime}+\eps)-g(t_{i-1}^{\prime}).
\end{split}
\end{equation*}
Let $\eps\rightarrow0$, we see 
$$
g(t_i)-g(t_{i-1}^{\prime})\leq C(t_i-t_i^{\prime})+g((t_i^{\prime})^+)-g(t_{i-1}^{\prime}).
$$
On the other hand, since $[t_{i-1},t_{i-1}^{\prime})\subset\{f>g\}$, we can conclude 
$$
g(t_{i-1}^{\prime})-g(t_{i-1})\leq C(t_{i-1}-t_{i-1}^{\prime}).
$$
Combining above calculations, we get:
\begin{equation*}
\begin{split}
&\sum_{i=1}^n(g(t_i)-g(t_{i-1}))^+=\sum_{i\in D}(g(t_i)-g(t_{i-1}))^++\sum_{i\notin D}(g(t_i)-g(t_{i-1}))^+\\
\leq\sum_{i\in D,t_i>t_i^{\prime}}C(t_i-t_i^{\prime})+&\sum_{i\in D}C(t_{i-1}^{\prime}-t_{i-1})+\sum_{i\notin D}C(t_i-t_{i-1})\\
&+\sum_{i\in D,t_i>t_i^{\prime}}(g((t_i^{\prime})^+)-g(t_{i-1}^{\prime}))^++\sum_{i\in D,t_i=t_i^{\prime}}(g(t_i)-g(t_{i-1}^{\prime}))^+
\end{split}
\end{equation*}
$$\leq C\mathcal{L}^1([a,b)\cap\{f>g\})+\sum_{i\in D,t_i>t_i^{\prime}}(g((t_i^{\prime})^+)-g(t_{i-1}^{\prime}))^++\sum_{i\in D,t_i=t_i^{\prime}}(g(t_i)-g(t_{i-1}^{\prime}))^+.$$
In the first inequality above, we used condition (iii). If $i\in D$ and $t_i>t_i^{\prime}$, we will have
$$
g((t_i^{\prime})^+)-g(t_{i-1}^{\prime})\leq f(t_i)-f(t_{i-1}^{\prime}).
$$
Similarly for $i\in C$ and $t_i=t_i^{\prime}$, we see
$$
g(t_i)-g(t_{i-1}^{\prime})\leq f(t_i)-f(t_{i-1}^{\prime}).
$$
Hence
\begin{equation*}
\begin{split}
\sum_{i=1}^n(g(t_i)-g(t_{i-1}))^+&\leq C\mathcal{L}^1([a,b)\cap\{f>g\})+\sum_{i\in D}(f(t_i)-f(t_{i-1}')\\
&\leq C\mathcal{L}^1([a,b)\cap\{f>g\})+f(b)-f(a)
\end{split}
\end{equation*}
So the desired result follows.

\end{proof}
\begin{cor}
For any $(a,b)\subset(0,T)$,
\begin{equation}
(\partial_tg)^+((a,b))\leq f(b)-f(a^+)+C\mathcal{L}^1((a,b)\cap\{f>g\}).
\end{equation}
\end{cor}
\begin{proof}
Apply previous lemma to $[a+\eps,b)$, and send $\eps\rightarrow0$.
\end{proof}

Now we can prove Lemma \ref{6.1} with the help of previous lemma.
\begin{proof}
Here $\partial_tf$ is the Radon measure defined on $[0,T)$ such that $\partial_tf([a,b))=f(b)-f(a)$ for any $[a,b)\subset[0,T)$.

First we show that $\partial_tf$ is concentrated on the set $\{f=g\}$. We show that $\partial_t(\{f>g\})=0$. Indeed, if $t_0\in\{f>g\}$, let $\beta_{t_0}=\sup\{t^{\prime}\geq t_0:f(t^{\prime})=f(t)\}$. From the assumption, we know $\beta_{t_0}>t_0$. Also we know $\{f>g\}=\cup_{f(t)>g(t)}[t,\beta_t)$. Each component of the set $\{f>g\}$ contains a non-degenerate interval, hence there are only countably many components. Therefore, $\{f>g\}=\cup_i C_i$, where $C_i$ are connected components, and has form $[a_i,b_i)$ or $(a_i,b_i)$. It is not hard to see that $f$ remains a constant on $C_i$. Otherwise, by continuity from the left, and also the point(i), one can conclude $f(t)=g(t)$ for some $t\in C_i$, contradiction. So $\partial_tf(C_i)=f(b_i)-f(a_i)=0$, or $\partial_tf(C_i)=f(b_i)-f(a_i^+)=0$.

Next let $E\subset\{f=g\}$ be a Borel set,
 we want to show $\partial_tf(E)=(\partial_tg)^+(E)$. 
First we show $(\partial_tg)^+(E)\leq\partial_tf(E)$.
 Fix $\eps>0$, 
then from the outer regularity of the Radon measure $\partial_tf$ and $\mathcal{L}^1$, we can find an open set $U$, with $E\subset U$, $\partial_tf(U-E)<\eps$, and $\mathcal{L}^1(U-E)<\eps$. Write $U=\cup_i(a_i,b_i)$, with $(a_i,b_i)$ pairwise disjoint. Then we know $\sum_i\partial_tf((a_i,b_i))
\leq\partial_tE+\eps$, and $\sum_i\mathcal{L}^1((a_i,b_i)-E)<\eps$.

Then from previous corollary, we know
\begin{equation*}
\begin{split}
(\partial_tg)^+(E)\leq\sum_i
(\partial_tg)^+((a_i,b_i))&\leq \sum_i(f(b_i)-f(a_i^+))+
\sum_iC\mathcal{L}^1
((a_i,b_i)\cap\{f>g\})\\
&\leq\partial_tf(E)+
\eps+C\sum_i\mathcal{L}^1((a_i,b_i)-E)\leq\partial_tf(E)+2\eps.
\end{split}
\end{equation*}
Since $\eps$ is arbitrary, it follows that $(\partial_tg)^+(E)\leq\partial_tf(E)$.

Now to prove the reverse inequality, Again we choose a cover $E\subset U$, with $U=\cup_i(a_i,b_i)$, $(a_i,b_i)$ pairwise disjoint, such that $(\partial_tg)^+(E)\geq\sum_i(\partial_tg)^+((a_i,b_i))-\eps$.
 
We can assume $(a_i,b_i)\cap E\neq\emptyset$ for each $i$. Denote
\begin{equation*}
\begin{split}
&a_i^{\prime}=\sup\{t\in(a_i,b_i):(a_i,t)\cap\{f=g\}=\emptyset\},\\
&b_i^{\prime}=\inf\{t\in(a_i,b_i):[t,b_i)\cap\{f=g\}=\emptyset\}.
\end{split}
\end{equation*}
For those $i$ with $b_i^{\prime}<b_i$, we can decrease $b_i$, so that $f(b_i)-f((b_i^{\prime})^+)+|g(b_i)-g((b_i^{\prime})^+)|<\eps2^{-i}$. From the left continuity of $f$ and $g$, one has $f(b_i^{\prime})=g(b_i^{\prime})$.
Now if $b_i=b_i^{\prime}$, then 
$$
(\partial_tg)^+((a_i,b_i))\geq g(b_i)-g(a_i^+)\geq f(b_i)-f(a_i^+).
$$
If $b_i>b_i^{\prime}$, then
\begin{equation*}
\begin{split}
(\partial_tg)^+((a_i,b_i))&\geq(g((b_i^{\prime})^+)-g(b_i^{\prime}))^++(g(b_i^{\prime})-g(a_i^+))^+\\
&\geq f((b_i^{\prime})^+)-f(b_i^{\prime})+f(b_i^{\prime})-f(a_i^+)\geq f(b_i)-f(a_i^+)-\eps2^{-i}.
\end{split}
\end{equation*}
So sum up, we get
\begin{equation*}
\begin{split}
&\sum_i(\partial_tg)^+((a_i,b_i))\geq\sum_i(f(b_i)-f(a_i^+))-\eps\geq\partial_tf(E)-\eps.
\end{split}
\end{equation*}
So the proof is complete.
\end{proof}

\subsection*{Acknowledgements}

The algorithm in the paper was developed at an Environmental Modelling in Industry - Study Group, 21th-24th September 2015, Isaac Newton Institute, Cambridge. It was sponsored by the EPSRC Living with Environmental Change - Maths Foresees Network and the NERC Probability, Uncertainty and Risk in the Environment (PURE) Network 

The work of Jingrui Cheng was supported in part by the National Science Foundation under Grant DMS-1401490. Jingrui also wishes to thank his advisor Mikhail Feldman for helpful discussions and suggestions.

\end{document}